\documentclass[11pt]{amsart}
\usepackage{amsmath,amscd,amssymb,amsfonts,amsthm,hyperref}
\usepackage{xcolor}  
\hypersetup{   colorlinks,   linkcolor={red!50!black},    citecolor={blue!70!black},   urlcolor={blue!80!black}}

\usepackage{tcolorbox}
\usepackage[cmtip,matrix,arrow]{xy}
\usepackage{tikz}
\usepackage{tikz-cd}

\newtheorem{theorem}{Theorem}[section]
\newtheorem{corollary}[theorem]{Corollary}
\newtheorem{proposition}[theorem]{Proposition}

\newtheorem{lemma}[theorem]{Lemma}
\theoremstyle{definition}    
\newtheorem{definition}[theorem]{Definition}
\theoremstyle{remark}

\newtheorem{remark}[theorem]{Remark}

\newtheorem{example}[theorem]{Example}

\newtheorem{exercise}{Exercise}[section]
\newcommand\A{\mathcal{A}}

\newcommand{\EE}{{\mathbb{E}}}
\newcommand{\Cour}[1]      {[\![#1]\!]}
\newcommand\M{\mathcal{M}}

\newcommand{\V}{\mathcal{V}}
\renewcommand{\L}{\mathcal{L}}
\renewcommand{\O}{\mathcal{O}}
\newcommand{\T}{\mathbb{T}}

\newcommand{\ca}{\mathcal}

\newcommand{\E}{\ca{E}}

\newcommand{\R}{\mathbb{R}}
\newcommand{\C}{\mathbb{C}}
\newcommand{\SU}{\on{SU}}

\newcommand{\Z}{\mathbb{Z}}

\newcommand\pt{\on{pt}}

\newcommand{\ez}{\mathsf{e}}
\newcommand{\fz}{\mathsf{f}}
\newcommand{\az}{\mathsf{a}}
\newcommand{\bz}{\mathsf{b}}
\newcommand{\cz}{\mathsf{c}}
\newcommand{\vz}{\mathsf{v}}
\newcommand\lie[1]{\mathfrak{#1}}
\renewcommand{\k}{\lie{k}}
\newcommand{\h}{\lie{h}}
\newcommand{\g}{\lie{g}}

\renewcommand{\a}{\mathsf{a}}
\renewcommand{\t}{\lie{t}}

\newcommand{\on}{\operatorname}

\newcommand{\Ad}{ \on{Ad} }

\newcommand{\Hom}{ \on{Hom}}

\renewcommand{\ker}{ \on{ker}}

\newcommand{\Mult}{  \on{Mult}}

\newcommand{\da}{\dasharrow}

\newcommand\qu{/\kern-.7ex/} 

\newcommand{\lra}{\longrightarrow}
\newcommand{\hra}{\hookrightarrow}

\renewcommand{\d}{{\mbox{d}}}
\newcommand{\ol}{\overline}

\newcommand\Phinv{\Phi^{-1}}

\newcommand\eps{\epsilon}

\newcommand{\f}{\frac}

\newcommand{\p}{\partial}
\renewcommand{\l}{\langle}
\renewcommand{\r}{\rangle}
\newcommand\hh{{\f{1}{2}}}

\newcommand{\ti}{\tilde}
\newcommand{\eeq}{\end{eqnarray*}}
\newcommand{\beq}{\begin{eqnarray*}}

\newcommand{\pr}{\on{pr}}
\newcommand{\wh}{\widehat}
\newcommand{\wt}{\widetilde}
\newcommand{\mf}{\mathfrak}
\newcommand{\rra}{\rightrightarrows}

\newcommand{\ul}{\underline}

\renewcommand{\subset}{\subseteq}

\renewcommand{\S}{{\mathcal{S}}}

\newcommand{\sz}{\mathsf{s}}
\newcommand{\gz}{\mathsf{g}}
\newcommand{\tz}{\mathsf{t}}

\begin{document}
\title{Introduction to moduli spaces and Dirac geometry}
\author{Eckhard Meinrenken}

\begin{abstract}	
Let $G$ be a Lie group, with an invariant metric on its Lie algebra $\g$. 	
Given a surface $\Sigma$ with boundary, and a collection of base points $\V\subset \Sigma$ meeting every boundary component, 
the moduli space (representation variety) $\M_G(\Sigma,\V)$ carries a distinguished 
`quasi-symplectic' 2-form.	We shall explain the 
finite-dimensional construction of this 2-form and discuss its basic properties, using quasi-Hamiltonian techniques and Dirac geometry. 
This article is an extended version of lectures given at the  summer school  'Poisson 2024' at the Accademia Pontaniana in Napoli, July 2024. 
\end{abstract}
\maketitle

\tableofcontents

\section{Introduction}

This article is an extended version of  lectures given at the  summer school  'Poisson 2024' at the Accademia Pontaniana in Napoli, July 2024.
The moduli spaces in its title are the moduli spaces $\M_G(\Sigma,\V)$ of flat $G$-bundles over compact oriented surfaces $\Sigma$, with framings at a finite collection $\V\subset \Sigma$ of base points, meeting every component of the boundary. 
Equivalently, they are described as spaces of homomorphims from the fundamental groupoid into the Lie group $G$. If the Lie algebra of $\g$ carries an invariant metric, then the moduli space acquires a distinguished 2-form $\omega$. For the case without boundary (and $\V=\emptyset$), this is the Atiyah-Bott symplectic form. When $\Sigma$ has non-empty boundary,  then this 2-form is neither closed or nondegenerate, but for $\V\subset \p\Sigma$ it is \emph{quasi-symplectic} in the sense of quasi-Hamiltonian geometry. 

My original plan for the lectures had been to divide the material into two parts. The first part would present a  direct `low-tech' construction of the 2-form via cutting and gluing of the surface.  Most properties of the 2-form are direct consequences from the construction, 
with the exception of `minimal degeneracy' of the 2-form. This would then serve as a motivation for the second part, leading to an introduction to Dirac geometry as the proper framework for quasi-Hamiltonian spaces, with moduli spaces serving as the main examples.  
In reality, this plan was too ambitious, and the lectures never actually reached the second part. Accordingly, Sections  \ref{sec:overview} -- \ref{sec:cutglue} of the present article  are  written in the style of lecture notes, replete with exercises. In later sections the presentation  gradually morphs into a survey style, discussing material that was only hinted at in the lectures. 

In more detail, the material is organized as follows. Section \ref{sec:overview} begins with a rapid overview of moduli spaces of flat bundles over surfaces, and associated representation varieties. Section \ref{sec:basics} defines the moduli space $\M_G(\Sigma,\V)$ for any finite collection $\V\subset \Sigma$ (meeting every component)  as the space of homomorphisms  from the fundamental groupoid 
$\Pi(\Sigma,\V)$ into $G$. The space carries an action of $G^\V$, and taking holonomies along boundary edges $\ez\in \E$ defines an equivariant map 
$\Phi\colon \M_G(\Sigma,\V)\to G^\E$. Starting in Section 
\ref{sec:2form}, we assume that $\g=\on{Lie}(G)$ carries an invariant metric. This determines a \emph{Cartan 3-form}  $\eta\in \Omega^3(G)$ and a related 2-form $\beta\in \Omega^2(G\times G)$. The main result of this section is the existence of a distinguished 2-form 
$\omega$ on the space $\M_G(\Sigma,\V)$, whose differential is the sum of pullbacks of $\eta$ under the boundary holonomies, and whose contractions with the generators of the $G^\V$-action are described explicitly.  The 2-form admits a direct description for the case that $\Sigma$ is a polyhedral region (an $n$-gon); the general case is reduced to this case via gluing diagrams. Using cutting and re-gluing, we prove that the 2-form does not depend on the choice of gluing diagram. Section \ref{sec:cutglue} continues the discussion of cutting and gluing operations, examining in particular the gluing of two boundary circles of a surface. Among other things, this leads to  a description of the 2-form for surfaces without boundary. Section \ref{sec:goldman} is devoted to Hamiltonian dynamics on the moduli spaces: 
Every $G^\V$-invariant function $f$ on the moduli space determines a Hamiltonian vector field $X_f$ (despite the fact that 
$\omega$ is degenerate). For special choices of $f$, this leads to the so-called \emph{Goldman flows} which we compute explicitly. 
Section \ref{sec:quasisymplecticgroupoid} shows that the moduli space for a cylinder is naturally a `quasi-symplectic groupoid', where the  groupoid multiplication is defined via gluing of cylinders. Finally, Section \ref{sec:dirac} gives an introduction to Dirac geometry, and supplies proofs of minimal degeneracy properties of the 2-form via a cross-section theorem for Dirac structures. 

It is a pleasure to thank the organizers of the Poisson 2024 summer school and the Accademia Pontaniana 
for the invitation to present these lectures in a wonderful setting, as well as INdAM for sponsoring the event.  Special thanks to Luca Vitagliano for 
help during the meeting and Alfonso Tortorella, Antonio de Nicola, and Chiara Esposito for editing the 
proceedings.

\section{Representation varieties}\label{sec:overview}

Let $G$ be a Lie group, and $\Sigma$ a compact, connected surface, without boundary.  
\begin{center}
	\includegraphics[width=0.3\textwidth]{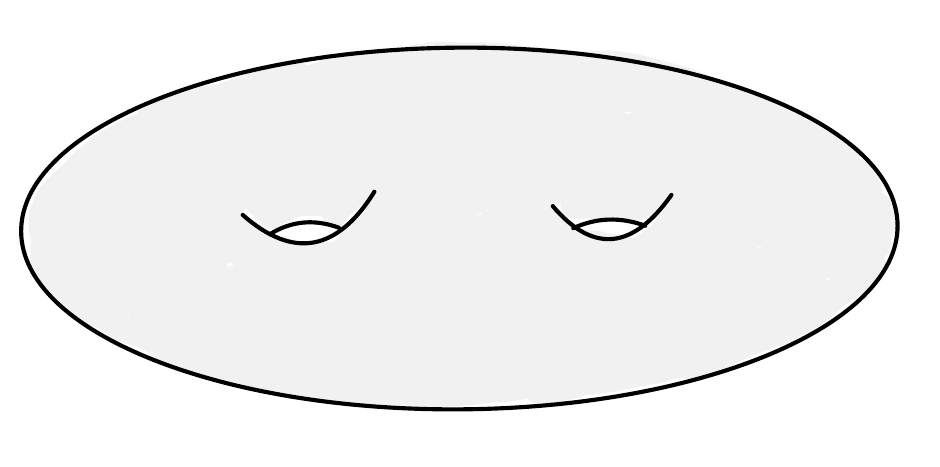}
\end{center}
The \emph{representation variety} or \emph{character variety} is the space 
\[ \M_G(\Sigma)=\Hom(\pi_1(\Sigma,x_0),G)/G\]
of homomorphisms from the fundamental group $\pi_1(\Sigma,x_0)$ into the group $G$, up to conjugacy.  
The definition of fundamental group involves the choice of a base point $x_0\in \Sigma$, but the quotient space no longer depends on this choice.  The space $\M_G(\Sigma)$ may also be regarded as the \emph{moduli space of flat bundles},
\[ \M_G(\Sigma)=\{\mbox{flat principal $G$-bundles}\ \  P\to \Sigma\}\big/\sim\]
where two such bundles are identified if they are related by 
a $G$-equivariant diffeomorphism intertwining the flat connections and inducing the identity on the base. The correspondence is induced by the map taking the homomorphism $\kappa\colon \pi_1(\Sigma,x_0)\to G$ to the associated principal bundle 
\[ P=\wt{\Sigma}\times_{\pi_1(\Sigma,x_0)} G.\] 
In the opposite direction, given a flat $G$-bundle $P\to \Sigma$, one obtains $\kappa$ as the parallel transport (holonomy) 
of the connection, after choice of trivialization $P|_{x_0}\cong G$. Different choices of trivialization at $x_0$ change $\kappa$ by $G$-conjugacy.

\begin{remark}
If the group $G$ is connected and simply connected, then every principal $G$-bundle $P\to \Sigma$ is isomorphic to the trivial one, and $\M_G(\Sigma)$ is the moduli space of flat connections on the trivial bundle.  Equivalently, it is 
the moduli space of Maurer-Cartan elements
\[ \M_G(\Sigma)=\Big\{A\in \Omega^1(\Sigma,\g)|\ \d A+\hh [A,A]=0=0\Big\}/C^\infty(\Sigma,G).\]
where  the group $C^\infty(\Sigma,G)$ acts 
by $g\cdot A=\Ad_g(A)-g^*\theta^R$. Here $\theta^R\in \Omega^1(G,\g)$ is the right-invariant Maurer-Cartan form. 
For general groups $G$, the space $\M_G(\Sigma)$  may also include non-trivial flat bundles. 
\end{remark}

One obtains a concrete description of the moduli space by choosing generators of the fundamental group. 
For a surface of genus $\gz$, there are the \emph{standard generators} 
$\az_1,\bz_1,\ldots,\az_g,\bz_g$ 
with the relation  $\az_1\bz_1 \az_1^{-1} \bz_1^{-1}
\cdots \az_g\bz_g \az_g^{-1} \bz_g^{-1}=1$. 
It corresponds to a gluing diagram, depicted below for genus $\mathsf{g}=2$.  
 \begin{center}
	\includegraphics[width=0.4\textwidth]{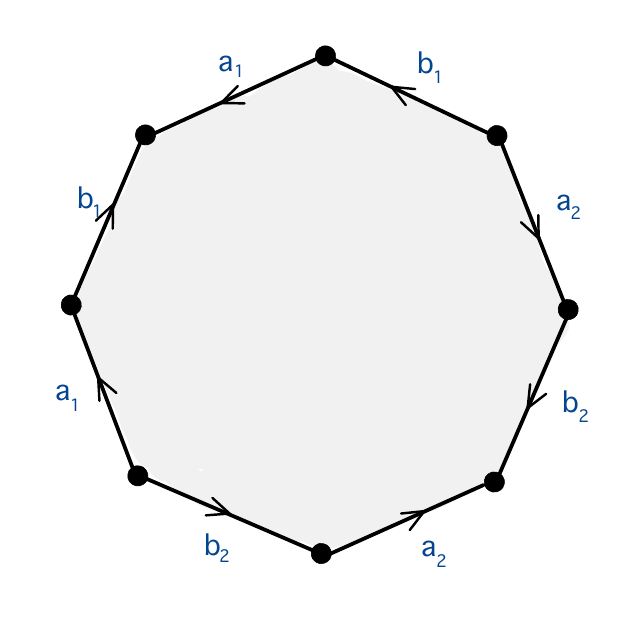}
\end{center}
In terms of these generators, the moduli space is 
identified with 
\[ \M_G(\Sigma)=\Big\{(a_1,b_1,\ldots,a_\gz,b_\gz)\in G^{2\gz}|\ \prod_{i=1}^\gz a_i b_i a_i^{-1} b_i^{-1}=e\Big\}\big/G\]
where we interpret $a_i,b_i\in G$ as the holonomies along $\az_i,\bz_i\in \pi_1(\Sigma,x_0)$. One may also consider other gluing diagrams presenting the surface, corresponding to `non-standard' generators. For example, gluing \emph{opposite} sides of an 8-gon gives a 
surface of genus $\gz=2$; this gluing pattern corresponds to generators $\cz_1,\ldots,\cz_4$ with a relation 
$\cz_1\cdots\cz_4\cz_1^{-1}\cdots \cz_4^{-1}=1$.

The space $\M_G(\Sigma)$ appears in many areas of mathematics and physics -- far too many to properly review them here. 
Let us just point out two major directions:
\begin{itemize}
\item For $G=\SU(n)$, and a given choice of complex structure on the surface, $\M_G(\Sigma)$ it is the moduli space of stable holomorphic vector bundles of rank $n$ and degree $0$ over $\Sigma$: this is the Narasimhan-Seshadri theorem \cite{na:st}.  In one direction, if $\kappa\colon \pi_1(\Sigma)\to \SU(n)$ 
is a homomorphism one obtains a holomorphic vector bundle 
\[ (\wt{\Sigma}\times \C^n)/\pi_1(\Sigma,x_0)\to \Sigma\] 
where $\pi_1(\Sigma,x_0)$ acts on $\C^n$ via $\kappa$. The theorem states that the holomorphic vector bundles obtained in this way are exactly the stable bundles of degree $0$; in particular, every
stable bundle arises in this way. The Narasimhan-Seshadri theorem has been vastly generalized to the \emph{non-abelian Hodge correspondence} for the \emph{Higgs moduli spaces} of Hitchin. 

\item If $G$ is connected but not simply connected, the isomorphism classes of principal $G$-bundles over $\Sigma$ are labeled by $\pi_1(G,e)=\Z$. Consider the case of $G=\on{PSL}(2,\R)$,  and suppose $\Sigma$ has negative 
Euler characteristic $\chi(\Sigma)=2-2\mathsf{g}<0$. By the Milnor-Wood inequality \cite{mil:ex,woo:bun}, the principal $\on{PSL}(2,\R)$-bundle labeled by the integer $k$ admits a flat connection if and only if $|k|\le |\chi(\Sigma)|$. 
Goldman \cite{gol:top} and Hitchin \cite{hit:self} proved that the component of $\M_G(\Sigma)$ labeled by $k$ is connected, and for  $|k|=|\chi(\Sigma)|$ coincides with the \emph{Teichm\"uller space}  of hyperbolic structures up to isotopy. (A sign change of $k$ amounts to a change  of orientation of the surface.)  Hitchin's discovery \cite{hit:lie} of similar contractible components 
(now called \emph{Hitchin components}) for other split semisimple groups $G$ such as $\on{PSL}(n,\R)$, 
has led to the subject of \emph{higher Teichm\"uller theory} \cite{bur:hig,foc:mod}.  In particular, it was shown by Choi and Goldman  
\cite{choi:con} that the Hitchin component for $G=\on{PSL}(3,\R)$ is the moduli space of convex $\R P(2)$-structures on $\Sigma$. 
\end{itemize}

It is a classical fact that a nondegenerate $\Ad_G$-invariant symmetric bilinear form (`metric') 
on the Lie algebra  
\[ \g=\on{Lie}(G)\]
determines a \emph{symplectic structure} on $\M_G(\Sigma)$. This symplectic structure was described by Atiyah-Bott \cite{at:mo} using a gauge theoretic approach (using infinite-dimensional symplectic reduction); 
in special cases it was observed earlier (e.g., by Narasimhan \cite{nar:mod1} for the moduli space of stable bundles, and by Ahlfors \cite{ahl:rem} for Teichm\"uller space).  In the first part of this article
we will explain a finite-dimensional construction, which 
was developed in the 1990s starting with work of Goldman  \cite{gol:sym}, Karshon \cite{kar:alg}, Weinstein \cite{we:symod},
and Guruprasad-Huebschmann-Jeffrey-Weinstein  \cite{gu:gr}. 
The approach below is based on the quasi-Hamiltonian techniques from  Alekseev-Malkin-Meinrenken \cite{al:mom}, 
with additional ideas due to Li-Bland and \v{S}evera \cite{lib:sypo,sev:mod}.  In this context, we mainly deal with moduli spaces associated to surfaces \emph{with non-empty boundary}. The case without boundary is included later, via `reduction'. As applications, we explain a 
computation of Goldman flows along these lines, as well as a construction of quasi-symplectic groupoids integrating the Cartan-Dirac structure (and its generalizations). 
\bigskip

\subsection{Exercises} 

\begin{exercise}
Explain in more detail why, for a connected surface $\Sigma$ and any two choices of base points $x_0,x_0'$, there is a \emph{canonical} isomorphism 
\[ \Hom(\pi_1(\Sigma,x_0),G)/G
\to \Hom(\pi_1(\Sigma,x_0'),G)/G.\] 
(Canonical in the sense that the isomorphism does not depend on additional choices.)  
\end{exercise}

\begin{exercise}
What surface is obtained by gluing the opposite signs of a $2n$-gon? (There will be two cases, depending on whether $n$ is even or odd.) 
\end{exercise}

\begin{exercise}[Representation varieties as moduli of flat bundles]
Let $\Sigma$ be a compact, connected, oriented surface without boundary, and $x_0\in \Sigma$ a base point. 
 Explain in detail why
 \[ \Hom(\pi_1(\Sigma,x_0),G)\] 
 is the space of  flat principal $G$-bundles $P\to \Sigma$ with a given trivialization (\emph{framing}) at the base point, $P|_{x_0}\cong G$, 
 up to  principal bundle isomorphisms intertwining the flat connections and the framings. 
\end{exercise}

\begin{exercise}\label{ex:cut}
		Let $\Sigma=\Sigma_g^0$ be compact, connected, oriented, without boundary, and  
	$D\subset \Sigma$ be an embedded closed disk. 
\begin{enumerate}
	\item Show that the surface with boundary $\wh{\Sigma}=\Sigma-\on{int}(D)$ retracts onto a wedge of $2g$ circles. (E.g.,  use the standard description of $\Sigma$ as being obtained from a $4g$-gon by boundary identifications.)  
	\item Use this to show that for $G$ connected, every principal $G$-bundle over $\Sigma-\on{int}(D)$ is trivial. 
	\item Suppose $G$ is connected. A general principal $G$-bundle $P\to \Sigma$ is obtained by gluing trivial bundles over 
	$\Sigma-\on{int}(D)$ and over $D$ by a clutching function, $\p D\to G$.   Show that this identifies the space of isomorphism classes of principal $G$-bundles with $\pi_1(G,e)$. 
\end{enumerate}	
\end{exercise}

\begin{exercise}
The previous exercise defines a 	map 
	\[ \M_G(\Sigma)\to \pi_1(G,e)\]
	taking isomorphism classes of flat $G$-bundles to isomorphism classes of $G$-bundles (forgetting the flat connection).  
This map can also be described directly in terms of homomorphisms of the 
fundamental group. 
 Let $\wh{\Sigma}=\Sigma-\on{int}(D)$ as above. Choose a base point $x_0$ on the boundary $\partial D
 \cong \partial \wh{\Sigma}$; this also serves as a base point for $\wh{\Sigma}$.  
	\begin{enumerate}
		\item Show that the fundamental group $\pi_1(\wh{\Sigma},x_0)$ is free on $2g$ generators,
		and fits into an exact sequence 
		\[ 1\to \pi_1(\partial\wh{\Sigma},x_0)\to 	\pi_1(\wh{\Sigma},x_0)\to \pi_1(\Sigma,x_0)\to 1\]
		where the subgroup $\pi_1(\partial\wh{\Sigma},x_0)\cong \Z$ is  generated by 
		the class of the oriented boundary loop of $\wh{\Sigma}$. 
		\item Suppose $G$ is connected, and let $\wt{G}$ 
		be its universal covering group. 
		Show that every homomorphism $\pi_1(\wh{\Sigma},x_0)\to G$ admits a lift to a homomorphism 
		$\pi_1(\wh{\Sigma},x_0)\to \wt{G}$. Show that its evaluation on the boundary loop is an element of 
		$\pi_1(G,e)\subset \wt{G}$, which furthermore is independent of the choice of lift. 
		\item Putting all this together, define a map 
		$\M_G(\Sigma)\to \pi_1(G,e)$. 
	\end{enumerate}
\end{exercise}
\bigskip

\section{The space $\M_G(\Sigma,\V)$}\label{sec:basics}
In order to develop tools of `cutting' and `gluing', it is convenient to generalize the setting to surfaces with possibly non-empty boundary and with more than one base point. This necessitates working with fundamental groupoids rather than just the fundamental group.


\subsection{Fundamental groupoids} \label{subsec:general}
Let $X$ be a topological space with finitely many path components, and  $Y\subset X$ a finite subset of base points (`vertices'), meeting all path components. 
We denote by 
\[ \Pi(X,Y)\rra Y\]
the \emph{fundamental groupoid}. The set $\Pi(X,Y)$ of arrows of this groupoid are homotopy classes of 
maps of pairs $\gamma\colon (I,\p I)\to (X,Y)$ relative to $\p I$, where $I=[0,1]$ is the unit interval.  
The source and target maps  are
\[ \sz([\gamma])=\gamma(0),\ \ \tz([\gamma])=\gamma(1),\]
and the groupoid multiplication is given by concatenation of paths $[\gamma']\circ [\gamma]=[\gamma'*\gamma]$,
using the convention 
\[ (\gamma_1*\gamma_2)(t)=\begin{cases}
\gamma_2(2t) & 0\le t\le \hh,\\
\gamma_1(2t-1) & \hh \le t\le 1.
\end{cases}\]
The groupoid inverse is given by pre-composition of paths with  $t\mapsto 1-t$.  The pairs $(X,Y)$ form a category, with morphisms 
$f\colon (X,Y)\to (X',Y')$ the continuous maps $f\colon X\to X'$ such that $f(Y)\subset Y'$; the fundamental groupoid construction is a 
covariant functor from this category into the category of groupoids. Note that $\Pi(f)\colon \Pi(X,Y)\to \Pi(X',Y')$ depends only on the homotopy class of 
$f$ relative to $Y$. 

\medskip

\begin{tcolorbox}
	\begin{definition}
		Given a Lie group $G$, we define the \emph{moduli space} $\M_G(X,Y)$ as the space of groupoid homomorphisms
		\[ \M_G(X,Y)=\Hom(\Pi(X,Y),G).\]
	\end{definition}
\end{tcolorbox}

An element $\kappa\in \M_G(X,Y)$ assigns to every path $\gamma$ with end points in $Y$ a \emph{holonomy}
$\kappa_\gamma\in G$, depending only on $[\gamma]$, in such a way that 
\[ \kappa_{\gamma_1\ast\gamma_2}=\kappa_{\gamma_1}\kappa_{\gamma_2}.\] 
The group $G^Y$ of maps $h\colon Y\to G,\ y\mapsto h_y$ 
acts on $\M_G(X,Y)$ by 
 \begin{equation}\label{eq:holtra}
  (h\cdot \kappa)_\gamma=h_{\gamma(1)}\,\kappa_\gamma h_{\gamma(0)}^{-1}.
 \end{equation}
The construction of $\M_G(X,Y)$ is covariant with respect to $G$ and contravariant  with respect to $(X,Y)$; 
in particular, any morphism of pairs $f\colon (X,Y)\to (X',Y')$ induces a map of moduli spaces, 
\begin{equation}\label{eq:functor} \M_G(f)\colon \M_G(X',Y')\to \M_G(X,Y),\end{equation} 
depending only on the homotopy class of $f$  relative to $Y$. 
The map $\M_G(f)$ is equivariant with respect to the pullback map $G^{Y'}\to G^Y$. 

\begin{example}\label{ex:quotients}
Consider the case that $f$ is the identity map of $X'=X$, but with $Y$ a subset of $Y'$. Then the map 
 \eqref{eq:functor}	is just the quotient map for the action of $G^{Y'-Y}\subset G^{Y'}$:  
\begin{equation}\label{eq:quotient1}
	 \M_G(X,Y)=\M_G(X,Y')/G^{Y'-Y}.\end{equation} 
To see this choose, for each element of $Y'-Y$, a path in $X$ to an  element of $Y$. Then every homomorphism $\kappa'\colon \Pi(X,Y')\to G$ is uniquely determined by its restriction $\kappa\colon \Pi(X,Y)\to G$ together with its values on the chosen paths. That is, the choice of paths identifies
\[ \M_G(X,Y')\cong \M_G(X,Y)\times G^{Y'-Y},\]
with the $G^{Y'-Y}$-action given by right multiplication on the second factor. 
\end{example}

This example indicates how to define $\M_G(X,Y)$ if $Y$ does not meet all components of $X$:  
enlarge to a set $Y'$ that does meet all components, and use \eqref{eq:quotient1}. In particular, 
we define 
\[ \M_G(X)=\M_G(X,\emptyset).\] 

\begin{example}
	Suppose $\Gamma$ is a directed graph (quiver) with vertices $\V_\Gamma$ and oriented edges $\E_\Gamma$. 
	Then  $\Pi(\Gamma,\V_\Gamma)$ is the \emph{free groupoid} over the set of edges, and 
\begin{equation}\label{eq:grap}
	\M_G(\Gamma,\V_\Gamma)=G^{\E_\Gamma}.
\end{equation}	
Here the same notation is used for the abstract graph $\Gamma$ and its geometric realization, taking a quotient of 
$ \sqcup_{\ez\in\E_\Gamma} I$ (one copy for each edge) by the relation determined by the source and target maps $\sz,\tz\colon \E_\Gamma\to \V_\Gamma$.
\end{example}
 As a consequence, whenever  $X$ retracts onto a (geometric) graph $\Gamma$ with $\V=Y$ as its vertices, then 
 $\Pi(X,Y)$ is free, and so $\M_G(X,Y)$ is a manifold diffeomorphic to an $\#\E_\Gamma$-fold 
  product of $G$.

\subsection{Surfaces with boundary}
Let $\Sigma$ be a compact oriented surface  with boundary. 
 Let $\V\subset \Sigma$ a set of vertices.  
 
\begin{tcolorbox}
 \begin{proposition} \label{prop:free}
 Suppose every component of $\Sigma$ has non-empty boundary, and $\V\subset \Sigma$ is a set of vertices 
 meeting each component. Then $\Sigma$ deformation retracts onto a graph $\Gamma\subset \Sigma$, with vertex set $\V_\Gamma=\V$. 
 		The number of edges of the graph is 		
 		\[\#\E_\Gamma= \#\V-\chi(\Sigma),\]
 		where $\chi(\Sigma)$ is the Euler characteristic. 
 	\end{proposition}
 \end{tcolorbox}
 
 \begin{proof} 
 The following pictures give examples of such graphs; we leave the general case as an exercise. (Exercise \ref{ex:2.3}.) 
  	\begin{center}
  			\includegraphics[width=0.22\textwidth]{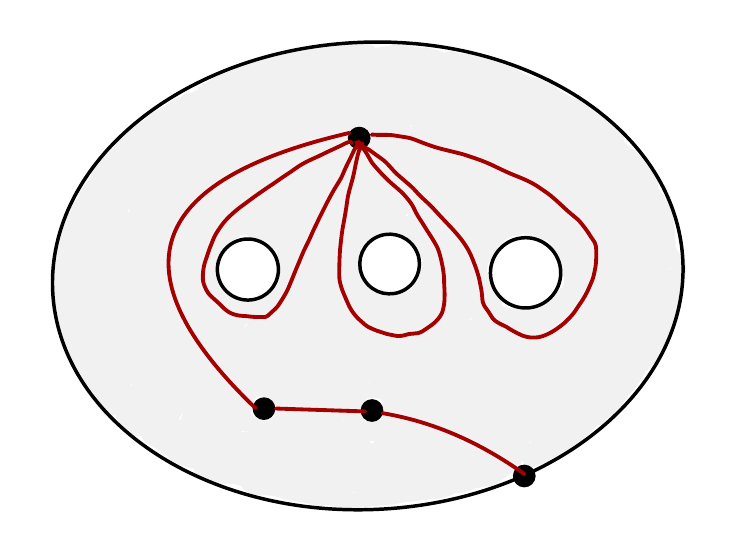}\ \ 
 	\includegraphics[width=0.25\textwidth]{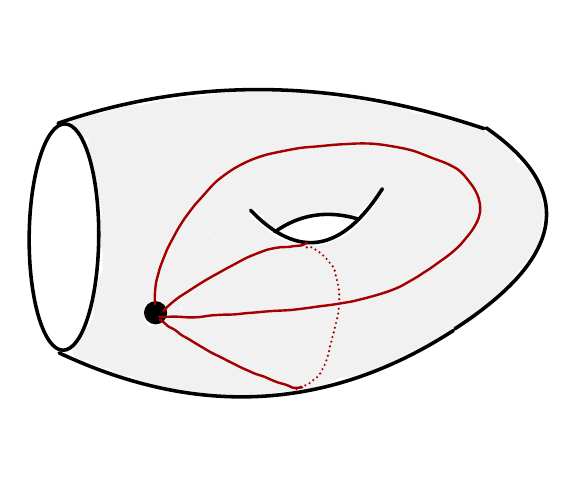}
 \end{center}
 The number of edges of $\Gamma$ is obtained from the Euler characteristic calculation $\chi(\Sigma)=\chi(\Gamma)=\#\V-\#\E_\Gamma$.
 \end{proof}
 
 Hence, under the assumptions of the proposition, the fundamental groupoid $\Pi(\Sigma,\V)$ is free on  
 $\#\V-\chi(\Sigma)$ generators. Of course, there are many choices of generators, in general.

 \begin{example}
 	Consider a 	cylinder (annulus) with three vertices on one boundary component, and one on the other.
 	The following pictures indicate choices of graphs $\Gamma$, each having four edges:
 	\begin{center}
 		\includegraphics[width=0.45\textwidth]{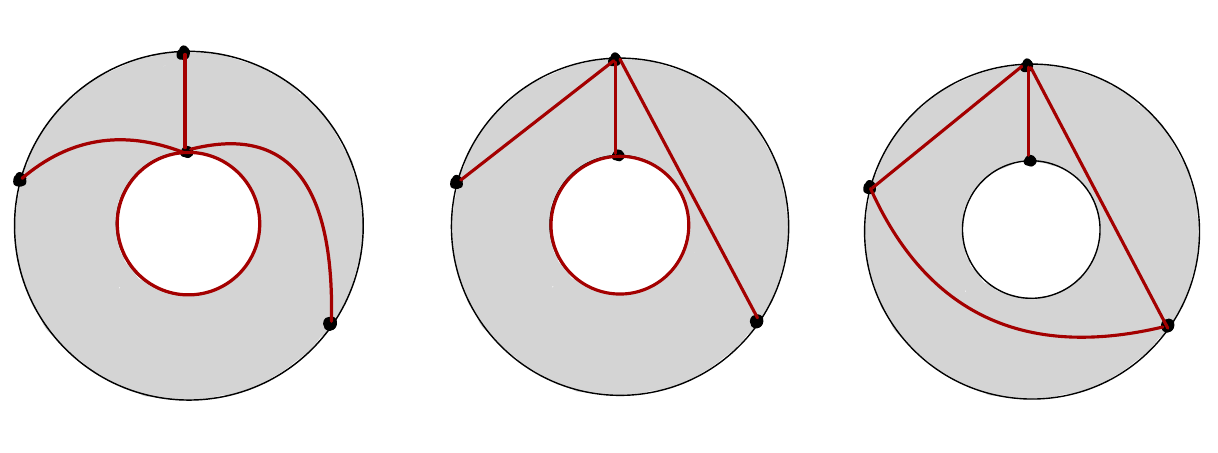}
 	\end{center}

 \end{example}

 Let us now turn our attention to the moduli space 
 	\[ \M_G(\Sigma,\V)=\Hom(\Pi(\Sigma,V),G)\]
with the action of $G^\V$.  By \eqref{eq:grap}, the choice of $\Gamma$, and of an orientation on $\Gamma$, 
gives a $G^\V$-equivariant identification, 
\begin{equation}\label{eq:ident}
\M_G(\Sigma,\V)\cong  G^{\E_\Gamma}\cong \underbrace{G\times \cdots\times G}_{\#\V-\chi(\Sigma)}.\end{equation}
 The resulting manifold structure on $\M_G(\Sigma,\V)$ does not depend on the choice of $\Gamma$. 
 (Exercise \ref{ex:2.3b}.) 
The space $\M_G(\Sigma,\V)$ may be identified with the moduli space of flat principal $G$-bundles $P\to \Sigma$, together with given trivialization (\emph{framing}) over the base points, $P|_\V\cong \V\times G$. 
	The equivalence relation is given by principal bundle isomorphisms intertwining the flat connection and the framing.   The group $G^\V$ acts by change of framing. 	(Exercise \ref{ex:2.1}.)
Any morphism of pairs $f\colon (\Sigma,\V)\to (\Sigma',\V')$ induces a smooth map of moduli space, 
$\M_G(f)\colon \M_G(\Sigma',\V')\to \M_G(\Sigma,\V)$, equivariant with respect to the group morphism 
$G^{\V'}\to G^\V$.

\subsection{Boundary holonomies, mapping class group}
From now on, we will usually make the following assumptions:
\begin{enumerate}
	\item[(A1)] \label{it:a1} Every component of $\Sigma$ has non-empty boundary. 
	\item[(A2)] \label{it:a2} The set $\V$ of vertices meets every component of $\p\Sigma$.
	\item[(A3)] \label{it:a3} The set $\V$   is contained in $\p\Sigma$.
\end{enumerate}
\begin{center}
	\includegraphics[width=0.40\textwidth]{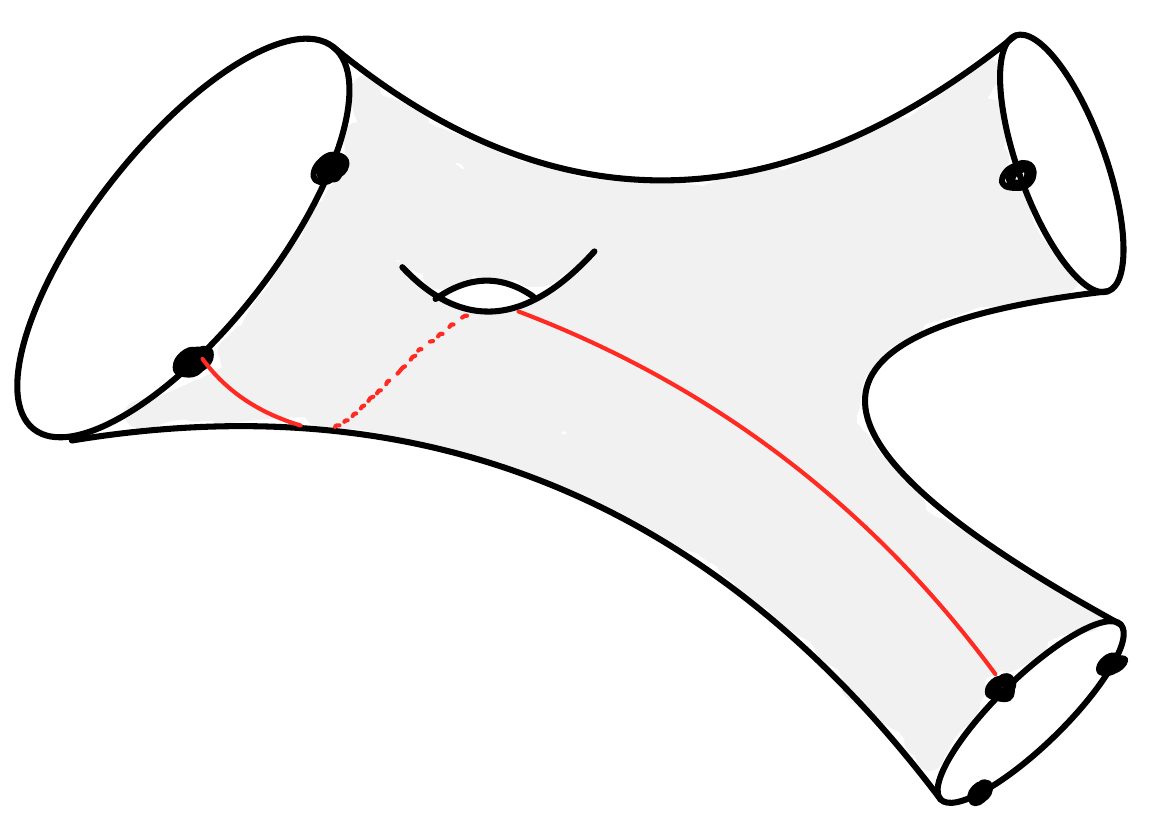}
\end{center}
Sometimes we will omit (A3), also allowing for a non-empty set $\V^{\on{int}}=\V\cap \on{int}(\Sigma)$ 
of `interior vertices'. 
By assumption (A2), the vertex set  divides the boundary into oriented boundary segments. 
Their homotopy classes a subset
\[ \E\subset  \Pi(\p\Sigma,\V\cap \p\Sigma),\]
called the \emph{boundary edges}. 
Hence, $\M_G(\p\Sigma,\V\cap \p\Sigma)\cong G^\E$, and the inclusion 
$(\p\Sigma,\V\cap \p\Sigma)\to (\Sigma,\V)$ determines  a $G^\V$-equivariant map 
\[ \Phi\colon \M_G(\Sigma,\V)\to G^\E, \] 
called the \emph{boundary holonomies}. 
In the theory to be developed below, the map $\Phi$ will play the role of a momentum map (analogous to the momentum maps in symplectic geometry). 

We define the \emph{mapping class group}
\begin{equation}
\label{eq:mappingclassgroup}
 \on{MCG}(\Sigma,\V)
\end{equation}
to be the group of isotopy classes of orientation preserving diffeomorphisms $f\colon \Sigma\to \Sigma$, preserving the 
boundary  $\p\Sigma$ and the set $\V$ of base points. Note that $f$ is not required to preserve the individual boundary components. 
By functoriality, the map of pairs $f\colon (\Sigma,\V)\to (\Sigma,\V)$ 
induces an 
action 
\[ \on{MCG}(\Sigma,\V)\circlearrowright \M_G(\Sigma,\V),\ [f]\mapsto \M_G(f^{-1})\]
on the moduli space. It combines with the $G^\V$-action into an action of the semidirect product 
\begin{equation}\label{eq:semidirect} G^\V\rtimes \on{MCG}(\Sigma,\V),\end{equation}
and the map $\Phi$ is equivariant for this action.

\subsection{Examples} \label{subsec:exa}
We shall denote by $\Sigma_g^r$ the connected, oriented surface of genus $g$ with $r$ boundary components. 

\begin{example}\label{ex:polygon}
	Let $\Sigma=\Sigma_0^1$ be the disk, and suppose 
	$\V$ consists of $n\ge 1$ points on the boundary. 
	\begin{center}
		\includegraphics[width=0.25\textwidth]{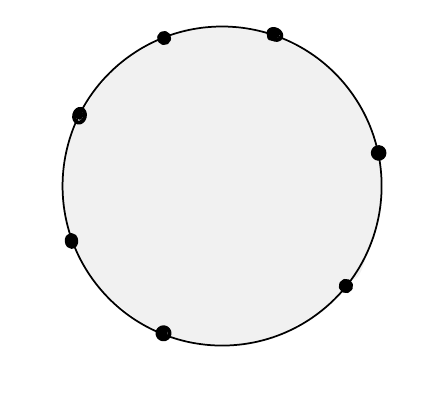}
	\end{center}
	Label the vertices clockwise by 
	$\vz_1,\ldots,\vz_n$, and let $\ez_i$ denote the edge from $\vz_{i+1}$ to  $\vz_{i}$ (with $\vz_{n+1}=\vz_1$). 
	The edges satisfy the relation $\ez_1\cdots \ez_n=1$ in $\Pi(\Sigma,\V)$. Omitting any edge from 
	 $\E=\{\ez_1,\ldots,\ez_n\}$ gives a set of generators of $\Pi(\Sigma,\V)$, and identifies 
	$\M_G(\Sigma,\V)\cong G^{n-1}$.
	The boundary holonomy map $\Phi\colon \M_G(\Sigma,\V)\to G^\E$ is injective, and  
	realizes the moduli space as the submanifold
	\[ \M_G(\Sigma,\V)=\{(a_1,\ldots,a_n)\in G^n|\ \prod a_i=e\},\]	 
	with $G^\V$ acting as 
	\begin{equation}\label{eq:actionngon}
	(h_1,\ldots,h_n).(a_1,\ldots,a_n)=(h_1a_1h_2^{-1},h_2a_2h_3^{-1},\ldots,h_na_nh_1^{-1}).\end{equation}
	The mapping class group $\on{MCG}(\Sigma,\V)=\Z_n$ acts by cyclic permutation on $G^n$.
\end{example}
\begin{example}
	Let $\Sigma=\Sigma_0^2$ be the cylinder (or annulus or 2-holed sphere), with one vertex on each boundary component. 
	\begin{center}
		\includegraphics[width=0.5\textwidth]{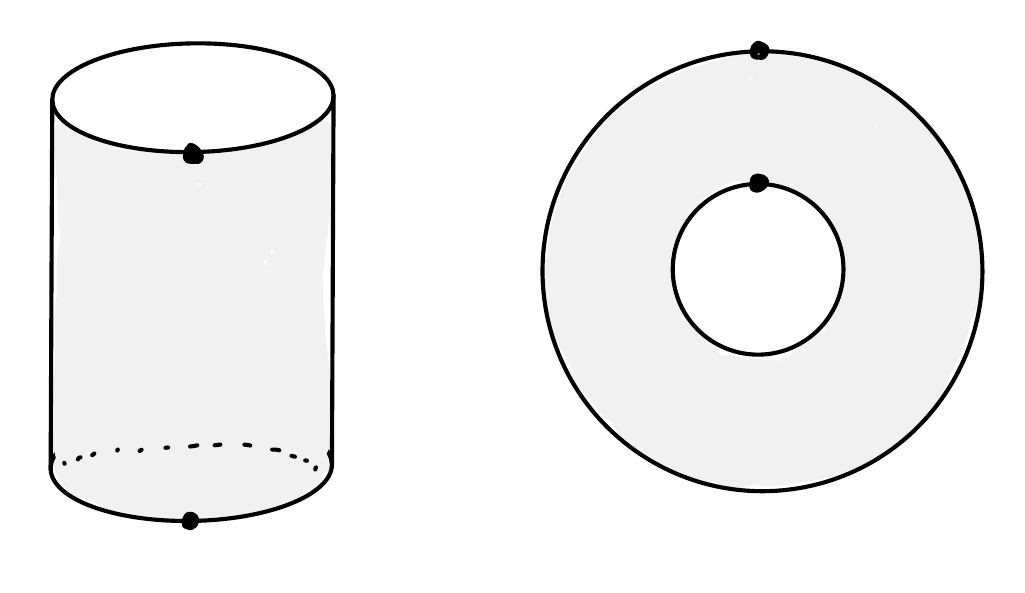}
	\end{center}
	Label the vertices by $\vz_1,\vz_2$. The fundamental groupoid $\Pi(\Sigma,\V)$ has a set of free generators 
	$\az,\bz$, where $\az$ is the class of the oriented boundary loop at $\vz_1$, and $\bz$ the class of a path from $\vz_1$ to $\vz_2$.
	Thus, 
	\[ \M_G(\Sigma,\V)\cong G^2\]
	where $(a,b)\in G^2$ are the holonomies along $\az,\bz$, respectively. 
	The action of $G^\V$ reads as 
	\[ (h_1,h_2).(a,b)=(h_1 a h_1^{-1},\ h_2 b a_1^{-1}).\]
	The holonomy around the second boundary component (based at $v_2$) is $b a^{-1} b^{-1}$, and so 
	\[ \Phi(a,b)=\big(a,b a^{-1} b^{-1}\big).\]
	For the mapping class group one finds 
	\[ \on{MCG}(\Sigma,\V)= \Z\rtimes \Z_2.\] 
	Here the $\Z_2$-generator interchanges the two boundary components; its action on the moduli space is 
	\[ (a,b)\mapsto (b a^{-1} b^{-1},b^{-1}).\]
	The generator of $\Z$ acts by a `Dehn twist' corresponding to a full turn of one of the 
	boundary components while leaving the other one fixed. The action on $\M_G(\V,\Sigma)$ is 
	\[ (a,b)\mapsto (a,ba)\]
	Note that this operation preserves $\Phi$.
\end{example}

\begin{example}
	Let $\Sigma=\Sigma_1^1$ be the 1-holed torus, and $\V$ a single base point on its boundary. 
	\begin{center}
		\includegraphics[width=0.4\textwidth]{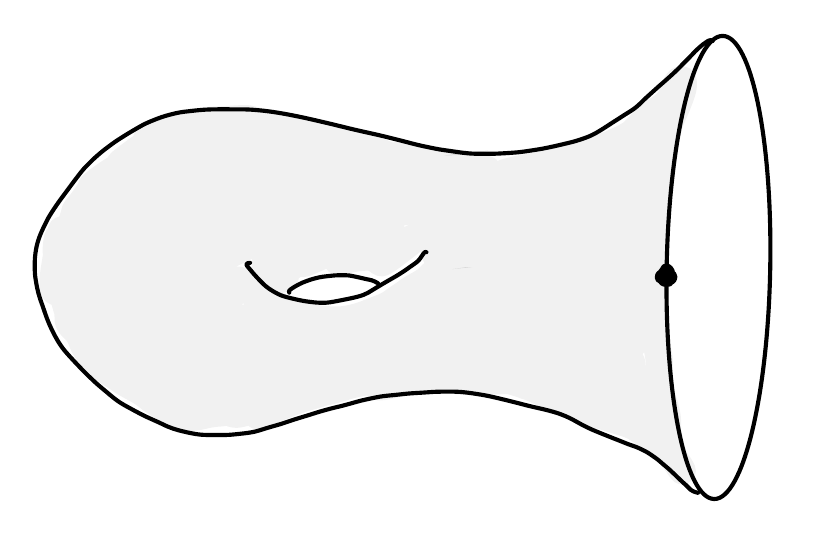}
	\end{center}
	As generators of the fundamental group(oid), we take lifts $\az,\bz\in \Pi(\Sigma,\V)$ of standard generators of the fundamental group of  the 2-torus $T^2=\Sigma_1^0$, chosen in such a way that the boundary loop becomes the group commutator,  $\az\bz\az^{-1}\bz^{-1}$. Let $a,b$ be the holonomies along $\az,\bz$, identifying  
	\[ \M_G(\Sigma,\V)\cong G^2,\]
	with $\Phi(a,b)=aba^{-1}b^{-1}$. The action of $G^\V=G$ is conjugation, $h.(a,b)=(\Ad_h a,\Ad_h b)$. 
	As shown in \cite[Section 3.6.4]{far:pri},  the mapping class group is 
	\[ \on{MCG}(\Sigma,\V)\cong \ca{B}_3,\] 
	the braid group on three strands: thus $\ca{B}_3$ has generators $S,T$ with a single relation 
	\[  STS=TST.\]
	The generators $S,T$ are represented as Dehn twists along the two generators $\az,\bz$; 
	the action on $\M_G(\Sigma,\V)$ is 
	\[ S\colon (a,b)\mapsto (a,ba),\ \ T\colon (a,b)\mapsto (ab^{-1},b).\]
	Note that these transformations preserve $\Phi$. 
	\item  \label{it:e}
	More generally, for $\Sigma=\Sigma_g^1$ with a single base point on its boundary, 
	take $\az_1,\bz_1,\ldots,\az_g,\bz_g$ to be standard generators of the fundamental group. 
	Then $\M_G(\Sigma,\V)\cong G^{2g}$, with 
	\[ \Phi(a_1,b_1\ldots,a_g,b_g)=\prod_{i=1}^g a_i b_i a_i^{-1} b_i^{-1}.\] 
	The group $\on{MCG}(\Sigma,\V)$ gets rather complicated for $g>1$. 
\end{example}

\subsection{Exercises}

\begin{exercise}\label{ex:2.0} 
Given a pair $(X,Y)$ as in Section \ref{subsec:general}, show that the moduli space of 
$(X\times I,Y\times \p I)$ (where $I$ is the unit interval) is naturally a groupoid 
\[ \M_G(X\times I,\,Y\times \p I)\rra \M_G(X,Y).\]
Show that this groupoid is isomorphic to the action groupoid $G^Y\times \M_G(X,Y)\rra \M_G(X,Y)$. 
\end{exercise}

\begin{exercise}[Relation with flat bundles] \label{ex:2.1}
Explain how $\M_G(\V,\Sigma)$ is the moduli space of flat $G$-bundles $P\to \Sigma$ with a trivialization (framing) over $\V$, up to principal bundle isomorphisms intertwining the flat connection and the framing. (Given $\kappa
\in \M_G(\Sigma,\V)$, construct $P$ as a quotient of $\wt{\Sigma}\times G$ where $\wt{\Sigma}$ is the space of  homotopy classes of paths with initial point in $\V$.) 
\end{exercise}

\begin{exercise}\label{ex:2.2}
Show that the Euler characteristic of surfaces with boundary  is uniquely characterized by the following three properties:
\begin{enumerate}
	\item[(E1)] If $\Sigma$ is obtained from a surface $\Sigma'$ by cutting along an embedded circle in the interior of $\Sigma'$, then $\chi(\Sigma)=\chi(\Sigma')$. 
	\item[(E2)] For disjoint unions, $\chi(\Sigma_1\sqcup \Sigma_2)=\chi(\Sigma_1)+\chi(\Sigma_2)$.
	\item[(E3)] For a disk $D$, $\chi(D)=1$.
\end{enumerate}
Use these properties to verify $\chi(\Sigma_g^r)=2-2g-r$. 	
\end{exercise}

\begin{exercise}\label{ex:2.3}
Show that a connected surface with non-empty boundary retracts onto a wedge of  $1-\chi(\Sigma)$
circles. (We suggest using gluing diagrams.) Use this to prove Proposition \ref{prop:free}. 
\end{exercise}

\begin{exercise}\label{ex:2.3b}
Let $G$ be a Lie group, and  $\Pi\rra \V$ a finitely generated free groupoid. 
Any choice $S\subset \Pi$ of generators defines an inclusion $\Hom(\Pi,G)\hra G^S$, which is a 
bijection when $S$ is a set of free generators. Show that the resulting manifold structure 
on $\Hom(\Pi,G)$ does not depend on the choice of $S$. (One possibility, given two sets of generators, is to consider their union.) 
\end{exercise}

\begin{exercise} Consider the 1-punctured torus, as in  Example \ref{subsec:exa}\eqref{it:c}. Show that the Dehn twist around the boundary is expressed in terms of generators of the braid group as $(STS)^{-4}$. Verify that its action on 
the moduli space is conjugation by 	$\Phi(a,b)$. 
\end{exercise}

\begin{exercise}
For every boundary component $C$ of $\Sigma$, the group $\on{MCG}(\Sigma,\V)$ contains the `partial Dehn twists' represented by a diffeomorphism which is trivial outside a collar neighborhood of $C$ and acts on $C\cap \V$ by 
cyclic permutation. Describe its action on $\Pi(\Sigma,\V)$. Use the result to describe the action on 
$\M_G(\Sigma,\V)$ in terms of the components $\Phi_\ez$, for $\ez$ an edge in $C$. (The answer will depend on the size and direction of the twist.)
\end{exercise}\bigskip

\section{Two-forms on $\M_G(\Sigma,\V)$}\label{sec:2form}
Throughout this section, we consider pairs $(\Sigma,\V)$ of a surface with a collection  of vertices, satisfying (A1), (A2). We will show that a given metric on the Lie algebra of $G$ determines a 2-form on $\M_G(\Sigma,\V)$, with nice properties. If (A3) is satisfied as well, this 2-form is `minimally degenerate'.

\subsection{The 2-form}
We shall assume from now that the Lie algebra 
\[ \g=\on{Lie}(G)\]
comes equipped with an $\Ad_G$-invariant nondegenerate symmetric bilinear form $\cdot$, referred to as a 
\emph{metric}. For example, if $G$ is semi-simple we may take  the Killing form of $\g$. The metric determines 
a \emph{Cartan 3-form}
\begin{equation}\label{eq:eta}
 \eta=\f{1}{12} \theta^L\cdot [\theta^L,\theta^L]\in \Omega^3(G)
\end{equation}
(where $\theta^L,\theta^R\in \Omega^1(G,\g)$ are the Maurer-Cartan forms). Using that $\theta^R=\Ad_g\theta^L$ 
we see that this 3-form  is invariant under both left and right translations, and consequently (Exercise \ref{ex:3.1}) is closed: 
\[ \d\eta=0.\]
For $\xi\in \g^\V=\on{Map}(\V,\g)$, let $\xi_{\M_G(\Sigma,\V)}$ 
denote the generating vector field for  the $G^\V$-action on $\M_G(\Sigma,\V)$. 
The holonomy along the edge $\ez$ transforms according to (cf.~ \eqref{eq:holtra})
\begin{equation}\label{eq:action1} g\mapsto h_{\tz(\ez)}\ g\  h_{\sz(\ez)}^{-1}.\end{equation}
This shows that for each component $\Phi_\ez\colon \M_G(\Sigma,\V)$ of the map $\Phi$, 
\[ \xi_{\M_G(\Sigma,\V)}\sim_{\Phi_\ez}
\xi_{\sz(\ez)}^L-\xi_{\tz(\ez)}^R
\]
for $\xi\in \g^\V$, where the superscripts indicate left/right invariant vector fields, 
and
where the tilde symbol signifies related vector fields. 
%
\begin{tcolorbox}
	\begin{theorem} \label{th:2form} \cite{al:mom,lib:mod}
		Suppose $(\Sigma,\V)$ satisfies conditions (A1),\,(A2). 
		There is a canonically defined 
		2-form 
		\[ \omega\in \Omega^2(\M_G(\Sigma,\V) ),\] 
		invariant under the action of 
		$G^\V$ and of the mapping class group $\on{MCG}(\Sigma,\V)$, 
		with the following properties:\medskip
		\begin{enumerate}
			\item\label{it:a} $ \d\omega=-\sum_{\mathsf{e}\in \ca{E}} \Phi_{\mathsf{e}}^*\eta.$ \smallskip
			\item\label{it:b}
			$\iota(\xi_{\M_G(\Sigma,\V)})\omega=-\hh \sum_{\ez\in \ca{E}} 
			\Phi_{\ez}^*(\theta^R\cdot\xi_{\tz(\ez)}+\theta^L\cdot\xi_{\sz(\ez)})
			$\ \ \ 
			for $\xi\in\g^\V$. \medskip
		\end{enumerate}
	   
	If condition (A3) is satisfied, then this 2-form also has the \emph{minimal degeneracy} property 
	\vskip.08in
	\begin{enumerate}\setcounter{enumi}{2}
		\item \label{it:c} $\ker(\omega)\cap \ker(T\Phi)=\{0\}.$
	\end{enumerate}
	\end{theorem}
\end{tcolorbox}
By \emph{canonical}, we mean that the 2-form does not depend on additional choices; it is naturally associated with the 
pair $(\Sigma,\V)$.  The proof of Theorem \ref{th:2form} will take up the remainder of this section. We start out by 
providing some basic properties of the 3-form $\eta$ (Subsection \ref{subsec:eta}). We then give an explicit description 
of the 2-form $\omega$  for the special case of an $n$-gon  (Subsection \ref{subsec:polygon}). General surfaces are 
reduced to the case of $n$-gons via  gluing patterns (Subsection \ref{subsec:gluingpattern}).

\begin{tcolorbox}
\begin{proposition}\label{prop:modulirank}
If conditions (A1),(A2),(A3) are satisfied, then $\ker(\omega)$ has the explicit description 
\begin{equation}\label{eq:explicitkernel} 
	\ker(\omega)|_\kappa=\{\xi_{\M_G(\Sigma,\V)}|_\kappa|\  \forall\ez\in\E\colon
\xi_{\tz(\ez)}+\Ad_{\Phi(\kappa)_{\sz(\ez)}}\xi_{\sz(\ez)}=0
\}.\end{equation}
Furthermore, there is a canonical isomorphism, for all $\kappa\in \M_G(\Sigma,\V)$,  
\begin{equation}\label{eq:rankfact}
\on{ran}(T\Phi|_\kappa)\cong \on{ann}((\g^\V)_\kappa)
\end{equation}
where $(\g^\V)_\kappa$ is the stabilizer algebra for the $\g^\V$-action at $\kappa$. 	
\end{proposition}	
\end{tcolorbox}
\begin{proof}
The inclusion $\supseteq$ in \eqref{eq:explicitkernel} follows from \eqref{it:b}; the opposite inclusion is
best proved using Dirac geometry; see Proposition \ref{prop:kernel} below. The isomorphism 
\eqref{eq:rankfact} is equivalent to the claim that $\on{ann}(\on{ran}(T\Phi|_\kappa))\cong (\g^\V)_\kappa$. In fact, we will see 
that this space of covectors is spanned by all 
$\sum_{\ez\in \ca{E}} g_\ez^*(\theta^R\cdot\xi_{\tz(\ez)}+\theta^L\cdot\xi_{\sz(\ez)})
$ with $\xi\in (\g^\V)_\kappa$. Once again, the proof becomes more transparent  using Dirac geometry, 
where it follows from a general result, Proposition \ref{prop:range} below. 
\end{proof}
Equation \eqref{eq:rankfact} shows, in particular, that the rank of $\Phi$ at any given point $\kappa$ equals 
the codimension of the stabilizer of the $G^\V$-action at that point. As a special case, we recover the following result of Goldman 
\cite[Section 3.7]{gol:sym}:

\begin{tcolorbox}
\begin{corollary}\label{cor:goldman}
Let $G$ be a Lie group whose Lie algebra $\g$ admits an invariant metric. For all $\gz\ge 1$,  the rank of the products-of-commutator-map 
\[ \Phi\colon G^{2\gz}\to G,\ (a_1,b_1,\ldots,a_\gz,b_\gz)\to \prod_{i=1}^{\gz} a_ib_ia_i^{-1}b_i^{-1}\]
at a point $(a_1,b_1\ldots,a_\gz,b_\gz)$ equals the codimension of its stabilizer  of  
 under conjugation. 
	\end{corollary}
\end{tcolorbox}
\begin{proof}
Let $\Sigma=\Sigma_\gz^1$ be the surfaces of genus $\gz$ with one boundary component, and let $\V=\{\vz\}$ 
consist of a single point on $\p\Sigma$. Then $\Phi$ is just the momentum map for $\M_G(\Sigma,\V)$, and 
the result follows from \eqref{eq:rankfact}.  	
\end{proof}

Note that the group unit is never a regular value of the product-of-commutators map; for example, 
at $(e,\ldots,e)\in \Phi^{-1}(e)$ the rank is $0$.

\begin{remark}
If $G$ is non-simply connected, $\Phinv(e)$ can have several connected components, and it may happen that $\Phi$ has maximal rank on some of those components. For example, if $G=\on{PSL}(2,\R)$ the map 
$\Phi$ has maximal rank on all connected components except for the one containing $(e,\ldots,e)$. 
\end{remark}

\subsection{Properties of the  Cartan 3-form}\label{subsec:eta}
We shall consider the Cartan 3-form $\eta$ alongside the following 2-form on $G\times G$, 
\begin{equation}\label{eq:beta}
 \beta=\hh \pr_1^*\theta^L\cdot \pr_2^*\theta^R.\end{equation}
Here $\pr_1,\pr_2$ are the two projections. 

\begin{tcolorbox}
	\begin{proposition}\label{prop:mult}
		\begin{enumerate}
			\item The pullback of $\eta$ under group inversion $\on{Inv}\colon G\to G$ satisfies 
			\[ \on{Inv}^*\eta=-\eta.\]
			\item 	The pullback of $\eta$ under group multiplication $\on{Mult}\colon G\times G\to G$ satisfies 
			\[ \Mult^*\eta=\pr_1^*\eta+\pr_2^*\eta-\d\beta.\]
			
		\end{enumerate}
	\end{proposition}
\end{tcolorbox}

\begin{proof}
This is a straightforward verification, using 
\[\on{Inv}^*\theta^L=-\theta^R,\ \ \ \ \Mult^*\theta^L=\Ad_{g_2}^{-1}\pr_1^*\theta^L+\pr_2^*\theta^L\]
(where we write elements of $G\times G$ as $(g_1,g_2)$) and the Maurer-Cartan structure equation
$\d\theta^L+\hh[\theta^L,\theta^L]=0$. 
\end{proof}

The 3-form $\eta\in \Omega^3(G)$ and 2-form $\beta\in \Omega^2(G\times G)$ appear in the 
context of moduli spaces of $G$-bundles in Chern-Simons theory and the Polyakov-Wiegman formula 
\cite{pol:gol}; their significance for constructing the symplectic structure on moduli spaces of flat $G$-bundles 
was recognized by Weinstein in \cite{we:symod}. These forms may be interpreted in terms of the Bott-Shulman double complex 
\cite{bot:che,je:gr} for the simplicial realization of the classifying space $BG$. One has 
\[ \Omega^{i,j}(BG)=\bigoplus_{i+j=k}\Omega^i(G^j)\]
with total differential $\d+(-1)^i\delta$, where $\d$ is the de Rham differential (raising index $i$) and $\delta$ is the 
simplicial differential (raising index $j$), given on elements of bidegree $(i,j)$ by $\delta=\sum_{i=0}^j (-1)^i \partial_i^*$
with 
\[ \partial_i(g_1,g_2,\ldots,g_j)=\begin{cases}
(g_2,\ldots,g_j)& i=0,\\
(g_1,\ldots,g_ig_{i+1},\ldots,g_j) &0<i<j,\\
(g_1,\ldots,g_{j-1})&i=j.
\end{cases}
\]
The sum $\eta+\beta$ is a cocycle of total degree $4$ for the total complex. 
This amounts to the equations  
\[ \d\eta=0,\ \ \d\beta=\delta\eta,\ \ \delta\beta=0.\]
%
The cohomology class of $(\eta,\beta)$ for the total differential is the 1st Pontrjagin class of the classifying bundle $EG\to BG$.

\begin{remark}
	In shifted geometry parlance, $\eta+\beta$ is a 2-shifted 2-form on the stack $G\rra \pt$. For discussion, see e.g. Alvarez
	\cite{alv:tra}. 
\end{remark}

Using Proposition \ref{prop:mult}, one may check the following result, which will be useful in our construction of 
the 2-form on the moduli space. 

\begin{tcolorbox}
	\begin{proposition}[\v{S}evera \cite{sev:mod}] \label{prop:severa}
		Given any manifold $M$, the space 
		\[ C^\infty(M,G)\times \Omega^2(M)\] has a group structure with product 
		\[ (\Phi_1,\omega_1)\bullet (\Phi_2,\omega_2)=
		\left(\Phi_1\Phi_2,\omega_1+\omega_2-(\Phi_1,\Phi_2)^*\beta\right)\]
		and inverse  $(\Phi,\omega)^{-1}=(\Phi^{-1},-\omega)$.  
		The map 
		\[ C^\infty(M,G)\times \Omega^2(M)\to \Omega^3(M),\ (\Phi,\omega)\mapsto \d\omega-\Phi^*\eta\] 
		is a group homomorphism for this group structure. 
	\end{proposition}
\end{tcolorbox}
We leave the proof as an exercise (see 
	Exercise \ref{ex:3.2}). We remark that the associativity of the product involves the property $\delta\beta=0$.\medskip

For $\xi\in\g$, let $\xi^L$ denote the left-invariant vector field generated by $\xi$, and $\xi^R$ the 
right-invariant vector field. The generating vector fields for the $G\times G$-action on $G$, 
given by $(g',g).a=g' a g^{-1}$, are 
\begin{equation}\label{eq:gtimesg}
  (\xi',\xi)_G=\xi^L-(\xi')^R.
\end{equation}
As already mentioned, the 3-form $\eta$ is bi-invariant, hence its Lie derivative with respect to the vector field
\eqref{eq:gtimesg} vanishes. Since $\eta$ is closed, it follows that the contraction  with this vector field is closed. 
In fact, it is exact: 

\begin{tcolorbox}
	\begin{proposition} 
		The contractions of the 3-form $\eta\in \Omega^3(G)$ with the generating vector fields of the $G\times G$-action are given by 
		\[ 	\iota((\xi',\xi)_G)\eta=-\d\Big(\hh  (\theta^R\cdot\xi'+\theta^L\cdot\xi)\Big)\]
	\end{proposition}
\end{tcolorbox}
\begin{proof}
	From $\iota(\xi^L)\theta^L=\xi=\iota(\xi^R)\theta^R$ one obtains  $\iota(\xi^L)\eta=\f{1}{4}\xi\cdot [\theta^L,\theta^L]$, and similarly 
	for $(\xi')^R$.  Now use the Maurer-Cartan equations.
\end{proof}

\subsection{Polygons}\label{subsec:polygon}
Consider the closed disk $\Sigma=\Sigma_0^1$
with $\#\V=n$ vertices. (See Example \ref{ex:polygon}.) 
For the purposes of gluing diagrams, it is convenient to regard it as an $n$-gon.
	\begin{center}
	\includegraphics[width=0.23\textwidth]{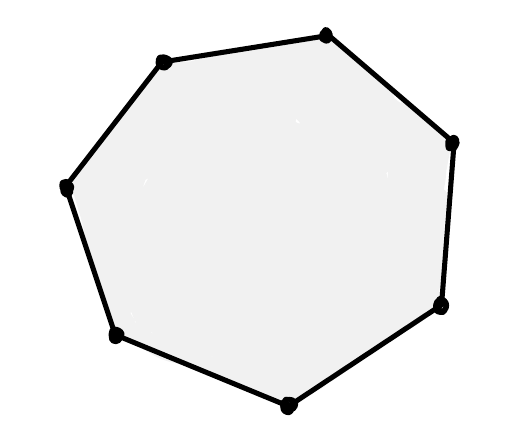}
\end{center}
Let $\V=\{\vz_1,\ldots,\vz_n\}$ be the set of vertices, labeled clockwise. This also determines a clockwise labeling of the 
edges, $\E=\{\ez_1,\ldots,\ez_n\}$, where $\ez_i$ is the oriented edge from $\vz_{i+1}$ to $\vz_{i}$ (using the convention
$\vz_{n+1}=\vz_1$). We denote the components $\Phi_{\ez_i}$ of 
\[ \Phi\colon \M_G(\Sigma,\V)\to G^\E\] 
by $g_i\colon \M_G(\Sigma,\V)\to G$. Clearly, $\Phi=(g_1,\ldots,g_n)$
is an embedding into $G^\E\cong G^n$ as the submanifold defined by $g_1\cdots g_n=e$. 
\begin{tcolorbox}
	\begin{proposition}\label{prop:diskcase}
There is a 	unique 2-form $\omega$ on $\M_G(\Sigma_0^1,\V)$ satisfying the properties of Theorem \ref{th:2form}. It is  described by the formula
\begin{equation}
\label{eq:severa} (e,\omega)=(g_1,0)\bullet \cdots \bullet (g_n,0)\end{equation}
(using \v{S}evera's formalism (Proposition \ref{prop:severa})). 
	\end{proposition}
\end{tcolorbox}
\begin{proof}
For existence, define $\omega$ by Equation \eqref{eq:severa}. Multiplying from the right by $(g_n,0)^{-1}=(g_n^{-1},0)$, and from the left by $(g_n,0)$, we obtain the alternative expression 
\[ (e,\omega)=(g_n,0)\bullet (g_1,0)\bullet \cdots \bullet (g_{n-1},0).\] 
This shows that $\omega$ is invariant under 
cyclic permutations of the vertices. That is, it is invariant under 
the action of $\on{MCG}(\Sigma_0^1,\V)$. We next verify the properties \eqref{it:a}, \eqref{it:b} and \eqref{it:c} from Theorem \ref{th:2form}. For  \eqref{it:a}, apply the group homomorphism $(\Phi,\omega)\mapsto \d\omega-\Phi^*\eta$ (cf.~ Proposition \ref{prop:severa}) to both sides of \eqref{eq:severa}. The 
left hand side gives $\d\omega$, the right hand side gives $-\sum_{i=1}^n g_i^*\eta$.  For  \eqref{it:b}, using  
cyclic symmetry, it suffices to check for any one of the $G$-factors in $G^\V=G^n$. 
Let us consider  the action of the second factor:  
\[ h\cdot (g_1,\ldots,g_n)=(g_1h^{-1},hg_2,g_3,\ldots,g_n).\]
Its generating vector field is $(\xi^L,-\xi^R,0,\ldots,0)$. 
By formula \eqref{eq:severa}, we have 
\[ \omega=-(g_1,g_2)^*\beta+\omega_1\]
where  $(e,\omega_1)=(g_1g_2,0)\bullet (g_2,0)\bullet \cdots \bullet (g_n,0)$. Since the product $g_1g_2$ and the 
maps $g_3,\ldots,g_n$ are  all \emph{invariant} under the action, the contraction of 
$\omega_1$ with the generating vector fields is zero.  On the other hand, 
\begin{align*}
\iota(\xi^L,-\xi^R,0,\ldots,0) \big((g_1,g_2)^*\beta\big)&=
\hh \xi\cdot g_2^*\theta^R+\hh \xi\cdot g_1^*\theta^L
\end{align*}
This proves \eqref{it:b}. Property \eqref{it:c} is automatic since $\Phi$ is an embedding, so $\ker(T\Phi)=0$.  

It remains to show that $\omega$ satisfying the properties of Theorem \ref{th:2form} is unique. But this just follows since 
the action of $G^\V$ on $\M_G(\Sigma_0^1,\V)$, given by \eqref{eq:actionngon}, is \emph{transitive}.  Hence, there is at most one 2-form $\omega$ satisfying Property \eqref{it:b}. 
\end{proof}

\begin{remark}\label{rem:n-1}
The 2-form $\omega$ may also be computed from the formula 
\[ (g_1\cdots g_{n-1},\omega)=(g_1,0)\bullet \cdots \bullet (g_{n-1},0),\]
obtained by right multiplication of Equation \eqref{eq:severa} by $(g_n^{-1},0)$ and using 
$g_n^{-1}=g_1\cdots g_{n-1}$. 	
So, to compute $\omega$ we only need to compute $n-2$ products. 
\end{remark}

\begin{example}
	Some special cases: For $n=1$ there is nothing to see since $\M_G(\Sigma,\V)=\pt$. 
	For $n=2$ (the 2-gon) we obtain $\M_G(\Sigma,\V)=\{(g_1,g_2)|g_1g_2=e\}\cong G$
	with $\omega=0$. For  $n=3$ (the triangle)
	\begin{center}
	\includegraphics[width=0.15\textwidth]{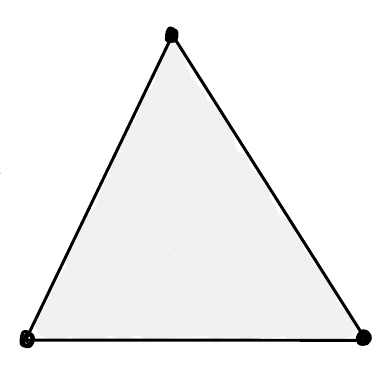}
\end{center}
	we find $\M_G(\Sigma,\V)=\{(g_1,g_2,g_3)|g_1g_2g_3=e\}\cong G^2$, with 
	\begin{equation}\label{eq:3dots}
	 \omega=-\hh g_1^*\theta^L\cdot g_2^*\theta^R,\end{equation}
	or similar expressions obtained by cyclic permutation of $g_1,g_2,g_3$. (One may  verify directly that the expression is invariant under such permutations, using $g_1g_2g_3=e$.) 
\end{example}

\subsection{Gluing patterns}\label{subsec:gluingpattern} 
A convenient reference for presentations of surfaces via gluing patterns is Thurston \cite[Chapter 1.3]{thu:3dim}. 
Let 
\begin{equation}\label{eq:polgons}
\wh{\Sigma}=\wh\Sigma^{(1)}\sqcup \ldots\sqcup \wh\Sigma^{(r)}\end{equation}
be a disjoint union of oriented polygons,  with set of vertices $\wh{\V}=\sqcup_k \wh{\V}^{(k)}$ where 
$\wh{\V}^{(k)}\subset \wh{\Sigma}^{(k)}$.   Let $\wh{\E}=\sqcup_k \wh{\E}^{(k)}$ be the set of boundary edges. A \emph{gluing pattern} is given by a subset 
$\wh{\E}'$ of `free' edges, and a fixed point free involution on the complement $\wh{\E}''=\wh{\E}-\wh{\E}'$ of `paired' edges. 
One obtains a compact oriented surface with boundary 
\[ \Sigma=\wh{\Sigma}/\sim,\]
with set of vertices $\V=\wh{\V}/\sim$, by gluing boundary segments of $\wh{\Sigma}$ as prescribed by the gluing pattern. 
Geometrically, the boundary segments corresponding to paired edges are glued by orientation-reversing diffeomorphisms. Note  that the 
quotient map induces a bijection between the set $\wh{\E}'$ of free edges with the set $\E$ of boundary edges of $(\Sigma,\V)$. 

Gluing patterns are conveniently described by `words' (one word for each $k=1,\ldots,r$), describing the relation arising from the
quotient map. 

\begin{example}
Each of the following gluing patterns describes a 2-torus $\Sigma_1^0$. 
\begin{enumerate}
	\item Identifying opposite sides of a square,  corresponding to the word
	 $ \az \bz \az^{-1} \bz^{-1}$. 
	 	\begin{center}
	 	\includegraphics[width=0.25\textwidth]{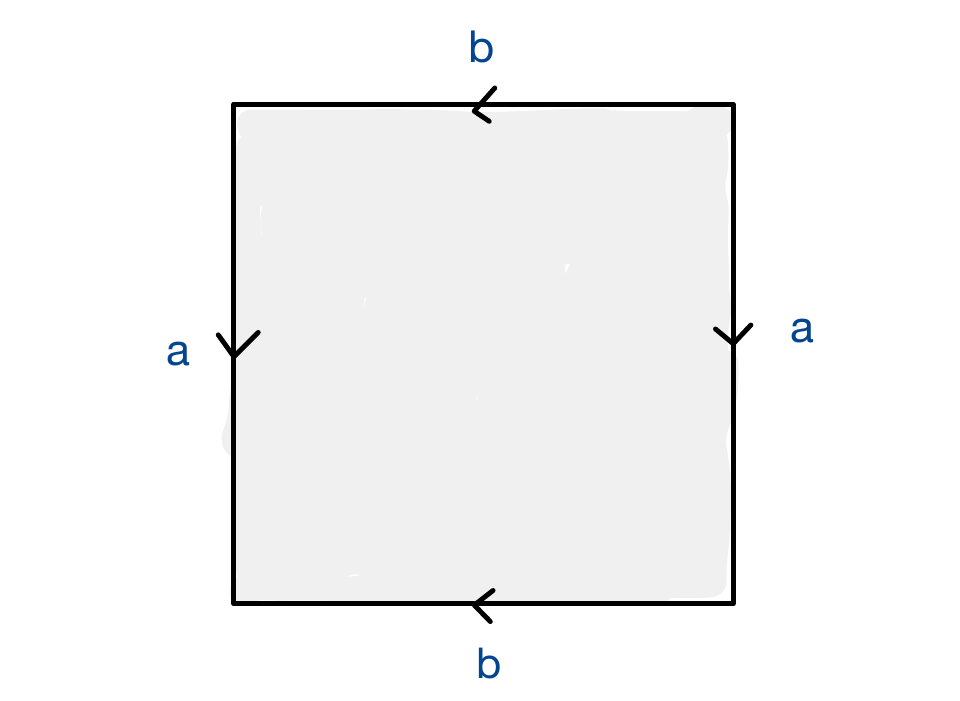}
	 \end{center}
	 \item Gluing two triangles,  corresponding to the two words 
	 $\az\bz\cz,\ \az^{-1}\bz^{-1}\cz^{-1}$.
	 	\begin{center}
	 	\includegraphics[width=0.30\textwidth]{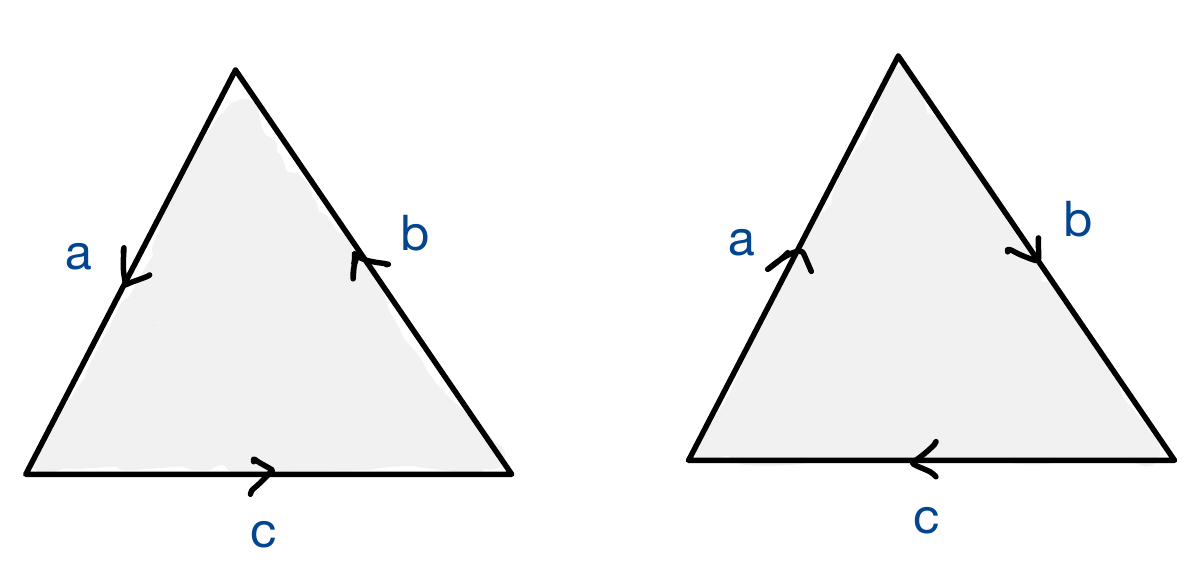}
	 \end{center}
	 \item Gluing opposite sides of a hexagon, with corresponding word
	  $\az\bz\cz\az^{-1}\bz^{-1}\cz^{-1}$.
	 \begin{center}
	 \includegraphics[width=0.25\textwidth]{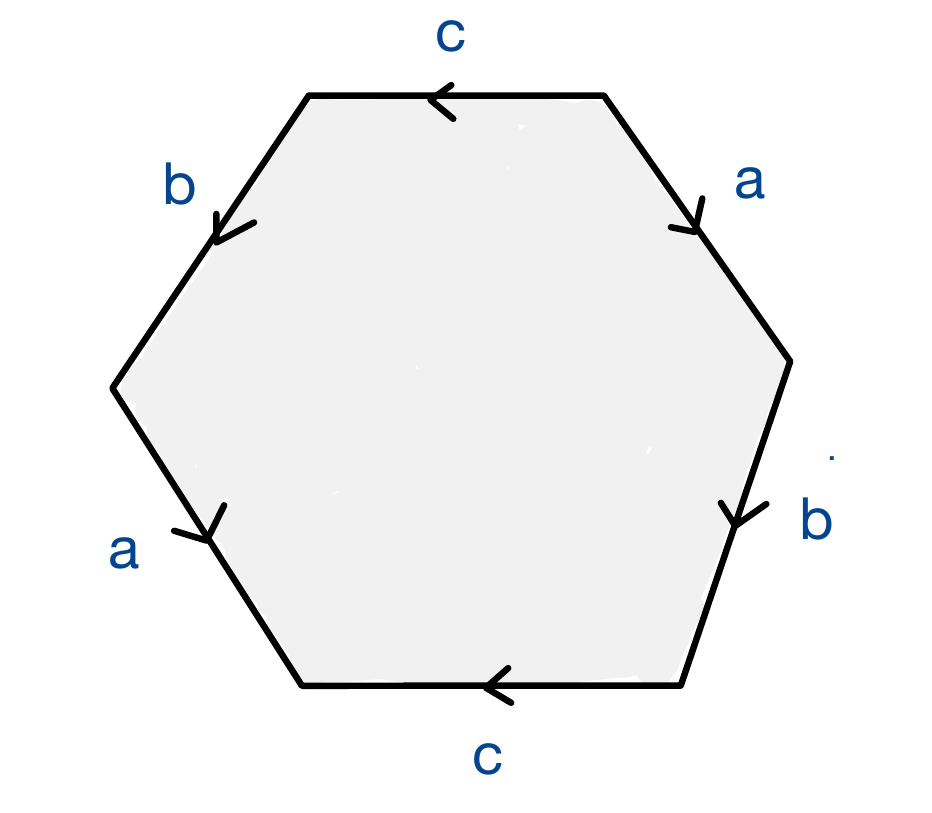}
 \end{center}
\end{enumerate}	
Observe that the set $\V=\wh{\V}/\sim$ of vertices for the glued surfaces $\Sigma$ has cardinality $\#\V=1$ for the first two gluing patterns, $\#\V=2$ for the third gluing pattern.
\end{example}	
\begin{example}
The word	\[ \az_1\bz_1 \az_1^{-1} \bz_1^{-1}
\cdots \az_g\bz_g \az_g^{-1} \bz_g^{-1}\cz_1\cdots\cz_n
\]
describes $(\Sigma,\V)$ where $\Sigma=\Sigma_g^1$ and $\#\V=n$. 
\end{example}

The pair $(\Sigma,\V)$ defined by a gluing pattern will always satisfy condition (A2) (each boundary component of $\Sigma$ meets $\V$). We will impose condition (A1) (that every component of $\Sigma$ has non-empty boundary) by hand; note that this condition is automatic if every component of $\wh\Sigma^{(k)}$ contains at least one free edge. 
Conversely, given  $(\Sigma,\V)$ satisfying (A1),(A2), one may obtain  a gluing pattern $(\wh{\Sigma},\wh{\V})$ through iterated cuts along paths between vertices. 
The quotient map 
\[(\wh{\Sigma},\wh{\V})\to (\Sigma,\V)
\]
induces an embedding  
\begin{equation}\label{eq:embedding}
\iota\colon \M_G(\Sigma,\V)\hra \M_G(\wh{\Sigma},\wh{\V})=\prod_{k=1}^r \M_G(\wh{\Sigma}^{(k)},\wh{\V}^{(k)}).
\end{equation}
%
The space $\M_G(\wh{\Sigma},\wh{\V})$ carries a 2-form $\wh{\omega}=\sum_{k=1}^r \wh{\omega}^{(k)}$
as a sum of 2-forms $\wh{\omega}^{(k)}$ on $\M_G(\wh{\Sigma}^{(k)},\wh{\V}^{(k)})$ 
from Subsection \ref{subsec:polygon}. Let 
\begin{equation}\label{eq:omegaaspullback}
 \omega=\iota^*\wh{\omega}.\end{equation}

\begin{tcolorbox}
	\begin{proposition}\label{prop:2form}
		The 2-form $\omega\in \Omega^2(\M_G(\Sigma,\V))$ given by \eqref{eq:omegaaspullback} does not depend on the choice of gluing pattern, and satisfies the properties of Theorem 
		\ref{th:2form}. 
		
		The interior vertices do not contribute to the 2-form, in the sense that the form on $\M_G(\Sigma,\V)$ is the pullback of the form on $\M_G(\Sigma,\V-\V^{\on{int}})$.		
	\end{proposition} 
\end{tcolorbox}

\begin{proof}
As shown in Section \ref{subsec:polygon}, the 2-form  $\wh{\omega}$ satisfies Property (a) of Theorem 
\ref{th:2form}.  Hence
\[ \d\omega=\iota^*\d\wh{\omega}=-\iota^*\sum_{\ez\in \wh{\E}} \wh{\Phi}_{\ez}^*\eta.\]
For any pair of edges $\ez_i,\ez_j\in \wh{\E}''$ that are identified under the quotient map, we have that 
$\iota^*\wh{\Phi}_{\ez_i}=\iota^*\wh{\Phi}_{\ez_j}^{-1}$. Since 
\[ \on{Inv}^*\eta=-\eta,\] these terms will cancel.
It follows that the sum over edges in $\wh{\E}''$ is zero, leaving only the sum over $\wh{\E}'\cong \E$. 
This gives Property (a) for $\omega$.  Property (b) is verified similarly, using that 
\[ \on{Inv}^*(\theta^R\cdot\xi'+\theta^L\cdot\xi)=- \theta^R\cdot\xi-\theta^L\cdot\xi'.\]

We next show independence of the choice of gluing pattern for $(\Sigma,\V)$. Suppose we are given one such 
gluing pattern, presenting $\Sigma$ as a quotient of a disjoint union \eqref{eq:polgons} 
of polygons. Given paired edges $\ez,\fz\in \wh{\E}''$ from 
distinct polygons, we obtain a new gluing pattern with $r-1$ polygons by gluing 
along those edges. In the opposite direction, we obtain a new gluing pattern with $r+1$ polygons by 
cutting some $\Sigma^{(k)}$ along a diagonal. 
	\begin{center}
	\includegraphics[width=0.5\textwidth]{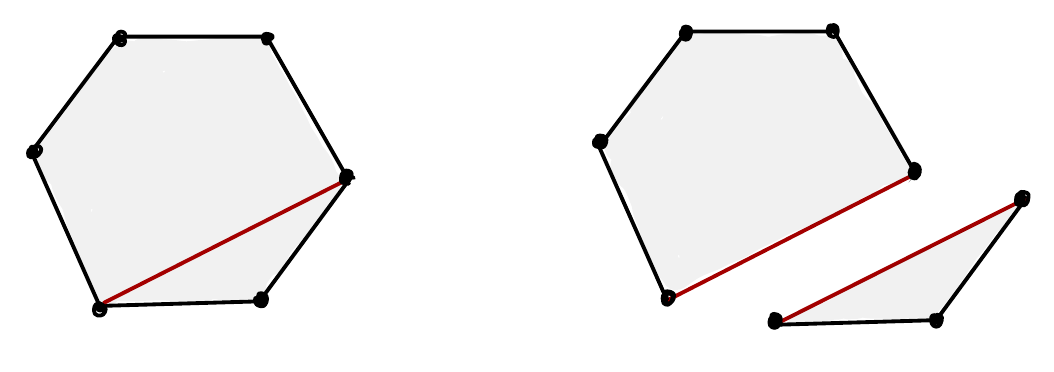}
\end{center}
Any two gluing patterns for $(\Sigma,\V)$ are related by iterated 
cuttings and gluings of this type. (Cf.~ \cite[Problem 1.3.12]{thu:3dim}.) Hence it suffices to see that the 2-form does not 
change under a simple cut of a polygon.

For any given polygon $\wh{\Sigma}^{(k)}\subset \wh{\Sigma}$, the contribution to the 2-form is described by \v{S}evera's formula 
$(e,\wh\omega^{(k)})=(g_1,0)\bullet \cdots \bullet (g_n,0)$ where $g_i$ are the holonomies along the sides of the polygon; 
with $g_1\cdots g_n=e$. 
Cutting $\wh\Sigma^{(k)}$ along a diagonal between non-adjacent vertices $\vz_i,\vz_j$ (thus $i\le j$ with $j\neq i,i+1$) 
produces two new polygons $\wh\Sigma^{(k)}_1,\wh\Sigma^{(k)}_2$, with corresponding 2-forms 
$\wh\omega^{(k)}_1,\  \wh\omega^{(k)}_2$. 
The cutting amounts to introducing a new variable 
\[ c=(g_i\cdots g_{j-1})^{-1}=g_j\cdots g_n g_1\cdots g_{i-1}.
\]
We have 
\begin{align*}
(e,\wh\omega^{(k)})&=
(g_1,0)\bullet \cdots \bullet (g_n,0)\\
&=(g_i,0)\bullet \cdots \bullet (g_n,0)\bullet (g_1,0)\bullet \cdots \bullet  (g_{i-1},0)\\
&=\Big((g_i,0)\bullet \cdots \bullet (g_{j-1},0)\bullet (c,0)\Big)\bullet 
\Big(c^{-1},0)\bullet  (g_{j},0)\cdots \bullet \cdots \bullet (g_{i-1},0)\Big)\\
&=(e,\wh\omega^{(k)}_1)\bullet (e,\wh\omega^{(k)}_2)\\
&=(e,\wh\omega^{(k)}_1+\wh\omega^{(k)}_2).
\end{align*}
Hence $\wh\omega^{(k)}=\wh\omega^{(k)}_1+\wh\omega^{(k)}_2$ as required. \smallskip

From the construction of $\omega$,  it follows that for any orientation preserving diffeomorphism 
$(\Sigma_1,\V_1)\to (\Sigma_2,\V_2)$, the resulting map on moduli spaces intertwines the 2-forms. In particular, the 2-form on $\M_G(\Sigma,\V)$ is invariant under the action of the mapping class group. \smallskip

We next show that interior vertices do not contribute to the 2-form. Conversely, we show that for any given 
$(\Sigma,\V)$, and any new interior vertex $\vz\in \on{int}(\Sigma)$, the 2-form $\omega'$ 
on $\M_G(\Sigma,\V\cup\{\vz\})$ descends to the 2-form $\omega$ on $\M_G(\Sigma,\V)$. Indeed, let $(\wh{\Sigma},\wh{\V})$ 
be a given gluing pattern for $(\Sigma,\V)$. The pre-image of $\vz$ is a single point in the interior of $\wh{\Sigma}$, which we again denote by $\vz$. 
We obtain a gluing pattern $(\wh{\Sigma}',\wh{\V}')$ 
for $(\Sigma,\V\cup\{\vz\})$ by cutting $\wh{\Sigma}$ along a segment from 
a vertex in $\wh{\V}$ to the new vertex $\vz$. Thus, $\wh{\Sigma}'$ has two additional edges $\ez_1,\ez_2$, 
meeting in a vertex $\vz$, which are identified under the map to $\wh{\Sigma}$. 
The formula for $(e,\omega')$ is obtained from that for $(e,\omega)$ by inserting 
$(g,0)\bullet (g^{-1},0)=(e,0)$; where $g$ is the holonomy along $\ez$. This shows that $\omega'$ decends to $\omega$.

It remains to show that if $(\Sigma,\V)$ 
satisfies (A3), then the minimal degeneracy condition 
(Property \eqref{it:c}) is satisfied. This proof uses Dirac geometry and will be presented in Section \ref{subsec:mindeg} below. 
\end{proof}

The last part of the proposition motivates Condition (A3), demanding that all vertices are contained in the boundary. Nevertheless, in some applications, for example for triangulations of a surface, or in the context of reduction (Section \ref{subsec:gluingreduction})
interior vertices arise in a natural way. 

\begin{example}
	The gluing pattern 	
	\begin{center}
		\includegraphics[width=0.2\textwidth]{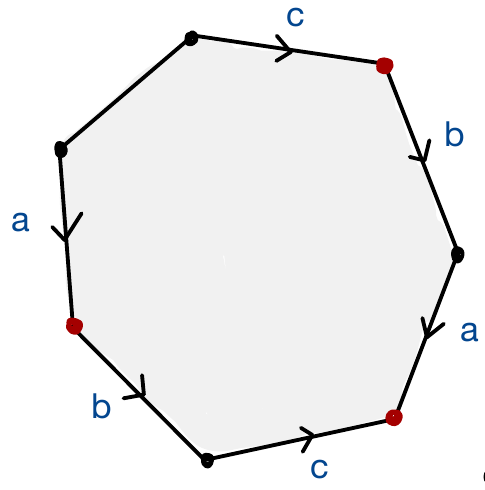}
	\end{center}
	describes a 2-torus with one boundary component. Here $\#\V=2$, with one vertex (indicated as red) in the interior. 
\end{example}


\begin{example}
	The disconnected gluing pattern 	
	\begin{center}
		\includegraphics[width=0.3\textwidth]{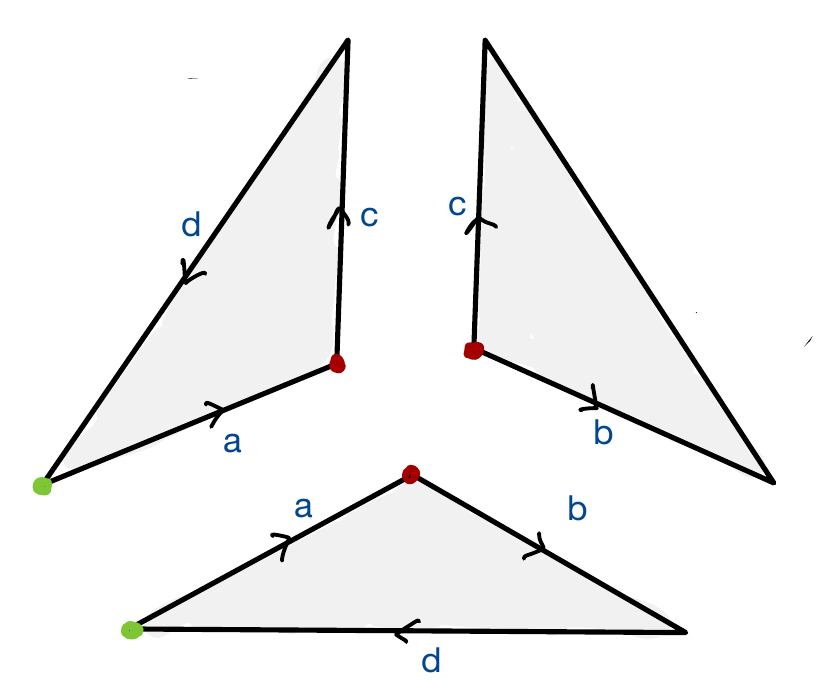}
	\end{center}
	corresponds to a triangulation of a disk. Here $\#\V=3$, with two vertices (indicated as red, green) in the interior of $\Sigma$. 
\end{example}

\subsection{Example}\label{subsec:cylinderexample}
Let us compute the 2-form $\omega$ in the following simple but important example: the moduli space 
$\M_G(\Sigma,\V)$ 
of a cylinder $\Sigma=\Sigma_0^2$, with $\V$ given by a single vertex on each boundary edge. We shall see later that this moduli space provides a basic example of a quasi-symplectic groupoid. 

\begin{center}
	\includegraphics[width=0.5\textwidth]{cylinder.pdf}
\end{center}
As generators for the fundamental groupoid, we take the class of the boundary edge $\az$ at vertex $\vz_1$ 
and the class $\cz$ of a path from $\vz_1$ to $\vz_2$. Let $\az'$ be the class of the boundary edge at $\vz_2$. 
The gluing pattern is thus 
	\begin{center}
	\includegraphics[width=0.3\textwidth]{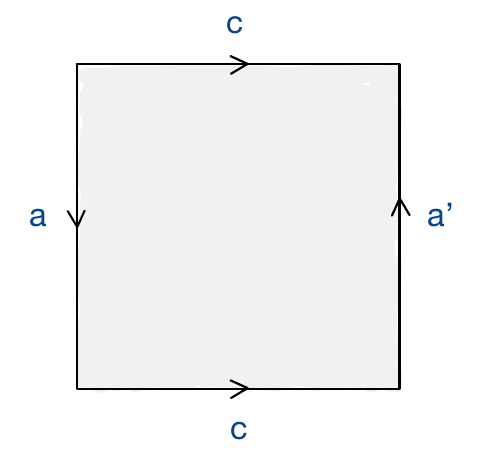}
\end{center}
corresponding to the word $\az\cz^{-1}\az'\cz$. 
Let $a,c\colon \M_G(\Sigma,\V)\to G$ denote the holonomies along $\az,\cz$. The momentum map components are
\[ \Phi_1(a,c)=a,\ \ \Phi_2(a,c)=ca^{-1}c^{-1}.\]
Using \v{S}evera's formula (Proposition \ref{prop:diskcase}), we obtain 
\[ (e,\omega)=(c,0)\bullet 
(a,0)\bullet (c^{-1},0)\bullet (a',0)\]
with $a'=ca^{-1}c^{-1}$. By Remark \ref{rem:n-1}, it suffices to work out the product of the first three terms:
$(c,0)\bullet 
(a,0)\bullet (c^{-1},0)=(cac^{-1},\omega)$. 
The result of this calculation is (Exercise \ref{ex:3.4}):
\begin{tcolorbox}
	\begin{proposition}[Moduli space of cylinder] \label{prop:cylinder}
		The 2-form on the moduli space for the cylinder $\M_G(\Sigma_0^2,\V)$ is given by 
		\[ \omega=-\hh c^*\theta^L\cdot (a^*\theta^L+a^*\theta^R)+\hh c^*\theta^L\cdot \Ad_{a}  (c^*\theta^L) .\]
	\end{proposition}
\end{tcolorbox}
One may verify directly that this 2-form (which was first described in \cite[Section 3.2]{al:mom}) 
 is invariant under the transformation $(a,c)\mapsto (a,ca)$, corresponding to the generator of 
$\on{MCG}(\Sigma,\V)\cong \Z$.

\subsection{Exercises}
\begin{exercise}\label{ex:3.1}
	Prove that a bi-invariant differential form $\alpha\in \Omega^k(G)$ on a Lie group $G$ satisfies \[ \on{Inv}^*\alpha=(-1)^k\alpha\]
	where $\on{Inv}\colon G\to G$ is the group inversion.  
	Use this fact to conclude that every bi-invariant differential form 
	on $G$ is closed. 
\end{exercise}

\begin{exercise}\label{ex:3.6}
	Let $G=\on{PSL}(2,\R)$, and $\wt{G}$ its universal cover. 
	Show that the product-of-commutators map $\Phi\colon G^{2g}\to G$ (cf.~ Corollary \ref{cor:goldman}) 
	lifts to a map $\wt{\Phi}\colon G^{2g}\to \wt{G}$ taking $(e,\ldots,e)$ to the group unit in $\wt{G}$.  
	Show that elements 
	$(a_1,b_1,\ldots,a_\gz,b_{\gz})\in \Phi^{-1}(e)-\wt{\Phi}^{-1}(e)$ have trivial stabilizer under conjugation; hence 
	Proposition \ref{prop:modulirank} shows that 
	$\Phi$ has maximal rank at such points. 
\end{exercise}

\begin{exercise}\label{ex:3.2}
	Prove \v{S}evera's  Proposition \ref{prop:severa}.
\end{exercise}

\begin{exercise}\label{ex:3.3} Verify directly that the expression 
	\eqref{eq:3dots} is unchanged under cyclic permutations of $g_1,g_2,g_3$ with $g_1g_2g_3=e$. 
\end{exercise}

\begin{exercise}\label{ex:3.4} 
	Prove Proposition \ref{prop:cylinder}, by calculating $(c,0)\bullet (a,0)\bullet (c^{-1},0)$. 
\end{exercise}

\begin{exercise}\label{ex:3.5}
	Every $(\Sigma,\V)$ satisfying (A1),(A2) can be cut into a disjoint union of triangles through a finite number of cuts.
	How many cuts are needed? What is the resulting number of triangles?
\end{exercise}


\section{More cutting and gluing }\label{sec:cutglue}
In the previous section, we constructed the 2-form on moduli spaces by cutting the surface into a disjoint union of polygons. It is also interesting to consider any finite number of  cuts, without  cutting all the way to polygons.

\subsection{Cutting along paths between vertices}\label{subsec:cut}
Let us first describe single cuts. Suppose $(\Sigma,\V)$ satisfies (A1),(A2), and that $(\Sigma',\V')$ is obtained by cutting  along an embedded path $\gamma$, with end points in $\V$ and the rest of the path in the interior of  $\Sigma$. The cut creates two new boundary edges: $\#\E'=\#\E+2$. 

There are various cases, depending on whether the end points of $\gamma$ are interior vertices or boundary vertices, and whether or not $\gamma$ is a loop: 
\begin{enumerate}
	\item $\gamma$ is a loop based at a boundary vertex
	\item $\gamma$ is a loop based at an interior vertex
	\item $\gamma$ is a path between distinct boundary vertices
	\item $\gamma$ is a path between a boundary vertex and an interior vertex
	\item $\gamma$ is a path between distinct interior vertices
\end{enumerate}
In cases (a) and (c) the cut created two new vertices, while $\chi(\Sigma)$ increases by $1$. 
See the picture below for a cut of type (a). 
\begin{center}
	\includegraphics[width=0.5\textwidth]{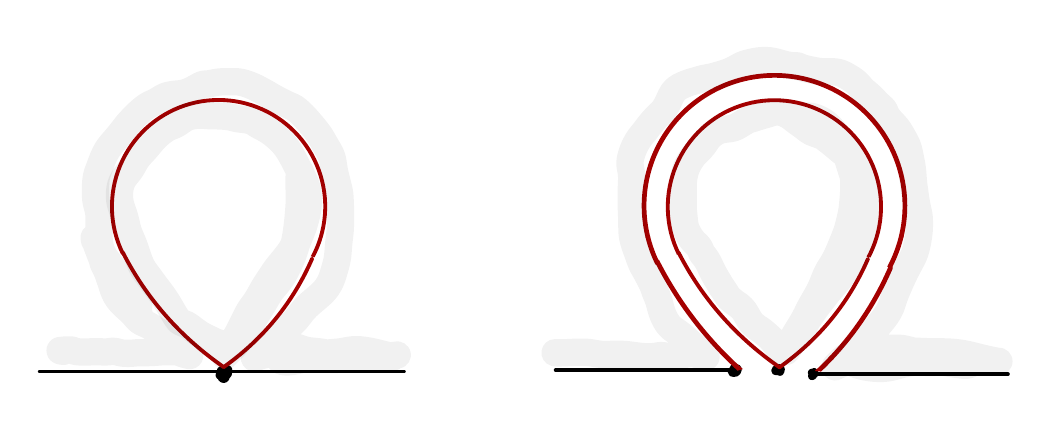}
\end{center}
(The region inside the loop may contain handles and boundary components.) In cases (b) and (d) the cut creates one new vertex, while the Euler characteristic is preserved. See the picture below for a cut of type (b). 
\begin{center}
	\includegraphics[width=0.4\textwidth]{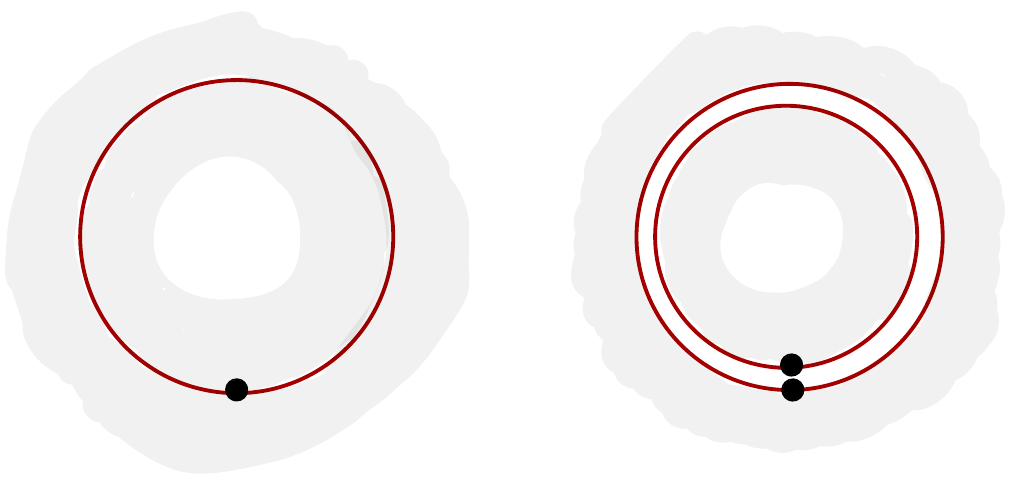}
\end{center}
Finally, a cut of type  (e) creates no new vertices, and the Euler characteristic decreases by $1$. 
Observe that in all of the cases, 
\[ \#\V'-\chi(\Sigma')=\#\V-\chi(\Sigma)+1,\]
corresponding to one new generator of the fundamental groupoid created by the cut.  Accordingly, 
the  quotient map 
defines an embedding 
\[ \M_G(\Sigma,\V)\hra \M_G(\Sigma',\V')\]
as a submanifold of codimension equal to $\dim G$.

\begin{example}
	Let $\Sigma=\Sigma^1_1$ be the one-holed torus, with $\V$ consisting of one vertex on its boundary.  
	Consider a cut (type (a)) along an embedded  loop $\gamma$ based at the vertex, winding once around the handle: 
	\begin{center}
		\includegraphics[width=0.3\textwidth]{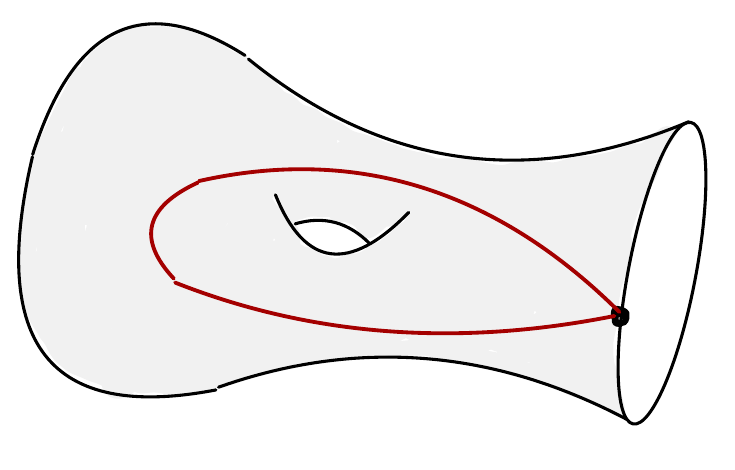}
	\end{center}
	The cut surface is connected, with two boundary components, and has Euler characteristic $\chi(\Sigma')=\chi(\Sigma)+1=0$. It is 
	thus an annulus $\Sigma'=\Sigma_0^2$. 
	\begin{center}
		\includegraphics[width=0.2\textwidth]{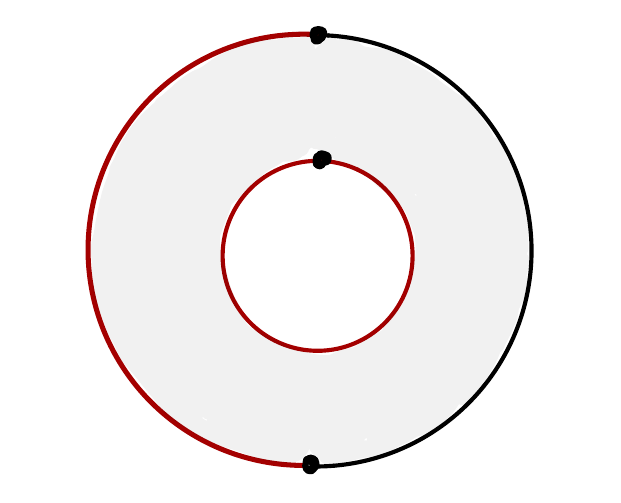}
	\end{center}
	The set 
	$\V'$ given by one vertex on one of the boundaries and two vertices on the other.  
\end{example}
\begin{example}
Consider a similar example of 	$\Sigma=\Sigma^1_1$, with $\#\V=2$ where one vertex is in the interior and one vertex on the boundary. Cut along a loop based at the interior vertex and winding around the handle. The result of this type (b) cut is a 3-holed sphere (pair of pants) with one vertex on each boundary component. 
\end{example}

Suppose more generally that $(\Sigma',\V')$ is obtained from 
$(\Sigma,\V)$ by iterated cuts of the type considered above. The quotient map 
induces a map  of fundamental groupoids, and hence a map of moduli spaces.

\begin{tcolorbox}
	\begin{proposition}\label{prop:partialcut}
		Suppose that $(\Sigma,\V)$ satisfies (A1),(A2), and that $(\Sigma',\V')$ is obtained by $N$ of cuts of the type considered above. Then the induced map 
		\[\M_G(\Sigma,\V)\to \M_G(\Sigma',\V')\]
		is an embedding as a submanifold of codimension $N\dim G$, 
		equivariant for the inclusion $G^\V\hra G^{\V'}$. The 2-form on $\M_G(\Sigma,\V)$ is the pullback of that on $\M_G(\Sigma',\V')$ under the embedding, and the boundary holonomy map $\Phi$ is given by restriction of 
		$\Phi'$ followed by the projection $G^{\E'}\to G^\E$. 
	\end{proposition}
\end{tcolorbox}

\begin{proof}
	Cutting further, we arrive at a gluing diagram $(\wh{\Sigma},\wh{\V})$ for $(\Sigma,\V)$, which also serves as a gluing diagram  for $(\Sigma',\V')$ by omitting the edge identifications for the quotient map $\Sigma'\to \Sigma$. 
	The various claims are evident from this description. In particular, 
	the forms $\omega,\omega'$ on the two moduli spaces are obtained by pullback of $\wh{\omega}$ 
	via embeddings 
	\[ \M_G(\Sigma,\V)\hra \M_G(\Sigma',\V')\hra \M_G(\wh{\Sigma},\wh{\V}).\qedhere\]	
\end{proof}

\subsection{Gluing equals reduction}\label{subsec:gluingreduction}
Sometimes, it is desirable to cut surfaces along loops not containing any vertices. 
To do so, simply add an  interior vertex $\vz$ on the loop, and use the same method as before. 

Suppose that $(\Sigma,\V)$ satisfies the assumptions (A1), (A2), and let $\Sigma'$ be the surface obtained by cutting $\Sigma$ along an embedded circle $C\subset \on{int}(\Sigma)$ with $C\cap \V=\emptyset$. 
	\begin{center}
	\includegraphics[width=0.5\textwidth]{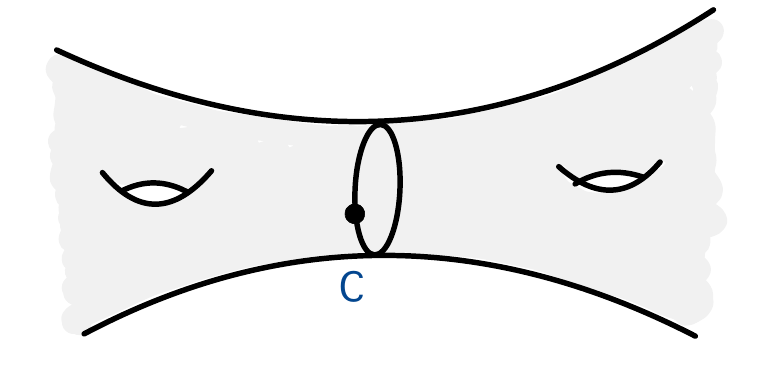}
\end{center}
Conversely, letting $C_\pm\subset \Sigma'$ be resulting new boundary components, the surface $\Sigma$ is recovered 
as a quotient $\Sigma=\Sigma'/\sim$ 
 by gluing along $C_\pm$. 
	\begin{center}
	\includegraphics[width=0.5\textwidth]{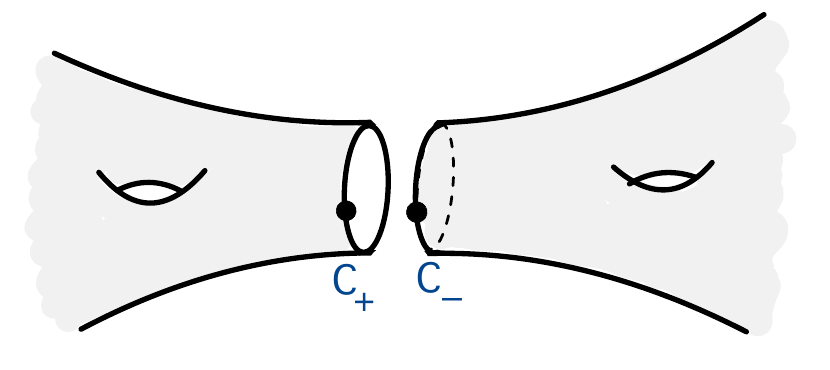}
\end{center}
We may regard $\V$ as a subset of $\Sigma'$ by taking the pre-image under the quotient map. The pair $(\Sigma',\V)$ does not satisfy (A2) since $\V$ does not meet all boundary components.  To fix this issue, choose a point $\vz\in C$, 
with pre-images $\vz_\pm \in C_\pm$ and let 
\[ \V'=\V\sqcup \{\vz_+,\vz_-\}.\]
Then $(\Sigma',\V')$ satisfies (A2). Note also that if $(\Sigma,\V)$ satisfies (A3) then so does $(\Sigma',\V')$. 

The set $\E'$ of boundary edges of $\Sigma'$ is $\E'=\E\sqcup \{\ez_+,\ez_-\}$ where $\ez_\pm$ are defined by the 
newly created boundaries $C_\pm$. 
Let $\Phi_\pm=\Phi_{\ez_\pm}\colon \M_G(\Sigma',\V')\to G$ denote the holonomies along these new boundary edges. 
The two maps 
\[ (\Sigma,\V)\lra  (\Sigma,\V\cup\{\vz\}) \longleftarrow (\Sigma',\V')\] 
induce  an  inclusion and projection 
\begin{equation}\label{eq:iotapi}
 \M_G(\Sigma,\V) \stackrel{\pi}{\longleftarrow} 
 \M_G(\Sigma,\V\cup\{\vz\}) 
 \stackrel{\iota}{\longrightarrow} \M_G(\Sigma',\V'),
\end{equation}
and the 2-forms are related by 
\[ \iota^*\omega'=\pi^*\omega.\]
We may think of this construction as a (quasi-)symplectic reduction: 
\[ \M_G(\Sigma,\V)=\M_G(\Sigma',\V')\qu G,\]
analogous to the `gluing equals reduction' result for Hamiltonian loop group spaces 
\cite{me:lo}.  
To make it more explicit, note that 
the image of the inclusion $\iota$ is the submanifold 
\[ Z=\{x|\ \Phi_+(x)=\Phi_-(x)^{-1}\}\subset \M_G(\Sigma',\V').\]
The diagonal $G\subset G\times G$-action (where the two $G$'s correspond to 
the two vertices $\vz_\pm$) is identified with the $G$-action 
on $\M_G((\Sigma,\V\cup\{\vz\})$ corresponding to $\vz$. As explained in Example \ref{ex:quotients} (see also Proposition \ref{prop:2form}) this $G$-action is free and  proper, and $\pi$ is the corresponding quotient map. 

\begin{remark}
A bit more generally, one can choose several points $\vz_1,\ldots,\vz_l\in C$, with corresponding pre-images 
$\vz_{i,\pm}$. Letting $\V'=\V\sqcup\{\vz_{i,\pm},i=1,\ldots,l\}$, one then has 
\[ \M_G(\Sigma,\V)=\M_G(\Sigma',\V')\qu G^l
\] 
where the right hand side is defined as the quotient of $\M_G(\Sigma,\V\sqcup\{\vz_i|\ i=1,\ldots,l\})$. In the opposite direction, this corresponds to the gluing of two boundary circles with the same number of base points. 
\end{remark}

We illustrate the `gluing equals reduction' principle with several examples.

\begin{example}
By cutting each of the handles of $\Sigma=\Sigma_g^r$, 
\begin{center}
	\includegraphics[width=0.40\textwidth]{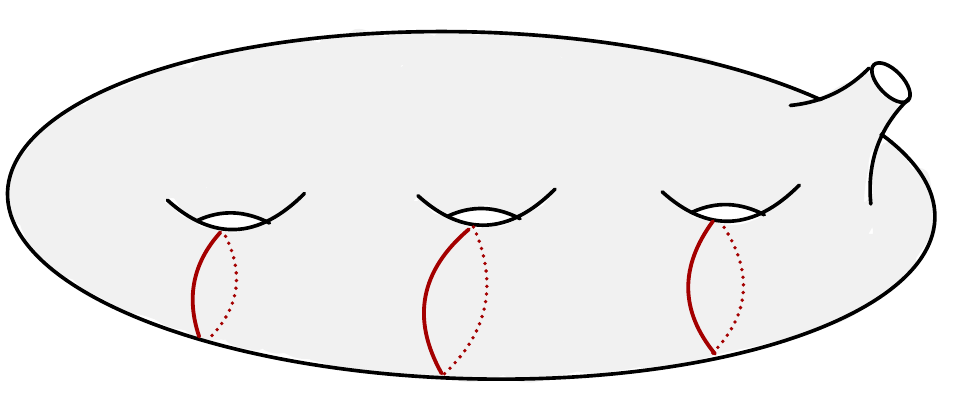}
\end{center}
one obtains
$\Sigma'=\Sigma_0^{r+2g}$ (a sphere with $r+2g$ holes). 
Given $\V\subset\p\Sigma$, let $\V'\subset \Sigma'$ be obtained by adding one pair of vertices for each pair of gluing circles $C_\pm$. 
Then 
\[ \M_G(\Sigma_g^r,\V)=\M_G(\Sigma_0^{r+2g},  \V')\qu G^g,\]
a $g$-fold reduction (with one copy of $G$ for each gluing circle). (The order of the reduction does not matter.) 
\end{example}

\begin{example}
Suppose $\Sigma=\Sigma_g^r$ with $\chi(\Sigma)=2-2g-r<0$. Then we may choose a 
\emph{pants decomposition}, given by a system of embedded circles cutting the surface into copies of 
$\Sigma_0^3$:
\begin{center}
	\includegraphics[width=0.40\textwidth]{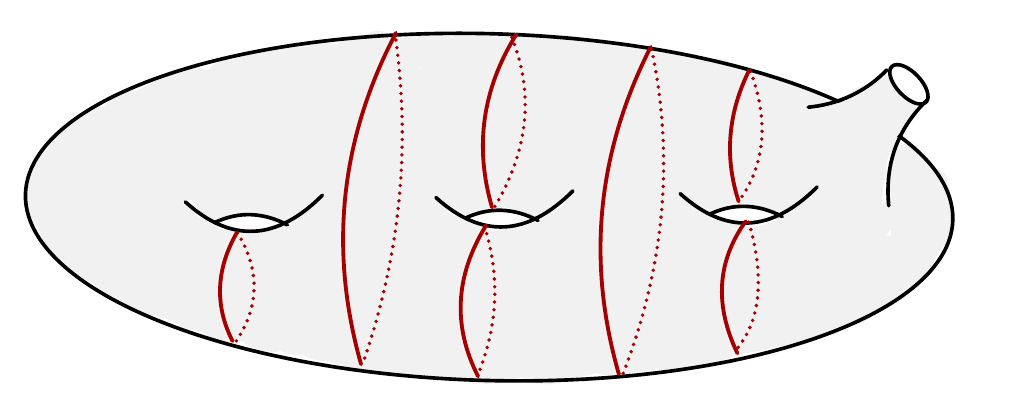}
\end{center}
Since $\chi(\Sigma_0^3)=-1$, the number of pants in such a decomposition is $-\chi(\Sigma)=2g+r-2$; the number of gluing circles is $3g+r-3$. After cutting, we obtain 
a disjoint union 
\[ \Sigma'=\Sigma_0^3\sqcup \cdots \sqcup \Sigma_0^3.\]
  Given $\V\subset \p\Sigma$, let $\V'\subset \p\Sigma'$ be obtained by adding a pair of vertices for each cut. Then 
\[ \M_G(\Sigma_g^r,\V)=\M_G(\Sigma_0^3\sqcup \cdots \sqcup \Sigma_0^3,\V')\qu G^{3g+r-3}.\]
\end{example}

\begin{example}\label{ex:cylindergluing}
Suppose $C\subset \on{int}(\Sigma)$ is homotopic to a boundary circle. 
\begin{center}
	\includegraphics[width=0.40\textwidth]{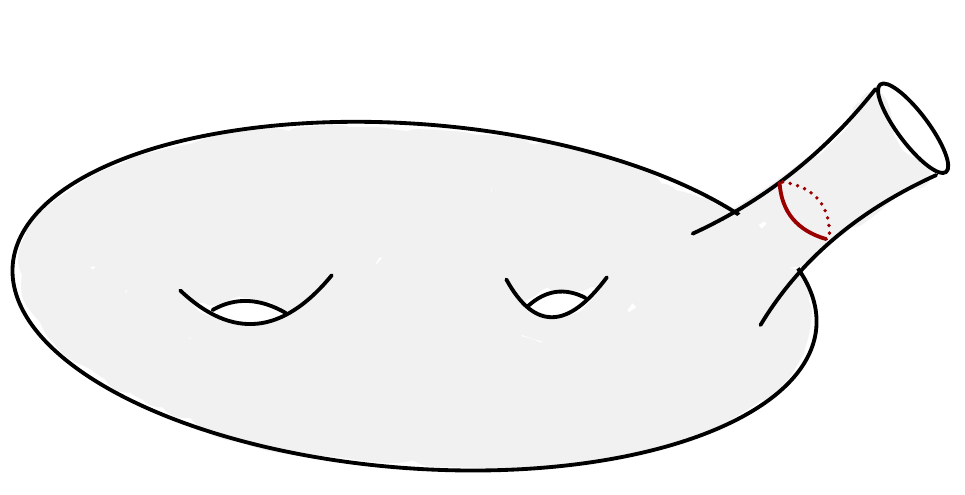}
\end{center}
Cutting along this circle produces a disjoint union of a surface diffeomorphic to $\Sigma$, and a cylinder.  
\[ \Sigma'=\Sigma\sqcup \Sigma_0^2.\]
Suppose for simplicity that $\V\subset\p\Sigma$ consists of one vertex on each boundary component. 
Then $\Sigma'$ has the same property, with $\V'=\V\cup \{\vz_1,\vz_2\}$, 
and 
\[ \M_G(\Sigma,\V)=\big(\M_G(\Sigma,\V)\times \M_G(\Sigma_0^2,\{\vz_1,\vz_2\})\big)\qu G\]
That is, the moduli space of the cylinder serves as the identity under reduction. As we shall explain in Section \ref{sec:quasisymplecticgroupoid} below, this property turns the moduli space of a cylinder into a \emph{quasi-symplectic groupoid}.  
\end{example}

\subsection{The symplectic structure for surfaces without boundary}
As another example for the `gluing reduction principle', consider the case that 
the loop $C\subset \on{int}(\Sigma)$ is contractible.  	
\begin{center}
		\includegraphics[width=0.40\textwidth]{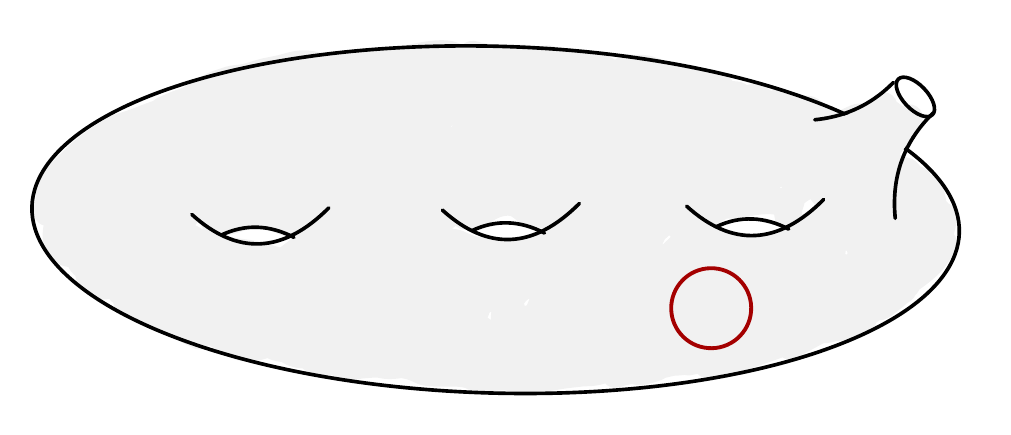}
\end{center}
Cutting along such a circle gives a disjoint union 
\[ \Sigma'=\Sigma''\sqcup \Sigma_0^1\]
where $\Sigma''$ is obtained from $\Sigma$ by removing  the open disk bounded by $C$. Conversely, we may think of 
$\Sigma$ as being obtained from $\Sigma''$ by ``capping off" a boundary component. Since 
$\M_G(\Sigma_0^1,\{\vz\})$ is just a point, we obtain  
\[ \M_G(\Sigma,\V)=\M_G(\Sigma'',\V'')\qu G=\Phi_C^{-1}(e)/G\]
where $\Phi_C\colon \M_G(\Sigma'',\V'')\to G$ is the holonomy around $C$. 
Note that the pair $(\Sigma'',\V'')$ satisfies the three assumptions (A1),(A2),(A3) if and only if $(\Sigma,\V)$ does. 
Consequently, in this case the 2-forms on the respective moduli spaces satisfy the minimal degeneracy property 
\eqref{it:c}.

This suggests a construction of the symplectic structure on the moduli space 
\[ \M_G(\Sigma)=\M_G(\Sigma,\emptyset)\] 
for a compact, connected surface $\Sigma$ \emph{without} boundary: Choose a  loop $C\subset \Sigma$ 
bounding a disk, as well as a base point $\vz\in C$, and let $\Sigma''$ be the surface with boundary $C$ obtained by removing the interior of the disk.  Thus, $\Sigma''\cong \Sigma_\gz^1$ where $\gz$ is the genus of the surface. 
Let $\omega''$ be the 2-form on $\M_G(\Sigma'',\{\vz\})$, and let 
$\Phi\colon \M_G(\Sigma'',\{\vz\})\to G$ be the holonomy around the boundary. (Using standard generators 
of the fundamental group, this is the product-of-commutators map 
$\Phi\colon G^{2\gz}\to G$.) 
Then 
\[ Z=\Phi^{-1}(e)=\M_G(\Sigma,\{\vz\})=\Hom(\pi_1(\Sigma,\vz),G),\]
and so 
\[ \M_G(\Sigma)=Z/G.\]
At this stage, we are faced with the problem that $e$ is \emph{not} a regular value of $\Phi$. 
By Corollary \ref{cor:goldman},  $\Phi$ has maximal rank exactly at those points $\kappa$ for which the stabilizer under the $G$-action is discrete. The set of elements in $Z$ having this property is a submanifold 
\[\iota\colon  Z_{\on{reg}}\hra \M_G(\Sigma'',\{\vz\}).\]
The property \eqref{it:a} from Theorem \ref{th:2form} shows that the pullback $\iota^*\omega''$ is 
closed. By \eqref{eq:explicitkernel}, the 2-form $\omega''$ on $\M_G(\Sigma'',\{\vz\})$ is non-degenerate at all points 
of $Z$, and by the momentum map property \eqref{it:b} the null foliation of 
 $\iota^*\omega''$ consists exactly of the orbit directions. Hence, if the $G$-action on $Z_{\on{reg}}$ is proper
 (which is automatic if $G$ is compact) then 
 \[ \M_G(\Sigma)_{\on{reg}}=Z_{\on{reg}}/G\] 
 has at worst orbifold singularities, and  $\iota^*\omega''$ descends to a symplectic 2-form $\omega$ on this space. 

\begin{remark}
If $G$ is compact, one gets better control over the singularities. 
In fact, $\Phinv(e)/G$ can be realized as a
singular symplectic quotient in this case, and $\M_G(\Sigma)$ has the structure of a stratified symplectic space in the sense of Sjamaar-Lerman \cite{sj:si}. 
\end{remark}

\subsection{Lagrangian boundary conditions}
In the previous section, we used reduction to arrive at 2-forms on moduli spaces that are actually symplectic, rather than just 
quasi-symplectic. A different method, producing many important examples, was introduced by \v{S}evera in \cite{sev:mod}. 

Suppose $(\Sigma,\V)$ satisfies assumptions (A1),(A2),(A3). Suppose furthermore that we are given a collection $\ul{H}=\{H_\ez\}$
of closed Lagrangian Lie subgroups
\[ H_\ez\subset G,\]
one for each edge. (A subgroup $H\subset G$ is called Lagrangian if its Lie algebra 
$\h\subset \g$ is Lagrangian, that is, $\h^\perp=\h$.)  
\begin{tcolorbox}
\begin{proposition}[\v{S}evera \cite{sev:mod}]	\label{prop:severa2}
	Suppose that $\h_{\ez}\cap \h_{\ez'}=0$ whenever $\sz(\ez')=\tz(\ez)$. 
Then 
\[ \M_G(\Sigma,\V;\ul{H})=\Phi^{-1}(\prod_{\ez\in \E} H_\ez)\] 
is a smooth submanifold, and the pullback of the 2-form $\omega$ to this submanifold is symplectic. 		
	\end{proposition}
\end{tcolorbox}
\begin{proof}
We shall use the following fact: If $\phi_1\colon K_1\to K$ and $\phi_2\colon K_2\to K$ are two morphisms of Lie groups 
such that $(T_e\phi_1)(\k_1)+(T_e\phi_2)(\k_2)=\k$, then the maps $\phi_1,\phi_2$ are transverse, and the product map 
\[ \Mult_K\circ (\phi_1\times \phi_2)\colon K_1\times K_2\to K\] 
is a submersion.

The submanifold
\begin{equation}\label{eq:q}
 Q=\prod_{\ez\in \E} H_\ez\subset G^\E.\end{equation}
is a product  $Q=\prod_C Q_C$ over the set of boundary components $C\subset \p\Sigma$, where 
 $Q_C=\prod_{\ez\in \E_C} H_\ez\subset G^\E_C$ is the product over the set $\E_C\subset \E$ of edges 
 that are contained in $C$. Let $\V_C=\V\cap C$. The quotient map for the $G^{\V_C}$-action on 
 $G^{\E_C}$ may be identified with the product map $G^{\E_C}\to G$ (defined after choice of an initial vertex in $\V_C$). 
 By the fact above, this quotient map remains a submersion when restricted to $Q_C$. It follows that $Q_C$ is transverse to the $G^{\V_C}$-orbits, and hence $Q$ is transverse to the $G^\V$-orbits. Since $\Phi$ is $G^\V$-equivariant, it follows that  
 $Q$ is transverse to the map $\Phi$, and hence $\Phi^{-1}(Q)$ is a submanifold. 
 
Let  $\iota\colon \Phi^{-1}(Q)\to \M_G(\Sigma,\V)$ be the inclusion map.  
Since the subgroups $H_\ez\subset G$ are Lagrangian, the pullback of $\eta\in \Omega^3(G)$ to these subgroups vanishes. 
Hence, applying $\iota^*$ to the identity $\d\omega=-\sum_{\ez}\Phi_\ez^*\eta$, we obtain 
\[ \d \iota^*\omega=0.\] 
The fact that $\iota^*\omega$ is symplectic may be proved using the cross-section theorem 
from Dirac geometry. We will give this argument in Section \ref{subsec:lagrangian}.
\end{proof}

The argument that \eqref{eq:q} is transverse to $G^\V$-orbits, and hence is transverse to $\Phi$, 
works in the same way provided that every boundary component contains at least one vertex $\vz$ such that 
$\h_\ez+\h_{\ez'}=\g$ for $\sz(\ez')=\vz=\tz(\ez)$. 

\begin{remark}
The assumption that the Lagrangian Lie subalgebras $\h_\ez\subset\g$ correspond to \emph{closed} subgroups is quite restrictive. 
The proposition extends to the more general setting of just having integration to \emph{immersed subgroups}, by defining 
$\M_G(\Sigma,\V;\ul{H})$ as a fiber product of $\prod_{\ez\in \E} H_\ez$ with $\M_G(\Sigma,\V)$ over $G^\E$. 
\end{remark}

\begin{example}
A pair of transverse Lagrangian Lie subalgebras $\h_1,\h_2\subset \g$ with transverse intersection defines a \emph{Manin triple}. 
Suppose that the inclusions maps $\h_i\subset \g_i$ integrate to Lie group morphisms $\phi_i\colon H_i\to G$. By Drinfeld's theory \cite{dri:quas}, the Manin triple determines Poisson structures on $H_1,H_2$, making them into dual Poisson Lie groups. 
Lu-Weinstein \cite{lu:gr} obtained an integration of these Poisson Lie groups  to a symplectic double Lie groupoid
\[ 
\begin{minipage}{4cm}
{
	\xymatrix{ S\ar[r]<2pt>\ar@<-2pt>[r]  \ar[d]<2pt>\ar@<-2pt>[d]  & H_2 \ar[d]<2pt>\ar@<-2pt>[d]\\
		H_1 \ar[r]<2pt>\ar@<-2pt>[r]  & \pt}
}\end{minipage}
\]
As a manifold, 
\[ S=\{(h_1,h_2,h_1',h_2')\in H_1\times H_2\times H_1\times H_2|\ 
\phi_1(h_1)\phi_2(h_2)=\phi_2(h_2')\phi_1(h_1')\}.\]
As shown in \cite{sev:mod}, the symplectic structure 
is conveniently described by picturing the elements of $S$ 
as `commuting squares',
	\begin{center}	\includegraphics[width=0.25\textwidth]{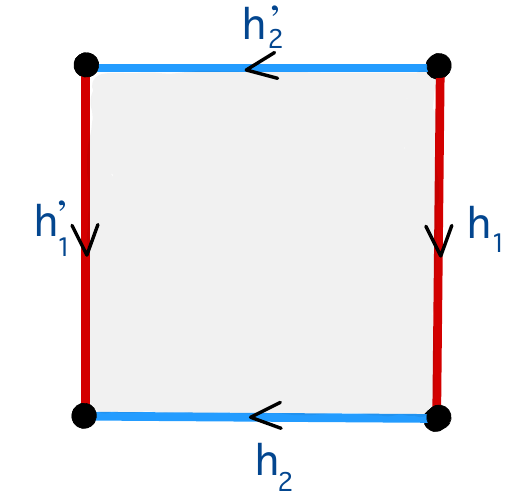}\end{center}
In fact, $S=\M_G(\Sigma,\V,\ul{H})$ where $\Sigma\cong \Sigma_0^1,\ \#\V=4$ is the 4-gon, with edges labeled 
by the groups $H_i$ as indicated. The two groupoid multiplications are understood as horizontal  and vertical concatenation
of squares.  
\end{example}

For other examples along these lines, see \cite{sev:mod}. The idea of `colouring' edges with Lagrangian subgroups can be carried much further, see \cite{alv:poi,boa:qu,boa:fis,lib:mod,lib:sypo}.

\section{Poisson structures and Goldman flows}\label{sec:goldman}
For a symplectic manifold $(M,\omega)$, any function $f\in C^\infty(M)$ defines a \emph{Hamiltonian vector field} $X_f$ by the equation 
 \begin{equation}\label{eq:hamvf} \iota(X_f)\omega=-\d f.\end{equation} 
This vector field satisfies $\L_{X_F}\omega=0$; hence its flow preserves the symplectic structure. 
The 2-form on our moduli spaces  $\M_G(\Sigma,\V)$ is usually degenerate, hence vector fields satisfying 
\eqref{eq:hamvf} may not exist, and are not unique in case they do.  As we shall explain, these issues can be resolved 
provided that $f$ is $G^\V$-invariant. For certain choices of $f$, the resulting Hamiltonian flows are the \emph{Goldman flows}.

\subsection{Hamiltonian vector fields}
Throughout,  we shall make  the assumptions (A1),(A2),(A3): All components of $\Sigma$ have non-empty boundary, and the set 
$\V$ is contained in the boundary and meets all components of the boundary. Hence $\M_G(\Sigma,\V)$ carries a $G^\V$-invariant 2-form  $\omega$ and a $G^\V$-equivariant map (boundary holonomies) 
\[ \Phi\colon \M_G(\Sigma,\V)\to G^\E,\] 
satisfying the properties listed in Theorem \ref{th:2form}. 

\begin{tcolorbox}
	\begin{theorem}
		If $f\in C^\infty(\M_G(\Sigma,\V))$ is $G^\V$-invariant, there is a unique \emph{Hamiltoian vector field} $X_f$ on $\M_G(\Sigma,\V)$ satisfying 
		\[ \iota(X_f)\omega=-\d f\]
		and such that the flow of $X_f$ fixes the boundary holonomies. (That is, $X_f$ is $\Phi$-related to $0$.) 
		These vector fields satisfy 
		\[ \L(X_f)\omega=0.\]
		On the open subset where $\Phi$ has maximal rank, the span of the Hamiltonian vector fields for invariant functions is exactly
		$\ker(T\Phi)$. 
	\end{theorem}
\end{tcolorbox}

This is a mild generalization of \cite[Proposition 4.6]{al:mom}.  
The proof  is best done within the framework of Dirac geometry; see Proposition \ref{prop:admissible} below. 
Using Hamiltonian vector fields, we obtain a Poisson bracket on the space  $C^\infty(\M_G(\Sigma,\V))^{G^\V}$ 
of invariant functions, by the usual formula 
		\[ \{f,g\}=\L_{X_f}g=\omega(X_f,X_g).\]		
In this way, the moduli space for a surface with boundary  
\[ \M_G(\Sigma)=\M_G(\Sigma,\V)/G^\V\]  
becomes canonically a Poisson manifold, possibly with singularities (since the $G^\V$-action need not be free and proper). 

Suppose $(\Sigma',\V')$ is obtained from $(\Sigma,\V)$ by cutting along an embedded circle $C\subset \on{int}(\Sigma)$,  
as in Section \ref{subsec:gluingreduction}. Thus 
\[ \M_G(\Sigma,\V)=\M_G(\Sigma',\V')\qu G\cong \M_G(\Sigma,\V\sqcup\{\vz\})/G
.\]
There is a natural map 
\[ C^\infty(\M_G(\Sigma',\V'))^{G^{\V'}}\to C^\infty(\M_G(\Sigma,\V\cup\{\vz\}))^{G^{\V\cup\{\vz\}}}\cong C^\infty(\M_G(\Sigma,\V))^{G^{\V}}.\]

\begin{tcolorbox}
	\begin{proposition} 
		The map on functions just described preserves Poisson brackets. 
	\end{proposition}	
\end{tcolorbox}

\begin{proof}
Write $Z=	\M_G(\Sigma,\V\cup\{\vz\})$, and let  
$\iota\colon Z\to \M_G(\Sigma',\V')$ be the inclusion, and 
$\pi\colon Z\to \M_G(\Sigma,\V)$ the projection. 	
Let $f',g'$ be $G^{\V'}$-invariant function on $\M_G(\Sigma',\V')$, and $f,g$ the resulting functions on 
$\M_G(\Sigma,\V)$, thus $\pi^*f=\iota^*f',\ \pi^*g=\iota^*g'$. 

Since the Hamiltonian vector field $X_{f'}$ is $\Phi'$-related to $0$, it  
is tangent to the level sets of $\Phi'$, and in particular is tangent to $Z$. By equivariance, $X_{f'}|_Z$ descends to a vector field on $\M_G(\Sigma,\V)$. This vector field is $\Phi$-related to zero, and its contraction with $\omega$ equals $-\d f$. It hence follows that $X_{f'}|_Z\sim_\pi X_f$. Using this fact, it follows that  
\[ \iota^*\{f',g'\}=\iota^* X_{f'}(g')=\pi^* X_f(g)=\pi^*\{f,g\}.\qedhere\] 
\end{proof}

\subsection{Goldman flows}
Interesting  examples of invariant functions on the moduli space may be constructed from conjugation invariant functions 
$\varphi\colon G\to \R$ and loops $\alpha\colon S^1\to \Sigma$. To define these functions 
\begin{equation}\label{eq:f} \varphi_\alpha\colon \M_G(\Sigma,\V)\to \R,\ \ 
\end{equation}
suppose first that $\alpha(0)=\alpha(1)=\vz \in \V$. Letting $\az\in \Pi(\Sigma,\V)$ be the (non-free) class
of $\alpha$ in the fundamental groupoid, and $\on{ev}_\az\colon  \M_G(\Sigma,\V)\to G$ the evaluation map
(holonomy along $\az$), we put 
\begin{equation}\label{eq:construction}
 \varphi_\alpha=\varphi\circ \on{ev}_\az.\end{equation} 
It is clear that $\varphi_\alpha$ is $G^\V$-invariant. 
\begin{lemma}
The function $\varphi_\alpha$ depends only on the \emph{free} homotopy class of $\alpha$. 
\end{lemma} 
\begin{proof} 
 Replacing $\az$ with $\az'=\bz \az \bz^{-1}$ for some 
$\bz\in \Pi(\Sigma,\V)$ with $\sz(\bz)=\vz,\ \tz(\bz)=\vz'$, the holonomy changes by conjugation, 
hence $\varphi\circ \on{ev}_\az=\varphi\circ \on{ev}_{\az'}$. \end{proof}

Hence, the functions \eqref{eq:f} are defined for arbitrary loops $\alpha$, not necessarily based at some point of $\V$. 
By $G^\V$-equivariance, they descend to functions on  $\M_G(\Sigma)=\M_G(\Sigma,\V)/G^\V$. 

In particular, the functions $\varphi_\alpha$ are also defined for surfaces without boundary. For this case, the  Poisson brackets $\{\varphi_\alpha,\psi_\beta\}$ of  functions of this type 
were calculated in the work of Goldman \cite{gol:inv}, leading to the  \emph{Goldman bracket} on free homotopy classes of loops. See also 
Massuyeau-Turaev \cite{mas:qua} for the case with boundary. 

We shall compute the Hamiltonian vector field of \eqref{eq:f} in simple cases. For every conjugation invariant function $\varphi$, 
let $\dot\varphi\in C^\infty(G,\g)$  be defined by 
\begin{equation}\d\varphi=\theta^L\cdot \dot{\varphi}\end{equation} 
Notice that $\Ad_h \dot{\varphi}(g)=\dot{\varphi}(hgh^{-1})$, hence we also have $\d\varphi=\theta^R\cdot \dot{\varphi}$. 
Using this notation, we have $\d \varphi_\alpha=\on{ev}_\az^*(\theta^L\cdot \dot{\varphi})$.

\begin{example}[Moduli space of cylinder] \label{ex:cylinder}
	Consider the example $(\Sigma,\V)=(\Sigma_0^2,\{\vz_1,\vz_2\})$, 
from Section \ref{subsec:cylinderexample}. 
Using the notation given there, take $\alpha$ to be the oriented boundary loop based at $\vz_1$; thus $\on{ev}_\az=\Phi_1$.  
We claim that the Hamiltonian vector field $X_f$ for $f=\varphi_\alpha$  
is given by
\[ \iota(X_f)a^*\theta^L=0,\ \ \iota(X_f)c^*\theta^L=\dot{\varphi}(a),\]
with corresponding flow 
\[ (a,c)\mapsto (a,c\exp(-t\dot{\varphi}(a))).\] 
	%
To see this, note first that $X_f\sim_\Phi 0$ since the flow preserves both momentum map components $\Phi_1(a,c)=a,\ \Phi_2(a,c)=ca^{-1}c$.
	Also,  using the explicit  formula for $\omega$ (see Proposition \ref{prop:cylinder}), 
	\begin{align*}
	\iota(X_f)\omega&=-\hh a^*(\theta^L+\theta^R)\cdot \dot{\varphi}(a) -\hh \dot{\varphi}(a) \Ad_a c^*\theta^L
	+\hh c^*\theta^L\cdot \Ad_a \dot{\varphi}(a)\\&= -a^*\theta^L\cdot \dot{\varphi}(a)\\&=-\d \varphi_\alpha
	\end{align*}
	as desired. In terms of the  $G^\V=G\times G$-action  $(h_1,h_2).(a,c)=(\Ad_{h_1}a,h_2 c h_1^{-1})$, the flow of $X_f$ is given by 
	the action of $(\exp(t\dot{\varphi}(a)),e)$. 
\end{example}
\smallskip
Generalizing this example, we have: 
\begin{tcolorbox}
\begin{proposition}\label{prop:bdryaction}
Suppose $(\Sigma,\V)$ satisfies (A1),(A2),(A3), and that $\alpha\colon [0,1]\to \p\Sigma$ is a boundary loop
based at $\vz\in \V$
with class  $\az\in \Pi(\Sigma,\V)$. Suppose also that the 
component of $\p\Sigma$ containing $\vz$ does not contain any other vertices. 
Given an invariant function $\varphi\in C^\infty(G)$, the flow of the Hamiltonian vector field $X_{\varphi_\alpha}$ is given by 
\[ \kappa_t=\exp(t\dot{\varphi}(\kappa(\az)))\cdot \kappa\]
using the action of the factor $G\subset G^\V$ corresponding to $\vz\in\V$. 
\end{proposition}	
\end{tcolorbox}
Letting $\xi=\dot{\varphi}(\kappa(\az))$ (for fixed $\kappa\in \M_G(\Sigma,\V)$), the proposition says that for all 
$\bz\in \Pi(\Sigma,\V)$, 
 \[ \kappa_t(\bz)=\begin{cases}
 \kappa(\bz)&  \sz(\bz)\neq \vz\neq \tz(\bz),\\
 \kappa(\bz)\exp(-\t\xi)& \sz(\bz)=\vz\neq \tz(\bz),\\
 \exp(t\xi) \kappa(\bz)& \tz(\bz)=\vz\neq \sz(\bz),\\
 \exp(t\xi) \kappa(\bz)\exp(-\t\xi)& \sz(\bz)=\tz(\bz)=\vz.
 \end{cases}
\]
\begin{proof}
	The curve $\kappa_t=\exp(t\xi)\cdot \kappa$ represents the tangent vector
	$\xi_{\M_G(\Sigma,\V)}|_\kappa$. Since $\kappa_t(\ez)=\kappa(\ez)$ for every boundary edge, 
	we see that $\xi_{\M_G(\Sigma,\V)}|_\kappa\in \ker(T\Phi|_\kappa)$. 
	The contractions with 
	$\omega$ are obtained from the moment map property \eqref{it:b}: 
	\[ \iota(\xi_{\M_G(\Sigma,\V)})\omega|_\kappa=-\hh \on{ev}_\az^*(\theta^L+\theta^R)|_\kappa\cdot \xi
	=-\hh \on{ev}_\az^*(\theta^L\cdot \dot{\varphi}+\theta^R\cdot \dot{\varphi})|_\kappa=
	-\d f|_\kappa.\]
	This shows that $\xi_{\M_G(\Sigma,\V)}|_\kappa=
	X_{\varphi_\alpha}|_\kappa$.
\end{proof}

We use this to prove the following result, which is a version of a theorem of Goldman \cite[Section 4]{gol:inv} (for surfaces without boundary). It involves the \emph{local intersection number} 
\[ I_p(\alpha,\beta)\in \{+1,-1\}\] for transverse intersections of paths $\alpha,\beta$ at a point $p\in \Sigma$, defined by the following pictures: 
\begin{center}
	\includegraphics[width=0.45\textwidth]{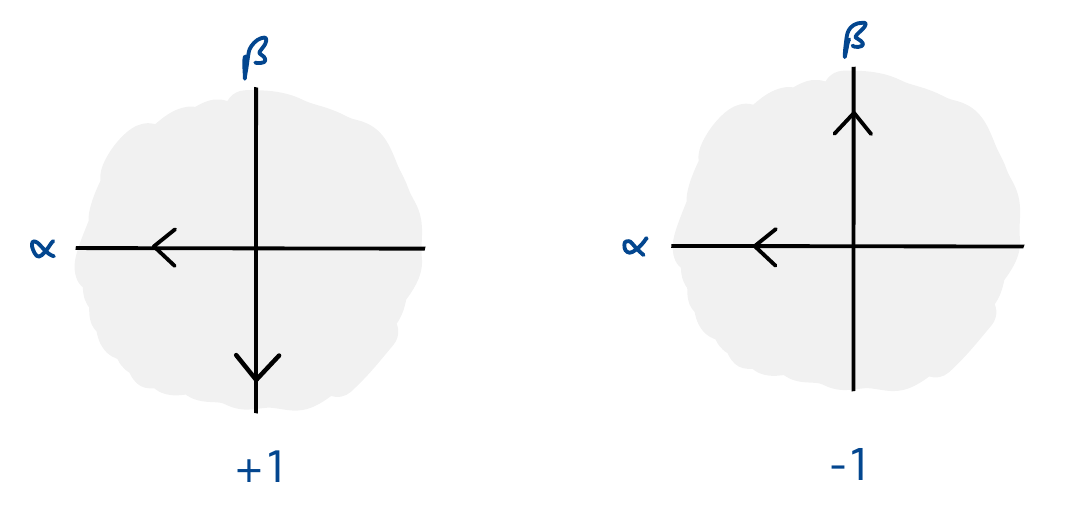}
\end{center}

\begin{tcolorbox}
\begin{theorem}[Goldman flows] \label{th:goldman}
Let $\Sigma$ be a connected surface with non-empty boundary and a 
non-empty collection of vertices $\V\subset \p\Sigma$. Suppose 
\[ \alpha\colon S^1\to \on{int}(\Sigma)\] 
is a {{\bf simple}} loop in its interior, and $\varphi\colon G\to \R$ is an invariant function. Then the 
flow $\kappa\mapsto \kappa_t$ of the  Hamiltonian vector field of $f=\varphi_\alpha$ is described on elements $\bz\in \Pi(\Sigma,\V)$ as follows. 

Represent $\bz$ by a path $\beta\colon [0,1]\to \Sigma$, having transverse intersection with $\alpha$,  with intersection points 
$p_i=\beta(t_i)$ for $t_1<\ldots<t_l$. Then 
\[ \kappa_t(\bz)=\prod_{i=1}^l \exp\big(t\, \epsilon_i\dot{\phi}(\kappa(\az_i))\big)
\kappa(\bz).\]
Here $\epsilon_i=I_{p_i}(\alpha,\beta)=\pm 1$ is the local intersection number at $p_i$, and 
$\az_i\in \Pi(\Sigma,\V)$ is the class of a loop based at $\tz(\bz)$, given by the negatively oriented segment of $\beta$ from 
$\tz(\bz)$ to the intersection point $p_i$,  followed by a positively oriented loop 
 around $\alpha$, followed by the positively oriented segment of $\beta$ from $p_i$ to 
$\tz(\bz)$. 
	\end{theorem}
\end{tcolorbox}
\begin{proof}
The range of $\alpha$ is an oriented  circle $C\subset\on{int}(\Sigma)$. The idea of proof is to cut $\Sigma$ along $C$, 
as in Section \ref{subsec:gluingreduction}, thus turning $\alpha$ into a boundary loop. Proposition \ref{prop:bdryaction} describes the Hamiltonian flow on the moduli space of the cut surface; our task is to deduce the resulting Hamiltonian flow 
on $\M_G(\Sigma,\V)$. 
	
We use the notation from Section \ref{subsec:gluingreduction}. Thus, we introduce the interior vertex $\vz=\alpha(0)$, and let $(\Sigma',\V')$ be the surface obtained by cutting along $C$. It has two new boundary components $C_\pm$, and two 
new vertices $\vz_\pm$ obtained as the pre-images of $\vz\in C$ under the quotient map. 
We choose the labeling in such a way that the boundary orientation on $C_+$ agrees with the orientation induced from $C$. 	The loop $\alpha$ lifts to a boundary loops $\alpha_\pm\colon [0,1]\to \Sigma'$, based at $\vz_\pm$, with image $C_\pm$.

The quotient map 
$\Sigma'\to \Sigma$ gives a morphism of fundamental groupoids
\[ \Pi(\Sigma',\V')\to \Pi(\Sigma,\V\sqcup\{\vz\})\supseteq
\Pi(\Sigma,\V).\]
Given $\bz\in \Pi(\Sigma,\V)$, a representative path $\beta\colon [0,1]\to \Sigma$ as in the statement of the theorem determines elements $\bz_0',\ldots,\bz_l'\in \Pi(\Sigma',\V')$, with images $\bz_0,\ldots,\bz_l\in \Pi(\Sigma,\V)$, such that   
\begin{equation}\label{eq:bproduct} \bz=\bz_0\cdots \bz_l.\end{equation} 

To construct these elements, use the points of intersection $p_i=\beta(t_i)$ to subdivide the path $\beta$ into 
segments. The end points of these segments need not lie in $\V\cup\{\vz\}$,  but this is remedied by sliding the end points along $C$. Thus, for each $i=1,\ldots,l-1$ we let 
\[ \beta_i\colon [0,1]\to \Sigma\]
be a path (for some choice of parametrization) given by the positively oriented segment of  $\alpha$ from $\vz$ to $p_{i+1}$, followed by the segment of $\beta$ from 
$p_{i+1}$ to $p_i$, followed by the negatively oriented segment of $\alpha$ from $p_i$ to $\vz$. For $i=l$, 
we let $\beta_l$ be the path given by the segment of $\beta$ from $\beta(0)$ to $p_1$ followed by the 
negatively oriented segment of $\alpha$ from $p_1$ to $\vz$; the description of $\beta_0$ is similar.  
The $\beta_i$ lift uniquely to paths $\beta_i'$ in $\Sigma'$, and we take $\bz_i$ (resp. $\bz_i'$) the corresponding elements of the fundamental groupoid. 

Notice that for all $i>0$, the initial point of the lifted path $\beta_{i-1}'$ equals $v_\pm$, where the 
sign is given by the  intersection number $\eps_i=I_{p_i}(\alpha,\beta)=\pm 1$. Furthermore, the end point of the subsequent path
$\beta_i'$ is then $v_\mp$. Hence
\[ \sz(\bz_{i-1}')=\vz_\pm,\ \ \tz(\bz_i')=\vz_\mp \ \mbox{ if }\eps_i=\pm 1.\]
Let $\az_\pm\in \Pi(\Sigma',\V')$ be the elements defined by the loops $\alpha_\pm$. 
Recall from Section \ref{subsec:gluingreduction}
that $\M_G(\Sigma,\V\sqcup\{\vz\})$ is identified with the submanifold $Z\subset\M_G(\Sigma',\V')$, consisting 
of homomorphisms $\kappa'\colon \Pi(\Sigma',\V')\to G$ such that 
$\kappa'(\az_+)=\kappa'(\az_-)$. In this case, $\kappa'$ descends to a homomorphism 
$\kappa\colon \Pi(\Sigma,\V\sqcup\{\vz\})\to G$, which then restricts to a homomorphism $\kappa\colon \Pi(\Sigma,\V)\to G$. 

The function $\varphi_{\alpha_+}\colon \M_G(\Sigma',\V')\to \R$ descends to $\varphi_\alpha$. Its 
flow $\kappa'\mapsto \kappa_t'$ satisfies $\kappa_t'(\az_+)=\kappa'(\az)$. Hence, the flow preserves $Z$, and 
descends  to the flow $\kappa\mapsto \kappa_t$ of $X_{\varphi_\alpha}$. On the element $\bz$ considered above, 
we obtain (for $\kappa'\in Z$ lifting $\kappa$)
\[ \kappa_t(\bz)=\kappa_t(\bz_0)\cdots \kappa_t(\bz_l)=\kappa_t'(\bz_0')\cdots \kappa_t'(\bz_l').\]

Proposition \ref{prop:bdryaction} provides a description of the flow $\kappa_t'$: 
letting $\xi=\dot{\phi}(\kappa'(\az_+))$, the element $\kappa'(\bz_i')$ gets multiplied from the left by 
$\exp(t\xi)$ if $\tz(\bz_i')=\vz_+$, and from the right by $\exp(-t\xi)$ if $\sz(\bz_i')=\vz_+$. 
It follows that 
\begin{align*}
\kappa_t(\bz)&=\kappa(\bz_0)\exp( t\eps_1\xi)\kappa(\bz_1)\exp( t\eps_2\xi)\cdots \kappa(\bz_l)\\
&=\prod_{i=1}^l \exp\Big( t\eps_i \Ad_{\kappa(\bz_0)\cdots \kappa(\bz_{i-1})}\xi\Big)
\kappa(\bz_0)\cdots \kappa(\bz_l)\\
&=\prod_{i=1}^l \exp\Big( t\eps_i \dot\varphi(\kappa(\az_i))\Big)\kappa(\bz)
\end{align*}
Here we used the property $\Ad_g \dot{\varphi}(h)=\dot{\varphi}(ghg^{-1})$ and 
\[ \az_i=\bz_0\cdots \bz_{i-1}\az_+\bz_{i-1}^{-1}\cdots \bz_0^{-1},\]
by the description of $\az_i$ given in the statement of the theorem. 
\end{proof}

The description of Hamiltonian flows allows us to compute Poisson brackets between functions of the form $\varphi_\alpha$. We obtain the following result (due to Goldman \cite{gol:inv} in the case without boundary).

\begin{tcolorbox}
	\begin{proposition}
		Suppose $\alpha,\beta\colon S^1\to \Sigma$ are  loops with transverse intersection, and let  
		$\varphi,\psi\in C^\infty(G)$ be invariant functions. We assume that $\alpha$ is simple. 
		Then the Poisson
	    bracket is given by 
	\[  \{\varphi_\alpha,\psi_\beta\}|_\kappa=\sum_i \epsilon_i\ \dot{\varphi}(\kappa(\az_i))\cdot \dot{\psi}(\kappa(\bz_i)).\]
Here, the  sum is over intersection points $p_i=\alpha(t_i)$, the signs $\epsilon_i=I_{p_i}(\alpha,\beta)=\pm 1$ are the local intersection numbers,  
and $\az_i$ (resp. $\bz_i$) are the classes of loops, given by a choice of 
path $\gamma_i$ from any point of  $\V$ to $p_i$, followed by the oriented loop $\alpha$ (resp. $\beta$), followed by the 
opposite path $\gamma_i^-$. 
	\end{proposition}
\end{tcolorbox}
Replacing $\gamma_i$ with a different choice $\ti\gamma_i$ changes 
both $\dot{\varphi}(\kappa(\az_i)),\ \dot{\psi}(\kappa(\bz_i))$ by 
 $\Ad_h$ for $h=\kappa(\gamma_i^{-}*\ti\gamma_i) \in G$; hence $\dot{\varphi}(\kappa(\az_i))\cdot \dot{\psi}(\kappa(\bz_i))$ 
does not depend on the choice.

\begin{proof}
By deforming the path $\beta$ we may assume, with no loss of generality, $\beta(0)=\vz\in \V$. Taking $\gamma_i$ to be the segment of $\beta$ 
from $\vz$ to $p_i$, we then have $\bz_i=\bz$, and $\az_i$ as described in Theorem  \ref{th:goldman}. 
Let $\kappa\mapsto \kappa_t$ be the flow of $X_{\varphi_\alpha}$. Using Theorem \ref{th:goldman}, we calculate 
\begin{align*} \{\varphi_\alpha,\psi_\beta\}|_\kappa&=\L_{X_{\varphi_\alpha}}\psi_\beta\\
&=\f{d}{d t}|_{t=0}\psi_\beta(\kappa_t)\\
&=\f{d}{d t}|_{t=0}\psi(\kappa_t(\bz)))\\
&=\f{d}{d t}|_{t=0}\psi\Big(\prod_{i=1}^l \exp\big(t\, \epsilon_i\dot{\phi}(\kappa(\az_i))\big)
\kappa(\bz)\Big)\\
&=\sum_i \f{d}{d t}|_{t=0} \psi\Big(\exp\big(t\epsilon_i\ \dot{\varphi}(\kappa(\az_i))\kappa(\bz)\Big)\\
&=\sum_i \epsilon_i\ \dot{\varphi}(\kappa(\az_i))\cdot \dot{\psi}(\kappa(\bz)).
\end{align*}
\end{proof}
\begin{remark}
In this discussion, we assumed that $\alpha$ is simple so that we could directly apply theorem \ref{th:goldman}. The statement holds true, however, without this assumption. 
\end{remark}

\begin{remark}
For surfaces without boundary, 	Goldman \cite{gol:inv} showed that the Poisson bracket of functions of the form $\varphi_\alpha$ may be obtained from a universal Lie algebra structure on the free abelian group spanned by (free) homotopy classes of loops in $\Sigma$. 
\end{remark}

\section{The quasi-symplectic groupoid}\label{sec:quasisymplecticgroupoid}\label{sec:quasi}
As described in Example \ref{ex:cylindergluing}, the moduli space of a cylinder, 
\[ \M_G(\Sigma_0^2,\{\vz_1,\vz_2\})\]
serves as the identity under the gluing operation. We shall see that this moduli space carries the structure of 
a \emph{quasi-symplectic groupoid}, with the moduli spaces 
for other surfaces (with one base point on each boundary) as \emph{Hamiltonian spaces}.

\subsection{Quasi-symplectic groupoids}
For background material on Lie groupoids, we refer to \cite{cra:lec,mac:gen}. In a nutshell, a Lie groupoid $\S\rra M$
consists of a manifold $\S$ of arrows, a submanifold $M$ of units, surjective submersions 
$\sz.\tz\colon \S\to M$ and a multiplication map 
\[ \Mult_\S\colon \S^{(2)}\to \S,\ (s_1,s_2)\mapsto s_1s_2\]
defined on the set $\S^{(2)}=\{(s_1,s_2)|\ \sz(s_1)=\tz(s_2)\}$ of \emph{composable arrows}.  The axioms for a 
Lie groupoid 
say that $\S$ is a category, with $M$ as its objects and $\S$ as its morphisms, such that 
$\Mult_\S$ is the composition of morphisms and such that every morphism is invertible. A \emph{(left) action} 
of a Lie groupoid $\S\rra M$ on a manifold $Q$ along a map $\Phi\colon Q\to M$ (sometimes called a momentum map) 
is given by an action map 
\begin{equation}
\label{eq:actionmap}
  \A\colon \S\ _{\sz}\!\!\times_{\Phi} Q\to Q,\ (s,q)\mapsto sq
\end{equation}

(the subscripts indicate the  fiber product), 
with $\Phi(sq)=\tz(s)$, satisfying the action properties $s_1(s_2q)=(s_1s_2)q$, with 
$mq=q$ for units $m\in M\subset \S$. For example, every groupoid has a (unique) action on its space $M$ 
of units along the identity map  $\Phi=\on{id}_M$; the action of $s\in \S$ takes $m=\sz(s)$ to $sm=\tz(s)$. This action on $M$ restricts to an action on orbits $\S m=\tz(\sz^{-1}(m))$.

A 2-form $\omega\in \Omega^2(\S)$ is called \emph{multiplicative} if it satisfies 
\[ \Mult_\S^*\omega=\pr_1^*\omega+\pr_2^*\omega\]
where $\pr_1,\pr_2\colon \S^{(2)}\to \S$ are the projections to the first and second factor. A \emph{symplectic groupoid} 
is a Lie groupoid with a multiplicative symplectic 2-form $\omega\in \Omega^2(\S)$. A \emph{Hamiltonian space} for a symplectic groupoid $(\S,\omega_\S)$ is a symplectic manifold $(Q,\omega_Q)$ with a left-action of $\S$ such that 
the action map satisfies $\A^*\omega_Q=\pr_1^*\omega+\pr_2^*\omega_Q$. As it turns out, the orbits of  
the action on units $M\subset \S$ acquire unique symplectic 2-forms for which they are Hamiltonian spaces.  

\begin{example}
	The basic example of a symplectic groupoid is the cotangent bundle of a Lie group, 
	\[ \S=T^*G\rra \g^*\] 
	with its standard symplectic form.
Using left trivialization  $T^*G\cong G\times \g^*$, with elements denoted $(c,\mu)$, the symplectic form is 
\begin{equation}\label{eq:tstarg}
 \omega=\d \l c^*\theta^L,\mu\r=-\l c^*\theta^L, \d\mu\r-\hh \l [c^*\theta^L,c^*\theta^L],\mu\r.\end{equation}
The groupoid structure is that of an action groupoid for the co-adjoint action: Its source and target map are
$\sz(c,\mu)=\mu,\ \tz(c,\mu)=(\Ad_{c^{-1}})^*\mu$, and the groupoid multiplication reads as
\[ (c_1,\mu_1)(c_2,\mu_2)=(c_1c_2,\mu_2).\]
	As shown by Mikami-Weinstein \cite{mik:mom}, a Hamiltonian space for this symplectic groupoid is exactly the same as a Hamiltonian $G$-space in the sense of symplectic geometry, with $\Phi$  as its momentum map. The symplectic structure 
	on the $G$-orbits $\O\subset \g^*\subset T^*G$ is the usual Kirillov-Kostant-Souriau 2-form.
\end{example}

Symplectic groupoids `integrate' Poisson manifolds. Their Dirac-geometric generalizations are the quasi-symplectic groupoids.  

 \begin{tcolorbox}
 	\begin{definition}\cite{bur:int,xu:mom} \label{def:quasisymplecticgroupoid}
 		A \emph{quasi-symplectic groupoid} is a Lie groupoid $\S\rra M$ with $\dim\S=2\dim M$, equipped with a 
 		2-form $\omega\in \Omega^2(\S)$ and a 3-form $\eta\in \Omega^3(M)$ such that
 		\begin{enumerate}
 			\item $\omega$ is multiplicative,
 			\item $\d\omega=\tz^*\eta-\sz^*\eta,\ \ \d\eta=0$,
 			\item $\ker(\omega)\cap \ker(T\sz)\cap \ker(T\tz)=0$. 
 		\end{enumerate}
 		A \emph{Hamiltonian space} for a quasi-symplectic groupoid is a manifold $Q$ with a 2-form $\omega_Q$, together with an $\S$-action along $\Phi\colon Q\to M$, such that 
 		\begin{enumerate}
 			\item $\A^*\omega_Q=\pr_\S^*\omega+\pr_Q^*\omega_Q$, 
 			\item $\d\omega_Q=-\Phi^*\eta$,
 			\item $\ker(\omega_Q)\cap \ker(T\Phi)=0$. 
 		\end{enumerate}
 	\end{definition}
 \end{tcolorbox}
Here $\A$ is the action map \eqref{eq:actionmap}, and $\pr_\S,\pr_Q$ are the projections from the fiber product to the two factors. 

Similar to the symplectic setting, the orbits $\O=\tz(\sz^{-1}(a))\subset M$ of a quasi-symplectic groupoid 
 $\S$  (acting on its units $M$) are examples of Hamiltonian $\S$-spaces, with momentum map the inclusion. The 2-form $\omega_\O$ is obtained by reduction: the pullback of $\omega$ to $\sz^{-1}(a)$ descends to $\omega_\O$; the result does not depend on the 
choice of $a$.

\begin{remark}
	In \cite{bur:int}, the terminology \emph{twisted pre-symplectic groupoid} was used. We prefer \emph{quasi-symplectic}, since $\omega$ does satisfy a nondegeneracy condition (property (c) above). 
	 There is also a Dirac-geometric 
	formulation of quasi-symplectic groupoids without the use of 2-forms, see e.g. \cite{lib:int}. 
\end{remark}

\subsection{The moduli space of a cylinder}\phantom{.}
Let $C$ be a compact, oriented 1-manifold (in other words, a finite collection of 
circles), with a finite subset $\V\subset C$ meeting each component of $C$. 
Thus $\M_G(C,\V)\cong G^\E$. The moduli space of the cylinder $C\times I$, with vertices $\V\times \p I$\
is a Lie groupoid 
\begin{equation} \S=\M_G(C\times I,\ \V\times \p I)\rra \M_G(C,\V)\end{equation}
(see Exercise \ref{ex:2.0}), with source and target map induced by the two boundary inclusions of 
$(C,\V)$ into $ (C\times I,\V\times \p I)$. The groupoid multiplication is pictorially described by the gluing of 
cylinders. 

The groupoid may be identified with the action groupoid: 
\[ \S\cong  G^\V\times \M_G(C,\V),\]
where the isomorphism is given by the source map together with the maps 
$\S\to G$, for $\vz\in \V$, given by the holonomy of the path from $\vz\times \{1\}$ to $\vz\times \{0\}$ along $\vz\times I$.

\begin{tcolorbox}
\begin{theorem}[Quasi-symplectic groupoids from moduli spaces]\label{th:quasisymplecticgroupoid}\phantom{.}
\begin{enumerate}
	\item The moduli space of $(C\times I,\ \V\times \p I)$, with the 2-form given by Theorem \ref{th:2form}, is a
	quasi-symplectic groupoid 	\[ \M_G(C\times I,\ \V\times \p I)\rra \M_G(C,\V).\] 
	\item If $(\Sigma,\V)$ satisfies (A1),(A2),(A3), then $\M_G(\Sigma,\V)$ is a Hamiltonian space for 
	the quasi-symplectic groupoid defined by $(\p\Sigma,\V)$. \
\end{enumerate}
\end{theorem}	
\end{tcolorbox}
\begin{proof}
\begin{enumerate}
	\item The properties $\d\omega=\tz^*\eta-\sz^*\eta$ and $\ker(\omega)\cap\ker(T\sz)\cap\ker(T\tz)=0$ 
	are part of the properties of the 2-form on moduli spaces, see Theorem \ref{th:2form}. We have to show that $\omega$ is multiplicative. Note that the description of the groupoid multiplication is parallel to the reduction operation from 
	\ref{subsec:gluingreduction}, except for the fact that the gluing circles may contain more than one vertex. 

	It suffices to consider the case that $C$ is connected. 
	Enumerate the vertices counter-clockwise: $\V=\{\vz_1,\ldots,\vz_N\}$. We shall identify the right boundary circle $C\times \{1\}$ with $C$ and denote $C\times \{0\}$ by $C'$; similarly we 
	write $\V\times\{1\}=\V$ and $\V\times\{0\}=\V'=\{\vz_1',\ldots,\vz_N'\}$. 
	The groupoid multiplication involves two copies of $\M_G(\Sigma,\V)$. For the first copy we denote by $c$ the holonomy from $\vz_1$ to $\vz_1'$ along $\{\vz_1\}\times I$, and by $a_i$ (resp. $a_i'$) the holonomies 
	of the boundary edge  from $\vz_{i+1}$ to $\vz_i$ (resp. $ \vz'_{i+1}$ to $\vz'_i$). The corresponding quantities 
	for the second copy are denoted $d$ and $b_i,b_i'$ respectively. 
	The subset $Z=\S\,_\sz\!\times_\tz\S$ is characterized by the equations $a_i=b_i'$ for $i=1,\ldots,N$. 
	As explained above, the quotient map $\pi\colon Z\to \S$ serves as the groupoid multiplication. 
	The submanifold $Z$ is the moduli space associated with 
	$\left(C\times \f{I\sqcup I}{\sim},\V\times \f{\p I\sqcup \p I}{\sim}\right)$, where $\sim$ is the 
    equivalence relation defined by the gluing of end points. The gluing diagram 
  	\begin{center}
  	\includegraphics[width=0.25\textwidth]{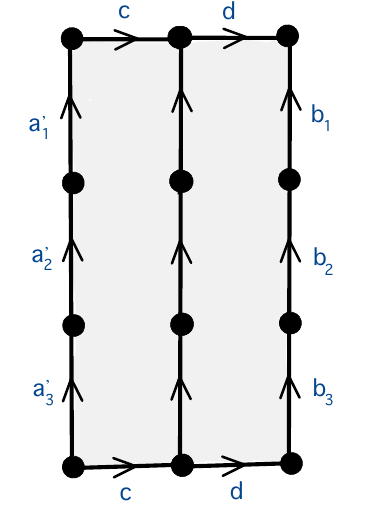}
  \end{center}
    gives 
 \[ (e,\pi^*\omega)=   (c^{-1},0)\bullet (a_N',0)\bullet\cdots (a_1',0)\bullet (c,0)\bullet (d,0)\bullet 
 (b_1^{-1},0)\bullet \cdots (b_N^{-1},0)\bullet (d^{-1},0).\]
 Using $a_i=b_i'$, the right hand side can be written as a product of 
 \begin{align*}
 (c^{-1},0)\bullet (a_N',0)\bullet\cdots (a_1',0)\bullet (c,0)\bullet  (a_1^{-1},0)\bullet \cdots (a_N^{-1},0)&=(e,\pr_1^*\omega)\\
 (b_N',0) \bullet \cdots \bullet (b_1',0)\bullet 
 (d,0)\bullet (b_1^{-1},0)\bullet \cdots (b_N^{-1},0)\bullet (d^{-1},0)&=(e,\pr_2^*\omega)
 \end{align*}
 In terms of gluing diagrams, the two products are depicted as 
 	\begin{center}
 	\includegraphics[width=0.3\textwidth]{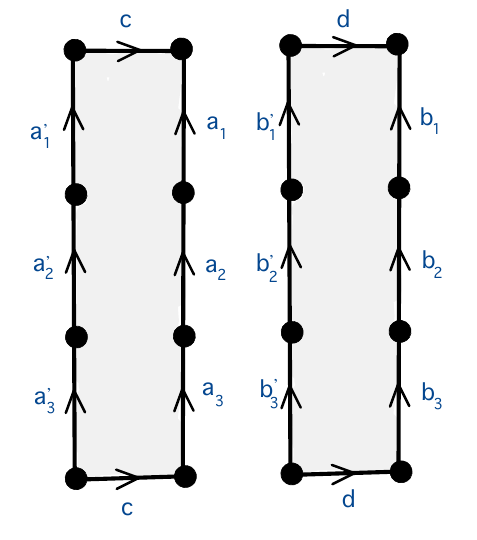}
 \end{center}
 
We hence obtain $(e,\pi^*\omega)= (e,\pr_1^*\omega)\bullet  (e,\pr_2^*\omega)
 =(e,\pr_1^*\omega+\pr_2^*\omega)$ as desired.
\item The argument is analogous to part (a), by choosing gluing a diagram for $(\Sigma,\V)$.  
\qedhere
\end{enumerate}
\end{proof}

\begin{remark}
An important notion in the theory of Lie groupoids is \emph{Morita equivalence}.
A Morita equivalence of Lie groupoids $\S_i\rra Q_i,\ i=1,2$ is given by a manifold $\E$, together with a left-action
of $\S_1$ along $\Phi_1\colon \E\to Q_1$ and a right action of $\S_1$ along $\Phi_2\colon \E\to Q_2$, such that the two actions commute, both actions are principal actions, and $\Phi_1$ is the quotient map for the 
$\S_2$-action while $\Phi_2$ is the quotient map for the $\S_1$-action. For a Morita equivalence of quasi-symplectic 
groupoids \cite{xu:mom}, the manifold $\E$ is equipped with a 2-form $\omega_\E$, in such a way that $\E$ becomes a Hamiltonian space for $\S_1\times \S_2^{\on{op}}$ (where $\on{op}$ indicates the opposite groupoid obtained by reversing arrows, with 2-form $-\omega_2$). Momentum map theories for Morita equivalent  quasi-symplectic groupoids are considered `essentially the same'. For more details and basic facts, see e.g. \cite[Appendix A]{al:coad}.

In the moduli setting, if $C$ is a compact oriented 1-manifold and $\V_1,\V_2\subset C$ are finite subsets meeting each component, then the quasi-symplectic groupoids $\S_1,\S_2$ associated to $(C,\V_1)$ and $(C,\V_2)$ are Morita equivalent, with the 
bimodule 
\[ \S_1\circlearrowright \ca{R} \circlearrowleft \S_2\]
given by $\ca{R}=\M_G(C\times I,(\V_1\times \{0\})\sqcup (\V_2\times \{1\}))$. 
\end{remark}

\subsection{Orbits}
As remarked after Definition \ref{def:quasisymplecticgroupoid}, the orbits of any quasi-symplectic groupoid $\S$ are 
Hamiltonian spaces for $\S$.  In the case of $\S=\M_G(C\times I,\V\times \p I)$, the groupoid is an action groupoid for the $G^\V$-action on $\M_G(C,\V)\cong G^\E$. For a description of the orbits of this action, we may assume that 
$C$ is connected. 

Suppose first that $\V=\{\vz\}$ is a single point. Then $G^\E=G$ with the usual conjugation action of $G^\V=G$, hence the orbits are conjugacy classes $\O=G.a\subset G$. The 2-form on $\O$, making it a quasi-Hamiltonian space 
for the quasi-symplectic groupoid of $(C,\{\vz\})$, may be computed from the formula for the 2-form on 
$\S$ given in Proposition \ref{prop:cylinder}: Pulling back to $\sz^{-1}(a)$ gives the 2-form 
 $\hh \theta^L\cdot \Ad_a \theta^L\in \Omega^2(G)$.
 That is, 
 \[ (\iota_{\sz^{-1}(a)}^*\omega)(\xi^L,\zeta^L)=-\hh (\Ad_a-\Ad_{a^{-1}})\xi\cdot \zeta.\]
It  descends to the 2-form on $\O=G/G_a$, given on generating vector fields by the same expression, 
\[ \omega_\O(\xi_\O,\zeta_\O)|_a 
=-\hh (\Ad_a-\Ad_{a^{-1}})\xi\cdot \zeta;
\]
We may regard $\omega_\O$ as a counterpart of the Kirillov-Kostant-Souriau form on coadjoint orbits $\O\subset \g^*$. It was first described in \cite{gu:gr}, see also \cite[Section 3.1]{al:mom}. Note that $\omega_\O$ is the unique 2-form on the orbit satisfying the momentum map condition \[\iota(\xi_\O)\omega_\O=-\iota_\O^*(\theta^L+\theta^R)\cdot \xi.\] 

In the general case, if $\V$ has $n\ge 1$  elements, the orbits $\O\subset \M_G(C,\V)$ of the groupoid 
consist of $n$-tuples $(a_1,\ldots,a_n)$ whose product $\prod_{i=1}^n a_i$ lies in a fixed conjugacy class. 
One may work out a formula for $\omega_\O$  by pulling back the 2-form on $\M_G(C\times I,\V\times \p I)$ 
to source fibers. 
 The result is, for 
$\xi=(\xi_1,\ldots,\xi_n),\ \zeta=(\zeta_1,\ldots,\zeta_n)\in \g^\V$, and any point $a=(a_1,\ldots,a_n)\in G^\E$,
\[  \omega_\O(\xi_\O,\zeta_\O)|_a 
=-\hh \sum_{i=1}^n (\Ad_{a_i}-\Ad_{a_i^{-1}})\xi_{i+1}\cdot \zeta_i.\]

\section{Dirac geometry}\label{sec:dirac}
Until now, we developed the theory of 2-forms on moduli spaces intrinsically, using gluing diagrams for surfaces. We did, however, postpone arguments having to do with the `minimal degeneracy' property of the 2-form from Theorem \ref{th:2form}, since these 
are better understood within the framework of Dirac geometry. We shall give these proofs now, after developing some foundational material. 

\subsection{Definitions}\label{subsec:definitions}
The notion of a \emph{Dirac structure} on a manifold was introduced by Courant and Weinstein \cite{cou:di,couwein:beyond} as a simultaneous generalization of Poisson structures and closed 2-forms. The basic idea is to describe  bivector fields and 2-forms on a manifold $Q$ in terms of their graphs. Fix a closed 3-form $\eta\in \Omega^3(Q)$, and denote 
\[\T_\eta Q=TQ\oplus T^*Q\]
with the \emph{Courant  bracket} 
(also known as \emph{Dorfman bracket}) 
on its space of sections 
\[ \Cour{X_1+\mu_1,X_2+\mu_2}=[X_1,X_2]+\L_{X_1}\mu_2-\iota_{X_2}\d\mu_1
+\iota_{X_1}\iota_{X_2}\eta\]
for vector fields $X_i$ and 1-forms $\mu_i$. The bracket satisfies a version of the Jacobi identity:  
\[ \Cour{\sigma_1,\Cour{\sigma_2,\sigma_3}}
=\Cour{\Cour{\sigma_1,\sigma_2,\sigma_3}}+\Cour{\sigma_2,\Cour{\sigma_1,\sigma_3}}.\]
It is not skew-symmetric, but its symmetric part is an exact 1-form:
\[ \Cour{\sigma_1,\sigma_2}+\Cour{\sigma_2,\sigma_1}=\d \l\sigma_1,\sigma_2\r.\]
Here  $\l\cdot,\cdot\r$ denotes the symmetric bilinear form  on $\ T_\eta Q$, given by 
\[\l X_1+\mu_1,X_2+\mu_2\r=\iota_{X_1}\mu_2+\iota_{X_2}\mu_1.\] 

A \emph{Dirac structure} on $Q$ is a subbundle $A\subset \T_\eta Q$  (with base $Q$)  such that $A$ is Lagrangian (i.e., $A=A^\perp$), and such that its space of sections is closed under the Courant bracket. For every Dirac structure, the 
Courant-Dorfman bracket restricts to a Lie bracket on $\Gamma(A)$, making $A$ into a Lie algebroid with anchor 
$\a_A\colon A\to TQ$ given by projection to the vector field part. 

For $\eta=0$, one finds that a bivector field $\pi\in \Gamma(\wedge^2 TQ)$ is a Poisson structure on $Q$ if and only if the graph of the map 
$\pi^\sharp\colon T^*Q\to TQ$ is a Dirac structure; similarly a 2-form $\omega\in \Gamma(\wedge^2 T^*Q)$ is closed if and only if the graph of $\omega^\flat\colon TQ\to T^*Q$ is a Dirac structure.

The Dirac structures relevant to quasi-Hamiltonian geometry arise from the $\eta$-twisted Courant bracket on 
$\EE=\T_\eta G$, where $G$ is a Lie group with an invariant metric $\cdot$ on its Lie algebra $\g$, and  $\eta\in \Omega^3(G)$ the Cartan 3-form \eqref{eq:eta}. It admits a trivialization \cite{al:pur}
\[ \EE\cong G\times (\ol{\g}\oplus \g)\] 
by the map $s\colon \ol{\g}\oplus \g\to \Gamma(\EE)$, 
\begin{equation}\label{eq:striv} s(\xi',\xi)=\xi^L-(\xi')^R+\hh (\theta^L\cdot\xi+\theta^R\cdot\xi')\in \Gamma(\T G).\end{equation}
One may verify that $s$ takes the Lie bracket and metric on $\ol{\g}\oplus \g$ (regarded as constant sections) to the Courant bracket and metric 
on $\Gamma(\EE)$. Hence, every Lagrangian Lie subalgebra $\mf{l}\subset (\ol{\g}\oplus \g)$ determines a Dirac structure 
$A=G\times \mf{l}$ inside $\EE$. In particular, the \emph{Cartan-Dirac structure} $A\subset \EE$ \cite{al:pur,bur:di}
is the subbundle corresponding to the diagonal, $\mf{l}= (\g)_\Delta$.

\subsection{Dirac morphisms, Hamiltonian spaces}
Let $Q_i,\ i=1,2$ be manifolds with closed 3-forms $\eta_i$. A \emph{Courant morphism} 
\begin{equation}\label{eq:courantmorphism}
\T_\omega \Phi\colon \T_{\eta_1}Q_1\da \T_{\eta_2}Q_2\end{equation}
is given by a smooth map $\Phi\colon Q_1\to Q_2$ and a 2-form 
$\omega\in \Omega^2(Q_1)$ such that $\eta_1-\Phi^*\eta_2=\d\omega$. (If the 2-form $\omega$ is zero, we use the 
notation $\T\Phi$.)  We  think of the Courant morphism 
as a  \emph{relation}, where $v_1+\mu_1\in \T_{\eta_1}Q_1|_{q_1}$ is related to 
$v_2+\mu_2\in \T_{\eta_2}Q_2|_{q_2}$ if and only if $q_2=\Phi(q_1)$ and 
\[ v_2=T\Phi|_{q_1}(v_1),\ \ \mu_1=(T\Phi|_{q_1} )^*(\mu_2)+\iota_{v_1}\omega.\]
 We shall write 
\[ v_1+\mu_1\sim_{\T_\omega\Phi} v_2+\mu_2\]
for related elements. 
Given another Courant morphism 
$\T_{\omega'}\Phi'\colon  \T_{\eta_2}Q_2\da \T_{\eta_3}Q_3$, the composition of relations is given by 
\[ \T_{\omega'} \Phi'\circ \T_\omega\Phi=\T_{\omega+ 
\Phi^*\omega'}(\Phi'\circ \Phi).\]
Given Dirac structures $A_i\subset \T_{\eta_i} Q_i$, we say that \eqref{eq:courantmorphism} defines a (strong) Dirac morphism, also known as  
\emph{morphism of Manin pairs} \cite{bur:cou},
\begin{equation}\label{eq:diracmorphism}
\T_\omega \Phi\colon (\T_{\eta_1}Q_1,A_1)\da (\T_{\eta_2}Q_2,A_2),\end{equation}
if for all $q_2=\Phi(q_1)$, every element of $(A_2)_{q_2}$ is $\T_\omega \Phi$-related to a \emph{unique} element of 
$(A_1)_{q_1}$. By definition, a Dirac morphism defines a bundle map $\Phi^*A_2\to A_1$, or equivalently a \emph{comorphism of vector bundles} $A_1\da A_2$. In fact, it is a \emph{comorphism of Lie algebroids}: the
pullback map on sections $\Gamma(A_2)\to \Gamma(A_1)$ preserves Lie brackets. 
\medskip

As an important special case, we define 
 
\begin{tcolorbox}
\begin{definition}\cite{bur:cou} \label{def:ham}
	A \emph{Hamiltonian space} for the Dirac structure $A\subset \T_\eta Q$ is a 
	manifold $M$ together with a Dirac morphism
\begin{equation}\label{eq:ham}
\T_\omega\Phi\colon (\T M,TM)\da (\T_\eta Q,A).
\end{equation}
 
	\end{definition}
\end{tcolorbox}

We refer to $\Phi$ as the \emph{momentum map} of the Hamiltonian space. The Lie algebra morphism 
\[ \Gamma(A)\to \Gamma(TM),\ \sigma\mapsto \sigma_M\] 
defines an \emph{action of the Lie algebroid $A$ on $M$} along $\Phi$.

Conversely, given a manifold $M$ with a 2-form $\omega$ and a Lie algebroid action of  $A\subset \T_\eta Q$  along a map $\Phi$, 
the pair $(\Phi,\omega)$ defines a Hamiltonian space  if and only if 
the following three conditions are satisfied: 
\medskip
	\begin{enumerate}
	\item\label{it:d1} $\d\omega=-\Phi^*\eta$
	\item\label{it:d2} $\iota(\sigma_M)\omega=-\Phi^* \alpha_A(\sigma)$ for all $\sigma\in \Gamma(A)$, 
	\item\label{it:d3} $\ker(\omega)\cap \ker(T\Phi)=0$. 
\end{enumerate}
\medskip

Here $\alpha_A(\sigma)\in \Omega^1(Q)$ is the 1-form component of $\sigma\in \Gamma(A)$. 
 By \eqref{it:d2}, the kernel of $\omega$ at $m$ contains all $\sigma|_M|_m$ such that $\alpha_A(\sigma)|_m=0$. 
 In fact, this is the entire kernel:
 \begin{tcolorbox}
 	\begin{proposition}\label{prop:kernel}
 		The condition $\ker(\omega)\cap \ker(T\Phi)=0$ is equivalent to the following explicit description of the kernel,
 		\[ \ker(\omega)|_m=\{\sigma_M|_m\colon \alpha_A(\sigma)|_{\Phi(m)}=0\},\]
 		for all $m\in M$. 
 	\end{proposition}
 \end{tcolorbox}
 
\begin{proof}
Let $v\in \ker(\omega)|_m$, with image $w=T\Phi(v)\in TQ|_{\Phi(m)}$. For all $\tau\in \Gamma(A)$ we have 
\[ 0=\omega(v,\tau_M|_m)=\iota(v) \Phi^* \alpha_A(\tau)|_m=
\iota(w) \alpha_A(\tau)|_{\Phi(m)}
=\l (w,0),\tau|_{\Phi(m)}\r.\]
This means that $(w,0)\in A_{\Phi(m)}^\perp=A_{\Phi(m)}$. Hence, there exists $\sigma\in \Gamma(A)$ such that 
\[ \a_A(\sigma)|_{\Phi(m)}=w,\ \ \ \alpha_A(\sigma)|_{\Phi(m)}=0.\] 
The difference 
$v-\sigma_M|_m$ lies in $\ker(\omega)\cap \ker(T\Phi)$, and hence is zero by \eqref{it:d3}. 
\end{proof}

For an action of a Lie algebroid  $A$ on a manifold $M$, the set of vector fields $\{\sigma_M|\ \sigma\in \Gamma(A)\}$ 
are a locally finitely generated Lie subalgebra of $\mf{X}(M)$.  
They define a singular foliation on $M$ (in the sense of Androulidakis-Skandalis \cite{and:hol}), with leaves the $A$-orbits. For $m\in M$, let $S_m=\{\sigma_M|_m|\ \sigma\in \Gamma(A)\}\subset TM|_m$ be the tangent space to the orbit, and $\k_m$ the \emph{stabilizer Lie algebra},  given by the exact sequence
\[ 0\lra \k_m\lra A_{\Phi(m)}\lra S_m\lra 0.\]
The Lie bracket on $\k_m$ is inherited from the bracket on sections of $A$.

\begin{tcolorbox}
	\begin{proposition}\label{prop:range}
Suppose  $M$ is a Hamiltonian $A$-space defined by a Dirac morphism \eqref{eq:ham}. 
For all $m\in M$ we have 
\[ \on{ann}(\on{ran}(T\Phi|_m))= \k_m.\]
where $\k_m\subset A_{\Phi(m)}\cap T^*Q|_{\Phi(m)}$ is regarded as a subspace of $T^*Q|_{\Phi(m)}$. 
Similarly, the map $\omega^\flat\colon TM\to T^*M$ restricts to 
an isomorphism 
\[ \ker(T\Phi|_m)\cong \on{ann}(S_m).\]
	\end{proposition}
\end{tcolorbox}
\begin{proof}
By definition of a Dirac morphism, $A_{\Phi(m)}$ consists of all elements of $w+\nu\in \T_\eta Q|_{\Phi(m)}$ that are $\T_\omega\Phi$-related to \emph{some} element of $TM|_m$. The subspace $\k_m\subset A_{\Phi(m)}$ consists of elements 
that are related to $0\in TM|_m$. This means that $w=0$ and $(T\Phi|_m)^*\nu=0$, i.e., 
$\nu\in \on{ann}(\on{ran}(T\Phi|_m))$. 

For the second part, note that since $\ker(\omega)\cap \ker(T\Phi)=0$, the map $\omega^\flat$ restricts to an injection on $\ker(T\Phi|_m)$. Since 
\[ \omega(\sigma_M|_m,v)=\iota(v)\iota(\sigma_M|_m)\omega=-\iota(v)\Phi^*\alpha_A(\sigma)|_m,\]
the image $\omega^\flat(\ker(T\Phi|_m))$ is contained in $\on{ann}(S_m)$. By dimension count (using the first part), this is an equality. 
\end{proof}

\begin{remark}
This proposition generalizes a well-known fact for Hamiltonian group actions in symplectic geometry (see, e.g., \cite{gu:sy}): If $(M,\omega)$ is a symplectic manifold with a Hamiltonian action of a Lie group $G$, with momentum map $\Phi\colon M\to \g^*$, then 
\[ \ker(T\Phi|_m)=T_m(G.m)^\omega,\ \ \on{ran}(T\Phi|_m)=\on{ann}(\g_m),\]
where $\g_m\subset \g$ are the stabilizer algebras.   	
\end{remark}

Hamiltonian spaces for the Cartan-Dirac structure (Section \ref{subsec:definitions}) are the quasi-Hamiltonian spaces from \cite{al:mom}. More precisely, one obtains an equivalence with quasi-Hamiltonian \emph{$\g$-spaces}. The integration to a $G$-action may be put in by hand, or more conceptually by the notion of Hamiltonian space for 
 the quasi-symplectic groupoid integrating the Cartan-Dirac structure \cite{bur:int,xu:mom}. The moduli spaces $\M_G(\Sigma,\V)$ are Hamiltonian  space for a Dirac structure 
on $G^\E$; this will be explained below. 

\subsection{Admissible functions}
The following result is due to Bursztyn, Iglesias-Ponte, and \v{S}evera \cite[Section 3.4]{bur:cou}; in the case of quasi-Hamiltonian $G$-spaces it was proved in \cite{al:mom}.
\begin{tcolorbox}
	\begin{proposition}\label{prop:admissible}  \cite{bur:cou}
		Let $M$ be a Hamiltonian $A$-space defined by a Dirac morphism \eqref{eq:ham}. 
		Suppose $f\in C^\infty(M)$ is invariant under the $A$-action on $M$, with momentum map $\Phi$. Then there is a unique vector field 
		$X_f$ such that 
		\[ \iota(X_f)\omega=-\d f,\ \ \ X_f\sim_\Phi 0.\]
		This vector field satisfies 
		\[ \L_{X_f}\omega=0.\]
	\end{proposition}
\end{tcolorbox}
\begin{proof}
Invariance of $f$ means that $\d f|_m\in \on{ann}(S_m)$ for all $m$.
By Property \eqref{it:c} of a Hamiltonian $A$-space, the vector bundle map 
\[ TM\to \Phi^*TQ \oplus T^*M ,\ v\mapsto (T\Phi(v),-\iota_v\omega)\] 
is injective.  Proposition \ref{prop:range} shows that the elements $(0,\d f|_m)$ lie in its range. 
We hence obtain $X_f$ as the pre-image of $(0,\d f)$ under that map. 
The second claim follows by Cartan's identity: 
\[ \L_{X_f}\omega=\iota_{X_f}\d\omega+\d\iota_{X_f}\omega=
-\iota_{X_f}\Phi^*\eta-\d\d f=0.\]
Here we used $X_f\sim_\Phi 0$.
\end{proof}

In the terminology of Courant \cite{cou:di}, the function $f$ is \emph{admissible}. Note that the conditions on $X_f$ mean that 
\begin{equation}\label{eq:hamvf1} X_f+\d f\sim_{\T_\omega \Phi} 0.\end{equation}


%
\begin{tcolorbox}
	\begin{proposition}
The formula 		\begin{equation}\label{eq:poissonbracket}\{f,g\}=\L_{X_f} g\end{equation}
defines a Poisson bracket on the  space of $A$-invariant functions, 
with 
\begin{equation}\label{eq:lie} X_{\{f,g\}}=[X_f,X_g].\end{equation}
	\end{proposition}
\end{tcolorbox}
\begin{proof}
Given $A$-invariant functions $f,g$, the properties $X_f\sim_\Phi 0,\ X_g\sim_\Phi 0$ imply 
$[X_f,X_g]\sim_\Phi 0$. Furthermore, 
\[ \iota( [X_f,X_g])\omega=\L_{X_f} \iota(X_g)\omega-\iota(X_g)\L(X_f)\omega=-\L_{X_f}d g=-\d (\L_{X_f}g)=-\d \{f,g\}.\]
This proves \eqref{eq:lie}. Skew symmetry of the bracket \eqref{eq:poissonbracket} follows from 
$\{f,g\}=-\iota(X_f)\d g=\iota(X_f)\iota(X_g)\omega$, and 
the Jacobi identity  follows by applying 
\eqref{eq:lie} to a third $A$-invariant function. 
\end{proof}

In general, the $A$-action on $M$ may not be well-behaved, and hence the space of global $A$-invariant functions may be rather 
small. However, consider the open subset 
\[ M_{\on{reg}}=\{m\in M|\ \k_m=0\}\] 
on which the $A$-action is regular, i.e, where all leaves have maximal dimension $\on{rank}(A)$.  On this set, every point admits an open neighborhood $U$ on which the foliation by $A$-orbits is fibrating. 
Hence, we can consider  Hamiltonian vector fields $X_f$ for \emph{local} $A$-invariant functions  $f\in C^\infty(U)$.  
By Proposition \ref{prop:range}, the span of local Hamiltonian vector fields is exactly $\ker(T\Phi)\subset TM$. 
If the $A$-action on $M_{\on{reg}}$ is \emph{globally} fibrating, the Poisson structure on $A$-invariant functions descends to a Poisson structure on the orbit space $M_{\on{reg}}/\sim$ with symplectic leaves given by the 
images, under the quotient map, of intersections of $A$-orbits with the fibers of $\Phi$. 

\subsection{Quasi-Poisson structures}
Let $M$ be a Hamiltonian $A$-space for a Dirac structure $A\subset \T_\eta Q$. 
 The Poisson bracket on $A$-invariant functions may be extended to a \emph{quasi-Poisson structure} on \emph{all} functions, as follows. Let $B\subset \T_\eta Q$ be a Lagrangian 
subbundle (not necessarily a Dirac structure) transverse to $A$. This defines a Lagrangian splitting
\[ \T_\eta Q=A\oplus B.\] 
The preimage  $(\T_\omega\Phi)^{-1}B\subset \T M$ consists of all elements $v+\mu$ such that $v+\mu\sim_{\T_\omega\Phi} y$ for some $y\in B$. It is a Lagrangian subbundle, which is transverse to $TM$. (Indeed, $v\in TM$ with $v\sim_{\T_\omega\Phi} y$ would mean $y\in A$.) See, e.g., \cite[Section 1.8]{al:pur}. 
It is hence the graph of a bivector field $\pi\in \Gamma(\wedge^2(TM))$. If $f$ is a (local) $A$-invariant function, defining a Hamiltonian vector field $X_f$, then $X_f+\d f\in (\T_\omega\Phi)^{-1}B$ by \eqref{eq:hamvf1}. Consequently, 
$X_f=\pi^\sharp(\d f)$. It follows that $\{f,g\}=\pi(\d f,\d g)$ for all (local) $A$-invariant functions. The construction of 
$\pi$ depends on the choice of $B$, but in practice it often happens that there is a canonical choice. 
We also remark that if $B$ is integrable, i.e., a Dirac structure, then its pre-image under $\T_\omega\Phi$ is again integrable, and hence is the graph of a \emph{Poisson} structure.

\subsection{Cross sections}
Suppose $A\subset \T_\eta Q$ is a Dirac structure. Given a smooth map $f\colon Q'\hra Q$, with the property that $f$ is transverse to 
the anchor map $\a_A\colon A\to TQ$, one obtains a \emph{pullback Dirac structure}
\[ f^!A=(\T f)^{-1}(A)\subset \T_{\eta'} Q'\]
with respect to $\eta'=f^*\eta_Q$. The bundle $f^!A$ consists of all 
$v'+\mu'\in \T Q'$ for which there exists $v+\mu\in A$ with $v=f_*v'$ and $\mu'=f^*\mu$. 
We will mainly need this construction for the case of an embedding as a submanifold. Hamiltonian spaces for 
$A\subset \T_\eta Q$ give rise to Hamiltonian spaces for the pullback. The following fact 
is a Dirac-geometric analogue of the Symplectic Cross Section Theorem of Guillemin-Sternberg \cite{gu:sy}. 

\begin{tcolorbox}
	\begin{proposition}[Cross sections]\label{prop:cross}
		Let $(M,\omega,\Phi)$ be a Hamiltonian space for the Dirac structure $A\subset \T_\eta Q$, and 
		\[ \iota\colon Q'\hra  Q\] 
		a submanifold transverse to the anchor $\a_A$ (and hence also to the map $\Phi$). Let 
		$M'=\Phi^{-1}(Q')$, and let  $\omega'$ be the pullback of $\omega$ and $\Phi'$ the restriction of $\Phi$. 
		Then $(M',\omega',\Phi')$ is a Hamiltonian space for the Dirac structure $A'=\iota^!A\subset \T_{\eta'} Q'$.    	
	\end{proposition}
\end{tcolorbox}
More generally, there is a Cross Section Theorem for abritrary Dirac morphisms. For the precise 
statement (and its proof) see \cite[Theorem B.7]{me:moma}.   See \cite{al:qu,balibanu2023reduction,cro:log} for related versions.

\subsection{Dirac structures associated with closed polygons} 
Let $C$ be a compact oriented 1-manifold, with vertices $\V$ and oriented edges $\E$.  (In our application to moduli spaces for surfaces, this will be $C=\p\Sigma$.) Thus $\M_G(C,\V)=G^\E$ with the natural action of $G^\V$. 
 On $G^\E$, we have the closed 3-form $\eta^\E=\sum_{\ez\in\E} g_{\ez}^*\eta$, and the corresponding standard 
Courant algebroid $\T_{\eta^\E}\M_G(C,\V)=(\T_\eta G)^\E$. 
Equation  \eqref{eq:striv} defines a trivialization 
\[ (\T_\eta G)^\E \cong 
G^\E\times (\ol{\g}\oplus \g)^\E.
\]
Every Lagrangian Lie subalgebra $\mf{l}=(\ol{\g}\oplus \g)^\E$ defines a Dirac structure $G^\E\times \mf{l}$. We will take 
$\mf{l}$ to be $\g^\V$, regarded as a Lie subalgebra by the inclusion 
\[ \g^\V\to (\ol{\g}\oplus \g)^\E\]
taking the function $\vz\mapsto \xi_\vz$ to the function  $\ez\mapsto (\xi_{\tz(\ez)},\xi_{\sz(\ez)})$. 
Thus, $A=G^\E\times \g^\V$ is spanned by sections
\begin{equation}\label{eq:sez}\sigma(\xi)=\sum_{\ez} s^{\ez}({\xi}_{\tz(\ez)},{\xi}_{\sz(\ez)}),\end{equation}
where the superscript indicates that we put $s(\cdot,\cdot)$ on the copy of $G$ labeled by the edge $\ez$. 
%
The properties \eqref{it:a}, \eqref{it:b}, \eqref{it:c} of Theorem \ref{th:2form} may be rephrased as follows: 
\begin{tcolorbox}
\begin{proposition}\label{prop:dirgeom}
For a pair $(\Sigma,\V)$ satisfying (A1),(A2),(A3), the boundary holomomies define a  $G^\V$-equivariant Dirac morphism 
\begin{equation}\label{eq:diracformulation}
	 \T_\omega\Phi\colon (\T\M_G(\Sigma,\V),T \M_G(\Sigma,\V))\da ((\T_\eta G)^\E,A).\end{equation}
That is, the moduli space is a Hamiltonian space for this Dirac structure. 
\end{proposition}	
\end{tcolorbox}
Recall that we had postponed the proof of property \eqref{it:c} of Theorem \ref{th:2form}, except for the case that  
$\Sigma$ is a disjoint union of disks. We shall now use the Dirac-geometric formulation to complete the proof. 

\subsection{Proof of minimal degeneracy}\label{subsec:mindeg}
Suppose $(\Sigma,\V)$ satisfies (A1),(A2),(A3). In Subsection \ref{subsec:gluingpattern}, we constructed $\omega\in \Omega^2(\M_G(\Sigma,\V))$ as the pullback $\omega=\iota^*\wh{\omega}$ of the 2-form on 
$\M_G(\wh\Sigma,\wh\V)$, where $\wh{\Sigma}$ is a disjoint union of disks (polgons) obtained from $\Sigma$ by cutting, and 
\[ \iota\colon  \M_G(\Sigma,\V)\to \M_G(\wh\Sigma,\wh\V)\]
is the inclusion. This inclusion may be regarded as a cross section, as follows: Let $\p\wh{\Sigma}$ be the boundary of the cut surface, and $\p\wh{\Sigma}/\sim$ its image under the quotient map $\wh{\Sigma}\to \Sigma$. Thus,  $\p\wh{\Sigma}/\!\!\sim$ is an embedded graph in $\Sigma$, consisting of the boundary edges of $\Sigma$ together with the paths 
defining the cutting. The corresponding moduli space will be denoted by $Q$. 
The maps $\p\wh{\Sigma}\to \p\wh{\Sigma}/\!\!\sim\,\, \leftarrow \p\Sigma$ induce maps
\[ \M_G(\p\wh{\Sigma},\wh{\V})=G^{\wh{\E}}\stackrel{j}{\longleftarrow} 
\M_G(\p\wh{\Sigma}/\!\!\sim\,\,  ,\V)=Q\stackrel{\pi}{\longrightarrow} 
\M_G(\Sigma,\V)=G^\E.\]
The submanifold $Q\subset G^{\wh{\E}}$ consists of maps $g\colon \wh{\E}\to G$ such that $g_{\ez_1}=g_{\ez_2}^{-1}$ for every pair of `glued' edges; the quotient map $\pi$ omits components corresponding to glued edges. Observe that 
$\M_G(\Sigma,\V)=j^{-1}(\M_G(\wh\Sigma,\wh\V))$. 

Let $\wh{A}\subset (\T_\eta G)^{\wh{\E}}$ be the counterpart of $A$ for the cut surface. Since the 
2-form $\wh{\omega}$ for the cut surface is known to satisfy the properties of Theorem \ref{th:2form}, it defines a 
Dirac morphism $\T_{\wh{\omega}}\wh{\Phi}$ as in \eqref{eq:diracformulation}.

The anchor of $\wh{A}$ is transverse to the submanifold $Q$, due to the fact that $G^{\wh\V}.Q=G^{\wh\E}$. Hence, 
Proposition \ref{prop:cross}  gives a Dirac morphism
\begin{equation}\label{eq:firststep}
\T_\omega(\wh{\Phi}|_{j^{-1}(Q)})
\colon  (\T\M_G(\Sigma,\V),T \M_G(\Sigma,\V))\da (\T_{\eta_Q}Q,j^!\wh{A}),
\end{equation}
where $\eta_Q=j^*(\eta^{\wh\E})$. The submanifold $Q$ is a direct product 
\[ Q=G^\E\times \prod_{\{\ez',\ez\}}Q_{\ez',\ez}\] 
where the second product is over pairs of glued edges, and $Q_{\ez',\ez}\subset G\times G$ is the anti-diagonal $\Mult_G^{-1}(e)$. 
The map $\pi$ is projection to $G^\E$. Since the pullback of $\pr_1^*\eta+\pr_2^*\eta\in \Omega^3(G\times G)$ to the anti-diagonal vanishes, we have that $j^*\eta^{\wh\E}=\pi^*\eta^\E$. 
\begin{tcolorbox}
	\begin{lemma}\label{lem:remains}
The projection map $\pi$ defines a Dirac morphism 
\begin{equation}\label{eq:secondstep}
 \T\pi\colon 		 (\T_{\eta_Q}Q,j^!\wh{A})\da ((\T_\eta G)^\E,A).
\end{equation}
	\end{lemma}
\end{tcolorbox}
Writing $\Phi=\pi\circ \wh{\Phi}|_{j^{-1}(Q)}$, Proposition \ref{prop:dirgeom} then follows by composition of the Dirac morphisms 
\eqref{eq:firststep} and \eqref{eq:secondstep}. 

\begin{proof}
The composition of maps $\V\to \g$ with the quotient map $\wh{\V}\to \V$ determines an inclusion 
\begin{equation}\label{eq:inclusion}
\g^\V\to \g^{\wh{\V}},\ \ \xi\mapsto \wh{\xi}
\end{equation}
with the properties 
\begin{equation}\label{eq:property1}
\wh{\xi}_{\sz(\ez)}=\xi_{\sz(\ez)},\ 	\wh{\xi}_{\tz(\ez)}=\xi_{\tz(\ez)}.
\end{equation} 
for all non-glued edges $\ez\in \E\subset \wh{\E}$ and 
\begin{equation}\label{eq:property2} \wh\xi_{\tz(\ez')}=\wh\xi_{\sz(\ez)},\ \  \wh\xi_{\tz(\ez)}=\wh\xi_{\sz(\ez')}\end{equation}
for every pair $\{\ez',\ez\}\subset \wh{\E}-\E$ of glued edges. Conversely, using a dimension count, we see that 
a given map $\wh{\xi}$ lies in the image of this map 
if and only if it satisfies \eqref{eq:property2}; the pre-image $\xi$ is then determined from \eqref{eq:property1}. 

The subbundle $A\subset (\T_\eta G)^\E$ is spanned by sections of the form \eqref{eq:sez} with $\xi\in \g^\V$. The image $\wh{\xi}$ under the inclusion \eqref{eq:inclusion} defines a section 
 \begin{equation} \label{eq:lagrangiansection}
\sigma(\wh{\xi})=\sum_{\ez\in\wh{\E}} s^{\ez}(\wh{\xi}_{\tz(\ez)},\wh{\xi}_{\sz(\ez)}),\end{equation}
of $\wh{A}\subset (\T_\eta G)^{\wh\E}$. The condition \eqref{eq:property2} guarantees that the image of this section under the anchor is tangent to $Q$. We claim that 
\begin{equation}\label{eq:sect1}
\sigma(\wh{\xi})\sim_{\T\pi}\sigma(\xi).
\end{equation}

To see this, note that for every pair $\{\ez',\ez\}$ of glued edges, the sections of $(\T_\eta G)^{\ez',\ez}=
\T_\eta G\times \T_\eta G$ of the form 
\begin{equation}\label{eq:sect} s^{\ez}(\xi_1,\xi_2)+s^{\ez'}(\xi_2,\xi_1)\end{equation}
with $\xi_1,\xi_2\in \g$
span a Lagrangian subbundle whose image under the anchor is tangent to $Q_{\ez',\ez}$. The restriction to $Q_{\ez',\ez}$ therefore descends to a Lagrangian subbundle of 
\[ a^{-1}(TQ_{\ez',\ez})/a^{-1}(TQ_{\ez',\ez})^\perp=\T Q_{\ez',\ez}.\] 
In fact, since the differential form part of \eqref{eq:sect} vanishes when pulled back to $Q_{\ez',\ez}$, we see that it simply descends to $TQ_{\ez',\ez}$ -- the Lagrangian subbundle defining the Courant morphism $\T\pi_{\ez',\ez}\colon \T Q_{\ez',\ez}\to 0$. 
In other words, \eqref{eq:sect} is $\T\pi_{\ez',\ez}$-related to $0$.  The morphism $T\pi$ is the product of the morphism $\T\pi_{\ez',\ez}$ for glued edges and the identity morphisms of $(\T_\eta G)^\ez$ for the  edges in $\E\subset \wh{\E}$.  
We hence obtain \eqref{eq:sect1}.
Conversely, suppose $\wh{\xi}$ satisfies 
\begin{equation}\label{eq:finally}
 \sigma(\wh{\xi})|_{\wh{g}}\sim_{\T\pi}0
 \end{equation}
for some $\wh{g}\in Q$. In particular, the image of $\sigma(\wh{\xi})$ under the anchor is tangent to $Q$ at $\wh{g}$. 
This implies \eqref{eq:property2}. As observed above, this implies that $\wh{\xi}$ is the image of a unique $\xi\in\g^\V$
under the inclusion \eqref{eq:inclusion} where $\xi$ is determined by \eqref{eq:property1}. But \eqref{eq:finally} shows that $\xi=0$, and hence $\wh{\xi}=0$. 
\end{proof}

\subsection{Lagrangian boundary conditions}\label{subsec:lagrangian}
By similar arguments, we can complete the proof that for a 
collection of Lagrangian Lie subgroups $H_\ez$ attached to the edges, 
satisfying the conditions of Proposition \ref{prop:severa2}. As explained in the proof of Proposition \ref{prop:severa2}, 
the inclusion map $j\colon Q=\prod_\ez H_\ez\to G^\E$  is transverse to the $G^\V$-orbits in $G^\E$. Hence, the Cross Section Theorem \ref{prop:cross} applies, and gives a Dirac morphism
\[ \T_{\iota^*\omega}(\Phi|_{\Phinv(Q)})\colon\ (\T\Phi^{-1}(Q),T\Phi^{-1}(Q))\da (\T Q,j^!(G^\E\times \g^\V)).\]
Let $\pi\colon Q\to \pt$ be the constant map to a point. We claim that $\T\pi$ defines a Dirac morphism 
\[ \T\pi\colon  (\T Q,j^!(G^\E\times \g^\V))\da (0,0).\]
Suppose $\xi\in \g^\V$ is such that the image of $\delta(\xi)$ under the anchor is tangent to $Q$, at some given 
$h\in Q$ (with components $h_\ez\in H_\ez$).  
This means that for all $\ez$, the vector field $\xi_{\sz(\ez)}^L-\xi_{\tz(\ez)}^R$ on $G$ is tangent to $H_\ez$. That is, 
\[ \xi_{\sz(\ez)}-\Ad_{h_\ez}\xi_{\tz(\ez)}\in\h_\ez.\]
The 
resulting section $j^!\delta(\xi)$ of $j^!(G^\E\times \g^\V)$ satisfies $j^!\delta(\xi)\sim_{T\pi}0$ if and only if 
the 1-form part of $\delta(\xi)|_h$ pulls back to $0$ on $Q$. Equivalently, each 
\[ (\theta^L\cdot \xi_{\sz(\ez)}+\theta^R\cdot \xi_{\tz(\ez)})|_h\in T^*G|_{h_\ez}\]
pulls back to $0$ on $H_\ez$. 
Since $H_\ez$ is Lagrangian, this gives the condition 
 \[ \xi_{\sz(\ez)}+\Ad_{h_\ez}\xi_{\tz(\ez)}\in\h_\ez.\]
Taken together, we arrive at the condition $\xi_{\sz(\ez)},\ \xi_{\tz(\ez)}\in \h_\ez$ for all $\ez$. But since every vertex 
$\vz$ arises as $\sz(\ez)=\vz=\tz(\ez')$ for adjacent edges $\ez,\ez'$, and since $\h_\ez\cap \h_{\ez'}=0$ by assumption, we conclude $\xi=0$. This shows that $\T\pi$ is a Dirac morphism as desired. Letting $p=\pi\circ \Phi|_{\Phinv(Q)}\to \pt$ be the projection $\Phinv(Q)\to \pt$, it follows that $\iota^*\omega$ defines a Dirac morphism 
\[ \T_{\iota^*\omega}p\colon  (\T\Phi^{-1}(Q),T\Phi^{-1}(Q))\da (0,0).\]
Equivalently, $\iota^*\omega$ is symplectic.


\def\cprime{$'$} \def\polhk#1{\setbox0=\hbox{#1}{\ooalign{\hidewidth
			\lower1.5ex\hbox{`}\hidewidth\crcr\unhbox0}}} \def\cprime{$'$}
\def\cprime{$'$} \def\cprime{$'$} \def\cprime{$'$} \def\cprime{$'$}
\def\polhk#1{\setbox0=\hbox{#1}{\ooalign{\hidewidth
			\lower1.5ex\hbox{`}\hidewidth\crcr\unhbox0}}} \def\cprime{$'$}
\def\cprime{$'$} \def\cprime{$'$} \def\cprime{$'$} \def\cprime{$'$}
\providecommand{\bysame}{\leavevmode\hbox to3em{\hrulefill}\thinspace}
\providecommand{\MR}{\relax\ifhmode\unskip\space\fi MR }
\providecommand{\MRhref}[2]{%
	\href{http://www.ams.org/mathscinet-getitem?mr=#1}{#2}
}
\providecommand{\href}[2]{#2}

\end{document}